\documentclass[10pt]{elsarticle}

\usepackage{amsmath}
\usepackage{cases}
\usepackage{amsfonts}
\usepackage{amssymb}
\usepackage{color}
\usepackage{graphicx}
\usepackage{epstopdf}
\usepackage{subfigure}
\usepackage{color}
\usepackage{diagbox}
\usepackage{float}
\usepackage{algorithm}
\usepackage{algorithmic}
\usepackage{multirow}
\usepackage{amsthm}
\usepackage{geometry}%页面设置
\usepackage{graphics}
\usepackage{mathrsfs}
\usepackage{setspace}
\usepackage{caption}%图形编号

\allowdisplaybreaks[4]%%允许公式断行
\usepackage[hidelinks]{hyperref} %跳转
\providecommand{\U}[1]{\protect\rule{.1in}{.1in}}

\newtheorem{theorem}{Theorem}[section]

\newtheorem{assumption}[theorem]{Assumption}
\newtheorem{example}[theorem]{Example}

\newtheorem{lemma}[theorem]{Lemma}

\newtheorem{remark}[theorem]{Remark}
\numberwithin{equation}{section}

\begin{document}
\begin{frontmatter}

\title{Uniform Convergence Rate of the Nonparametric Estimator for Integrated Diffusion Processes}

\author {Shaolin Ji}
\ead{jsl@sdu.edu.cn}
\author{Linlin Zhu\corref{cor1}}
\ead{201611343@mail.sdu.edu.cn}
\cortext[cor1]{Corresponding author}

\address{Institute for Financial Studies, Shandong University, Jinan  250100, China}

\begin{abstract}
The nonparametric estimation of integrated diffusion processes has been extensively studied, with most existing research focusing on pointwise convergence. This paper is the first to establish uniform convergence rates for the Nadaraya-Watson estimators of their coefficients. We derive these rates over unbounded support under the assumptions of a vanishing observation interval and a long time horizon. Our findings serve as essential tools for specification testing and semiparametric inference in various diffusion models and time series, facilitating applications in finance, geology, and physics through nonparametric estimation methods.
\end{abstract}


\end{frontmatter}
\section{Introduction}
Statistical inference of diffusion process has become a hot topic in finance, economics, geology, and physics, including parameter and nonparametric estimation of the process based on discrete observations and other type of observations, for example, \cite{florens1993estimating, bibby1995martingale, jiang1997nonparametric, bandi2003fully, fan2003, tang2009parameter, ait2016bandwidth}. In several cases, the stochastic process can be seen as an integrated diffusion process, capturing the cumulative effect of past perturbations. For instance, such processes are used to model ice-core data on oxygen isotopes, which serve as proxies for paleo-temperatures \cite{ditlevsen2002fast}. Moreover, the integrated diffusion process generalizes the unit root process to a continuous-time setting  \cite{nicolau2007nonparametric},
and plays a crucial role in the study of realized volatility in finance  \cite{barndorff2002econometric, andersen2009realized}.

Parameter estimation for integrated diffusion processes has been investigated in \cite{gloter2001parameter, ditlevsen2004inference, gloter2006parameter}.
In terms of nonparametric estimation, the Nadaraya-Watson (N-W) estimators, local linear estimator, and re-weighted N-W estimator for the infinitesimal coefficients of the process have been established in \cite{nicolau2007nonparametric, wang2011locallinear, song2013reweighted}, among others. These studies primarily focused on pointwise properties, establishing results on consistency and asymptotic normality.  However, the uniform convergence for these estimators remain largely unexplored.

Uniform convergence properties for nonparametric estimators of the related functionals for the diffusion process are crucial for speciﬁcation testing of the process, see \cite{hong2005nonparametric, li2007testing, ait2009nonparametric, kristensen2011semi}. Additionally, in semiparametric estimation, where either the drift or the diffusion coefficient is parametrically specified while the other remains unrestricted, the uniform convergence rate of a nonparametric estimator for the unspecified term can be applied to get the asymptotic properties of the  semiparametric estimator of the unknown parameter, see \cite{bu2023uniform}.
In the discrete-time setting, uniform consistency of kernel-type estimators
have been studied by several researchers, including  \cite{bierens1983uniform, andrews1995nonparametric, liebscher1996strong, hansen2008uniform, kristensen2009uniform, li2012local, chan2014uniform, lidegui2021nonparametric}.
However, results concerning uniform consistence rates for estimators in diffusion models are relatively sparse.
\cite{koo2012estimation} investigated the uniform convergence of nonparametric estimators for both the drift and diffusion coefficients in locally semiparametric stationary diffusion process.
\cite{kanaya2017uniform} derived the uniform convergence rates for the modified N-W estimators of the drift and diffusion coefficients in continuous-time stationary diffusion models using a damping function. \cite{bu2023uniform} obtained uniform and $\mathbb{L}_{p}$ convergence rates of kernel-based estimators of the instantaneous conditional mean and variance functions in continuous-time recurrent diffusion processes.  This paper aims to fill this gap by studying the uniform convergence rates of the N-W estimators for the coefficients of integrated diffusion processes.

Consider an integrated diffusion process given by $Y_{t}=\int_{0}^{t}X_{s} ds$, where $\{X_{t}\}$ follows the stochastic differential equation,
\begin{flalign}
\label{model-sde}
d X_{t}=b (X_{t}) dt+\sigma(X_{t}) d W_{t},
\end{flalign}
where $b $ and $\sigma$ are the drift and diffusion functions. $W$ is a standard Brownian motion. Suppose the process $\{Y_{t}\}$ is observed at $t_{i}=i\Delta,i=0,1,\cdots,n$ over the time interval $\left[0,T\right]$, where $T>0$ and $\Delta=T/n$ is the time distance between adjacent observations.
Without loss of generality, we assume the observations are equispaced.
Estimating the coefficients of $\{X_{t}\}$ directly from the observations $\{Y_{t_{i}}\}$ is difficult due to the unknown conditional distribution of $\{Y_{t}\}$. Instead, we use  $\breve{X}_{i\Delta}=\frac{ Y_{i\Delta}-Y_{(i-1)\Delta} }{\Delta}$ as an estimate of $X_{i\Delta}$. Given the sample $\{\breve{X}_{i\Delta};i=1,2,\cdots,n\}$, the following relationships hold:
\begin{align}
& \mathbb{E}\left[\frac{\left(\breve{X}_{(i+1) \Delta}-\breve{X}_{i \Delta}\right)^2}{\Delta} \Bigg| \mathcal{F}_{(i-1) \Delta}\right]=\frac{2}{3} \sigma^2(X_{(i-1) \Delta})+O_{P}(\Delta),\\
& \mathbb{E}\left[\frac{\breve{X}_{(i+1) \Delta}-\breve{X}_{i \Delta}}{\Delta} \Bigg| \mathcal{F}_{(i-1) \Delta}\right]= b (X_{(i-1) \Delta})+O_{P}(\Delta),
\end{align}
where $\mathcal{F}_{t}=\sigma\{X_{s},s\leq t\}$. Based on these relationships, the N-W estimators of $\sigma^{2}(x)$ and $b (x)$ can be defined as follows,
\begin{flalign}
\hat{\sigma}^{2}(x)&=\frac{ \frac{1}{T} \sum_{i=1}^{n-1} K_{h}(\breve{X}_{(i-1)\Delta}-x)
 \frac{3}{2}(\breve{X}_{(i+1)\Delta}-\breve{X}_{i\Delta})^{2}  }
{ \frac{\Delta}{T} \sum_{i=1}^{n-1} K_{h}(\breve{X}_{(i-1)\Delta}-x)},
\label{estimator-sigma}  \\
\hat{b }(x)&=\frac{ \frac{1}{T} \sum_{i=1}^{n-1} K_{h}(\breve{X}_{(i-1)\Delta}-x)
(\breve{X}_{(i+1)\Delta}-\breve{X}_{i\Delta})  }
{ \frac{\Delta}{T} \sum_{i=1}^{n-1} K_{h}(\breve{X}_{(i-1)\Delta}-x)
  },  \label{estimator-mu}
\end{flalign}
where $K$ is the kernel function, $h$ is the bandwidth, $K_{h}\left(\cdot-x\right)=K\left((\cdot-x)/h\right)/h$.

In comparison,
the N-W estimators for $\sigma^{2}(x)$ and $b (x)$, based on direct observations $\{X_{i\Delta};i=0,1,\cdots,n\}$
are given by,
\begin{flalign}
\tilde{\sigma}^{2}(x)&=\frac{ \frac{1}{T} \sum_{i=1}^{n} K_{h}(X_{(i-1)\Delta}-x)
(X_{i\Delta}-X_{(i-1)\Delta})^{2}  }
{ \frac{\Delta}{T} \sum_{i=1}^{n} K_{h}(X_{(i-1)\Delta}-x)
 },     \label{estimatorsigmainitial}\\
\tilde{b }(x)&=\frac{ \frac{1}{T} \sum_{i=1}^{n} K_{h}(X_{(i-1)\Delta}-x)
(X_{i\Delta}-X_{(i-1)\Delta})  }
{ \frac{\Delta}{T} \sum_{i=1}^{n} K_{h}(X_{(i-1)\Delta}-x)
  }. \label{estimatormuinitial}
\end{flalign}
To establish pointwise properties, such as the consistency  of $\tilde{\sigma}^{2}(x)$ (or $\tilde{b }(x)$), it suffices to verify that
$\hat{\sigma}^{2}(x)-\tilde{\sigma}^{2}(x)=o_{P}(1)$ (or $\hat{b }^{2}(x)-\tilde{b }^{2}(x)=o_{P}(1)$). This follows from the relationships,
\begin{flalign*}
&\mathbb{E}\left[\frac{\frac{3}{2}\left(\breve{X}_{(i+1) \Delta}-\breve{X}_{i \Delta}\right)^2}{\Delta} -\frac{(X_{i\Delta}-X_{(i-1)\Delta})^{2} }{\Delta}\right]=O( \Delta),   \\
 &\mathbb{E}\left[\frac{\breve{X}_{(i+1) \Delta}-\breve{X}_{i \Delta}}{\Delta} -\frac{X_{i\Delta}-X_{(i-1)\Delta} }{\Delta}\right]=O(\Delta).
\end{flalign*}
These results also play a crucial role in establishing the asymptotic normality of the estimators. However, the uniform convergence rate cannot be derived directly from these expressions.
To obtain the uniform rate, it is necessary to control the discretization error using the global modulus of continuity of the process under a long observation time horizon ($T\to \infty$).
The covering-number technique is employed to derive the rate over unbounded support.
A key challenge in proving the uniform convergence rate of $\hat{\sigma}^{2}(x)$ arises due to time inconsistency. Specifically, unlike $\tilde{\sigma}^{2}(x)$, where the same observations are used, the terms $(\breve{X}_{(i+1)\Delta}-\breve{X}_{i\Delta})^{2} $ and
$K_{h}(\breve{X}_{(i-1)\Delta}-x)$ in $\hat{\sigma}^{2}(x)$ involve different observations. This makes it difficult to decompose the estimator into a structure that ensures  $\hat{\sigma}^{2}(x)$ achieves an optimal convergence rate applying It$\mathrm{\hat{o}}$ formula.
Furthermore, since the process $\{X_{t}\}$ is unobservable, we need to deal with terms like $(Y_{(i+1)\Delta} - Y_{i\Delta})^{2}$ rather than $(X_{(i+1)\Delta} - X_{i\Delta})^{2}$, as required for $\tilde{\sigma}^{2}(x)$.
Notably, the convergence rate of the diffusion estimator for the integrated diffusion process $((\log n)^{3}/n)^{2/5}$, differing from the rate $(\log n/n)^{2/5}$ for  continuous diffusion processes. See Remark \ref{remarkdiffusion} for details.

This paper is structured as follows. Section $2$ introduces the assumptions on the process $\{X_{t}\}$ and presents our main results. Section $3$ provides auxiliary lemmas, while Section $4$ contains the proofs. The finite-sample performance of the estimators is investigated through Monte-Carlo simulations in Section $5$.

\section{Assumptions and Main Theorems}
Let $\mathfrak{D}=(\breve{l},\breve{r})$ denote the domain of the process $\{X_{t}\}$, where $-\infty\le \breve{l}<\breve{r}\le \infty$.
In order to establish the uniform convergence rates for the estimators, we impose the following assumptions on the process $\{X_{t}\}$, and the kernel function.

\begin{assumption}
\label{assumption-stationary and mixing}
$(1)$ $b (\cdot):\mathfrak{D}\to \mathbb{R}$, $\sigma(\cdot):\mathfrak{D}\to (0,\infty)$ are both twice continuously differentiable, and
they satisfy global Lipschitz and linear growth conditions.     \\
$(2)$
The process $\{X_t\}_{t\geq 0}$ is strictly stationary and $\alpha$-mixing, with an invariant probability density $\pi(\cdot)$. Moreover, its mixing coefficient exhibits polynomial decay, i.e., there exist constants $\breve{\beta}>0$ and $\breve{A}>0$ such that $\alpha(s)\le \breve{A}s^{-\breve{\beta}}$.      \\
$(3)$
There exists a constant $q > 0$ such that $\mathbb{E}[|X_{t}|^{2+q}]<\infty$.
\end{assumption}
\begin{remark}
Condition $(1)$ ensures the existence of a unique strong solution to the model (\ref{model-sde}) up to an explosion time.
The mixing properties of the process are discussed in \cite{Doukhan1995Mixing, veretennikov1997polynomial}.
A sufficient condition for ensuring that the process $\{X_{t}\}$ is positive recurrent, with a stationary distribution and time-invariant density, is that
$\int_{\breve{l}}^{\breve{r}}m(x)dx<\infty$, where
$m(x)=\frac{1}{\sigma^{2}(x)}\exp\left\{ 2\int_{c}^{x}\frac{b (u)}{\sigma^{2}(u)}du\right\}$, for $c\in(\breve{l},\breve{r})$.
A positive recurrent process initiated at its stationary distribution is strictly stationary.

Given the linear growth condition of $b $ and $\sigma$, along with condition $(3)$, we have $\mathbb{E}[|b (X_{t})|^{2+q}]<\infty$, and $\mathbb{E}[|\sigma(X_{t})|^{2+q}]<\infty$. Consequently,
$\mathbb{E}[|X_{t}-X_{s}|^{2+q}]<C|t-s|^{\frac{2+q}{2}}$.
Therefore, according to Lemma A.1 in \cite{kanaya2016estimation}, as $\Delta\to 0$, the modulus of continuity of  $X_{t}$ satisfies,
\begin{equation}
\label{result-incrementop}
\sup_{|s-t|\in[0,\Delta],s,t\in[0,\infty)}|X_{s}-X_{t}| =O_{a.s.} (\Delta^{r}),\quad \text{for\ any\ } r\in\left[0,%\frac{\frac{2+\bar{q}}{2}-1}{2+\bar{q}}=
\frac{1}{2}-\frac{1}{2+q}\right).
\end{equation}
\end{remark}

\begin{assumption}
\label{assumption-kernel and damping}
$(1)$ $K(\cdot):\left(\mathbb{R}\rightarrow
\mathbb{R}\right)$ is symmetric and of bounded variation, satisfying the normalization condition  $\int_{\mathbb{R}}K(x)dx=1$. Additionally, it is uniformly bounded, and
$\int_{\mathbb{R}}x^{2}\vert K(x)\vert dx\le \breve{K}$ for $\breve{K}>0$.      \\
$(2)$ $K\left(\cdot\right)$ is compactly supported on $[-\bar{c}_{K},\bar{c}_{K}]$ for some constant $\bar{c}_{K}>0$ and is Lipschitz continuous.
\end{assumption}

Our main results are stated as follows.
%%%%%%%%%%%%%%%%%%%%%%%%%%%%%%%%%%%%%%%%%%% diffusion
\begin{theorem}
\label{thm-diffusion}
Suppose the following conditions hold:    \\
$(1)$ Assumptions $\ref{assumption-stationary and mixing}$ and $\ref{assumption-kernel and damping}$ are satisfied.  $\sigma$ and $(\sigma^{2}\pi)''$ are uniformly bounded.
$|\partial^{2} \tilde{f} (x)|=O(|x|)$ as $|x|\to\infty$, where $\tilde{f}=b ,\sigma$.  \\
$(2)$  The density function satisfies $\sup_{x}|\partial^{k}\pi(x)|<\infty$ for $k=0,1,2$, and
$\sup_{x,y}|\pi_{t,t+s}(x,y)|<\infty$ for all $s,t\geq 0$, where $\pi_{t,t+s}$ is the joint density of $X_{t}$ and $X_{t+s}$.       \\
$(3)$ $\Delta^{-1}=O(n^{\kappa})$, $\Delta^{\frac{1}{2}-\frac{1}{2+q}}/h\to 0$, and $(\log n)/(n^{\theta}h)\to 0$ as $n\rightarrow \infty$ and $\Delta, h \to 0$,
where constants $\kappa \in(0,1/2)$ and $\theta\in(0,1)$ satisfy $1-\theta -\kappa>0$.
Furthermore, $q\geq \frac{2(1+\theta-\kappa)}{\kappa}$.      \\
Let $a_{n,T}=\Delta h^{-1/(1+q)} +\sqrt{(\log{n})^{3}/(nh)} + h^2$.
If
\begin{flalign}
\label{betaconditiondiff}
\breve{\beta}>\max\left\{
\frac{2+3\theta}{1-\theta -\kappa},\ \frac{2+1/ (2+q)}{1-2\kappa} \right\},
\end{flalign}
and   $a_{n,T}/\delta_{n,T}\to 0$, $\Delta^{\frac{1}{2}-\frac{1}{2+q}}/(h\delta_{n,T})\to 0$, $\sqrt{(\log{n})/(n^{\theta}h)}/\delta_{n,T}\to 0$, where $\delta_{n,T}:=\inf_{\vert x\vert\le \bar{b}_{n,T}}\pi(x)>0$ for any sequence $\bar{b}_{n,T} \to \infty$ as $T,n \to \infty$.
Then we have, as $\Delta, h \to 0$ and $n,T \to \infty$,
\begin{equation}
\label{conclusionthmdiffusion}
\sup_{\vert x\vert\le \bar{b}_{n,T}} \left\vert\hat{\sigma}^{2}(x)-\sigma^{2}(x)\right\vert
=O_{P}\left(a_{n,T}/\delta_{n,T}\right).
\end{equation}  	
\end{theorem}
%%%%%%%%%%%%%%%%%%%%%%%%%%%%%%%%%%%%%%%%%%% REMARK
\begin{remark}
\label{remarkdiffusion}
The convergence rate $a_{n,T}$ consists of three components: the discretization error, the variance-effect term, and the smoothing error. In estimating the diffusion coefficient of a continuous diffusion process, we need to handle a martingale term defined as follows,
\begin{equation*}
\widehat{M}_{T}(x):=\frac{1}{Th}\sum_{j=1}^{n-1}
K((X_{j\Delta}-x)/h)\int_{j\Delta}^{(j+1)\Delta}
\widehat{\rho}_{s,j\Delta}
\mathbb{I}_{\{\widehat{\rho}_{s,j\Delta}\leq\sqrt{\Delta}\log(1/\Delta)\}}dW_{s},
\end{equation*}
where $\widehat{\rho}_{s,j\Delta}$ is an adapted process. Typically,
$\widehat{\rho}_{s,j\Delta}$  corresponds to a term like
$\int_{j\Delta}^{s}f(X_{s})dW_{s}$. If $f$ is uniformly bounded, then for sufficiently small $\Delta$, we have $\mathbb{I}_{\{\widehat{\rho}_{s,j\Delta}\leq\sqrt{\Delta}\log(1/\Delta)\}} =1$ almost surely.  Using the proof of Lemma \ref{lemmafydw}, we obtain $\sup_{ x\in \mathbb{R}}|\widehat{M}_{T}(x)|=O_{P}(\sqrt{(\log{n})^{3}/(nh)})$ under appropriate conditions, which is the same as the convergence rate of the variance-effect term for the integrated diffusion process. Moreover, in this theorem, if the process has finite moments of all orders, then if
$\kappa=2/5$ and $h=((\log{n})^{3}/n)^{1/5}$, the convergence rate of the diffusion coefficient is $((\log{n})^{3}/n)^{2/5}$.
\end{remark}

%%%%%%%%%%%%%%%%%%%%%%%%%%%%%%%%%%%%%%%%%%% drift
\begin{theorem}
\label{thm-drift}
Assume that    \\
$(1)$ Assumptions $\ref{assumption-stationary and mixing}$ and $\ref{assumption-kernel and damping}$ are satisfied.
$b (\cdot)$ and $(b \pi)''(\cdot)$ are uniformly bounded.      \\
$(2)$ $\sup_{x}|\partial^{k}\pi(x)|<\infty$ for $k=0,1,2$, and
$\sup_{x,y}|\pi_{t,t+s}(x,y)|<\infty$.  \\
$(3)$ As $\Delta, h \to 0$ and $T\rightarrow \infty$,  $\Delta^{-1}=O(T^{\bar{\kappa}})$ for some $\bar{\kappa}>0$,
$(\log T)/(T^{\bar{\theta}}h) \to 0$ for $\bar{\theta}\in (0,1)$, $\Delta^{\frac{1}{2}-\frac{1}{2+q}}/h\to 0$, moreover, 
$ 1-(1+\frac{4}{q})\bar{\theta}-\frac{2\bar{\kappa}}{q}>0$. \\
Let $a^*_{n,T}=\Delta^{\frac{1}{2}-\frac{1}{2+q}} + \sqrt{(\log T)/(T^{\bar{\theta}}h)} + h^2$. If
\begin{equation}
\label{betaconditiondrift}
\breve{\beta}>\max\left\{
\frac{ \frac{3}{2}+\bar{\theta}+\bar{\kappa}}{ 1-\bar{\theta}(1+\frac{4}{q})-\frac{2\bar{\kappa}}{q}}-2,\
\frac{(2+\frac{1}{2+q})\bar{\theta}}{1-\bar{\theta}} \right\},
\end{equation}
and the conditions $a^*_{n,T}/\delta_{n,T}\to 0$ and  $\Delta^{\frac{1}{2}-\frac{1}{2+q}}/(h\delta_{n,T})\to 0$ hold, then as $\Delta, h \to 0$ and $n,T \to \infty$, we have,
\begin{equation}
\label{conclusionthmdrift}
\sup_{\vert x\vert\le \bar{b}_{n,T}} \left\vert\hat{b }(x)-b (x)\right\vert
=O_{P}\left(a^*_{n,T}/\delta_{n,T}\right).
\end{equation}
\end{theorem}
\section{Auxiliary Lemmas}
%%%%%%%%%%%%%%%%%%%%%%%%%%%%%%%%%%%%%%%%%%%%%%%%%%%%%%%% lemma 1
\begin{lemma}[Lemma A.3 in \cite{kanaya2017uniform}]
\label{lemma-covering number}
Let $K(\cdot)$ $(: \mathbb{R} \rightarrow \mathbb{R})$ be a function of bounded variation, satisfying $\sup _{x \in \mathbb{R}}\vert K(x)\vert \leq \breve{K}$ for some $\breve{K}>0$. Denote $\mathbb{L}^{\breve{r}}(Q)$ as the space of functions $\breve{g}\left(: \mathbb{R} \rightarrow \mathbb{R}\right)$ satisfying $\left[\int \vert \breve{g}\vert^{\breve{r}} dQ\right]^{1/\breve{r}}<\infty$, where $\breve{r}\geq 1$ and $Q$ is a probability measure on $\mathbb{R}$. Consider the class of rescaled and translated versions of $K$:
\begin{equation*}
\mathcal{K}:=\left\{K\left(\left(\cdot-x\right)/h\right) \vert x \in \mathbb{R},\ \ h>0\right\}.
\end{equation*}
Then, for $\varepsilon \in(0,1)$, the covering number of $\mathcal{K}$ satisfies
\begin{equation*}
\sup _{Q} N\left(\varepsilon 8 \breve{K}, \mathcal{K}, \mathbb{L}^{\breve{r}}\left(Q\right)\right) \leq \Lambda \varepsilon^{-4 \breve{r}},\ \Lambda>0\ \text{is\  independent\ of}\ Q.
\end{equation*}
The quantity $ N\left(\varepsilon 8 \breve{K}, \mathcal{K}, \mathbb{L}^{\breve{r}}\left(Q\right)\right)$ denotes the minimal  number of $\varepsilon 8 \breve{K}$-balls in $\mathbb{L}^{\breve{r}}(Q)$ need to cover $\mathcal{K}$, where an $\varepsilon 8 \breve{K}$-ball centered at $\breve{g}\in \mathbb{L}^{\breve{r}}(Q)$ is defined as $\Big\{\check{f}\in \mathbb{L}^{\breve{r}}(Q)\big\vert\left[\int \vert\check{f}-\breve{g}\vert^{\breve{r}} dQ\right]^{\frac{1}{\breve{r}}}  <\varepsilon 8 \breve{K}\Big\}$.
\end{lemma}
%%%%%%%%%%%%%%%%%%%%%%%%%%%%%%%%%%%%%%%%%%%%%%%%%%%%%%%%%%%% lemma
\begin{lemma}[$(15)$ in \cite{kanaya2017uniform}]
\label{lemmarhodwasbound}
Let $\mathbb{P}[\Omega^*]=1$, and for each $\omega \in \Omega^*$, there exists some $\breve{\Delta}>0$ such that for any $\Delta \in[0, \breve{\Delta}]$,
\begin{equation*}
\sup _{|t_{2}-t_{1}| \in[0, \Delta];t_{1}, t_{2} \in[0, \infty)}\left|\int_{t_{1}}^{t_{2}} \breve{\rho}_u d W_u\right| \leq \sqrt{2 \Delta \log (1 / \Delta)} \sqrt{\max \left\{1, \sup _{u \in[0, \infty)} \breve{\rho}_{u}^2\right\}},
\end{equation*}
where $\{\breve{\rho}_u\}$ is uniformly bounded over $u \in[0, \infty)$, and $\sup _{u \in[0, \infty)} \breve{\rho}_{u}^2>0$. Additionally, the process $\{M_s\}_{s \geq 0}$ defined as $M_s:=\int_0^s \breve{\rho}_u d W_u$ is well-posed.
\end{lemma}
%%%%%%%%%%%%%%%%%%%%%%%%%%%%%%%%%%%%%%%%%%%%%%%%%%%%%%%% lemma 2
\begin{lemma}
\label{lemmafydw}
Assume that  Assumption $\ref{assumption-stationary and mixing}$ and condition $(1)$ of  Assumption  $\ref{assumption-kernel and damping}$  hold. $\sigma$ is uniformly bounded and 
$\sup_{x}|\pi(x)|<\infty$. As $\Delta, h \to 0$ and $n\rightarrow \infty$, we have the following conditions: 
$\Delta^{-1}=O(n^{\kappa})$ for some constant $\kappa>0$, and $(\log n)/(n^{\theta}h) \to 0$ for $\theta\in (0,1)$. Additionally, 
$q\geq \frac{2(1+\theta-\kappa)}{\kappa}$.
If
\begin{flalign*}
\breve{\beta}>\frac{2+3\theta}{1-\theta-\kappa},
\end{flalign*}
then as $n,T \to \infty$ and $\Delta, h \to 0$, we obtain,
\begin{flalign*}
&\sup\limits_{x\in \mathbb{R}}
\bigg|\frac{1}{T\Delta} \sum_{i=1}^{n-1} K_{h}\left(X_{i\Delta}-x\right)
\int_{i\Delta}^{(i+1)\Delta}\left(
\sigma(X_{t})(Y_{t}-Y_{i\Delta})-\sigma(X_{t})\int_{i\Delta}^{t}X_{i\Delta}ds\right)dW_{t}\bigg|
=O_{P}(\sqrt{(\log n)^{3}/(nh)}).
\end{flalign*}
\end{lemma}
%%%%%%%%%%%%%%%%%%%%%%%%%%%%%%%%%%%%%%%
\begin{proof}
Using triangle inequality, we have,
\begin{flalign*}
&\sup\limits_{x\in \mathbb{R}}
\bigg|\frac{1}{T\Delta} \sum_{i=1}^{n-1} K_{h}\left(X_{i\Delta}-x\right)
\int_{i\Delta}^{(i+1)\Delta}\left(
\sigma(X_{t})(Y_{t}-Y_{i\Delta})-\sigma(X_{t})\int_{i\Delta}^{t}X_{i\Delta}ds\right)dW_{t}\bigg|  \\
&\leq\sup\limits_{x\in \mathbb{R}}
\bigg|\frac{1}{T} \sum_{i=1}^{n-1} K_{h}\left(X_{i\Delta}-x\right)
\int_{i\Delta}^{(i+1)\Delta}\rho_{t,i}dW_{t}\bigg|
+\sup\limits_{x\in \mathbb{R}}
\bigg|\frac{1}{T} \sum_{i=1}^{n-1} K_{h}\left(X_{i\Delta}-x\right)
\int_{i\Delta}^{(i+1)\Delta}\tilde{\rho}_{t,i}dW_{t}\bigg| \\
&\quad+\sup\limits_{x\in \mathbb{R}}
\bigg|\frac{1}{T} \sum_{i=1}^{n-1} K_{h}\left(X_{i\Delta}-x\right)
\int_{i\Delta}^{(i+1)\Delta}\tilde{\rho}^{(1)}_{t,i}dW_{t}\bigg|
+\sup\limits_{x\in \mathbb{R}}
\bigg|\frac{1}{T} \sum_{i=1}^{n-1} K_{h}\left(X_{i\Delta}-x\right)
\int_{i\Delta}^{(i+1)\Delta}\bar{\rho}_{t,i}dW_{t}\bigg|        \\
&:=V_{1}+V_{2}+V_{3}+V_{4},
\end{flalign*}
where
\begin{flalign*}
\rho_{t,i}&:=\frac{1}{\Delta}\sigma(X_{t})
\left[\int_{i\Delta}^{t}\int_{i\Delta}^{s}
b (X_{u})\mathbb{I}_{\{|X_{i\Delta}|>\bar{\psi}_{n}\}}duds\right] \bar{e}_{\Delta,i},  \\
\tilde{\rho}_{t,i}&:=\frac{1}{\Delta}\sigma(X_{t})
\left[\int_{i\Delta}^{t}  \left(\int_{i\Delta}^{s}
b (X_{u})\mathbb{I}_{\{|X_{i\Delta}|\leq \bar{\psi}_{n}\}}du
+\int_{i\Delta}^{s}\sigma(X_{u})dW_{u}\right)e_{\Delta,i}
ds\right]\bar{e}_{\Delta,i}  ,    \quad t\in[i\Delta,(i+1)\Delta],      \\
\tilde{\rho}^{(1)}_{t,i}&:=\frac{1}{\Delta}\sigma(X_{t})
\left[\int_{i\Delta}^{t}  \left(\int_{i\Delta}^{s}
b (X_{u})\mathbb{I}_{\{|X_{i\Delta}|\leq \bar{\psi}_{n}\}}du
+\int_{i\Delta}^{s}\sigma(X_{u})dW_{u}\right)(1-e_{\Delta,i})
ds\right] \bar{e}_{\Delta,i},         \\
\bar{\rho}_{t,i}&:=
\left(\rho_{t,i}+\tilde{\rho}_{t,i}+\tilde{\rho}^{(1)}_{t,i}\right)(1-\bar{e}_{\Delta,i}). \\
e_{\Delta,i}&:=\mathbb{I}_{\left\{\sup_{s\in[i\Delta,(i+1)\Delta]}
\left|\int_{i\Delta}^{s}\sigma(X_{u})dW_{u}\right|\leq \sqrt{\Delta}\log(1/\Delta)\right\}},       \quad
\bar{e}_{\Delta,i}:=\mathbb{I}_{\left\{\sup_{s\in[i\Delta,(i+1)\Delta]}
\left|X_{s}-X_{i\Delta}\right|\leq 1\right\}},
\end{flalign*}
where $\bar{\psi}_{n}$ is a sequence of positive numbers.
According to Lemma $\ref{lemmarhodwasbound}$ and $(\ref{result-incrementop})$, $V_{3}=0$ and $V_{4}=0$ almost surely for sufficiently small $\Delta$.
Using H{\"o}lder and BDG inequalities, we have,
\begin{flalign*}
\mathbb{E}\left[V_{1}\right]
&\leq \frac{\breve{K}(n-1)}{Th} \left[\frac{1}{n-1}\sum_{i=1}^{n-1}
\mathbb{E}\left[\left(\int_{i\Delta}^{(i+1)\Delta}
\rho_{t,i} dW_{t}\right)^{2}\right]\right]^{1/2} \\
&\leq \frac{\breve{K}(n-1)}{Th}\left[\mathbb{E}\left[\int_{i\Delta}^{(i+1)\Delta}
\rho^{2}_{t,i} dt\right]\right]^{1/2}    \\
&\leq \frac{\breve{K}\Delta^{3/2}}{\Delta h} \left(\frac{\mathbb{E}[|X_{i\Delta}|^{2+q}]}{\bar{\psi}_{n}^{q}}\right)^{1/2}
=O\left(\frac{\Delta^{1/2}}{h\bar{\psi}_{n}^{q/2}}\right),
\end{flalign*}
the third inequality follows from
\begin{flalign*}
|b (X_{u})|\mathbb{I}_{\{|X_{i\Delta}|>\bar{\psi}_{n}\}}\bar{e}_{\Delta,i}
=\left|b (X_{u})-b (X_{i\Delta})+b (X_{i\Delta})\right|
\mathbb{I}_{\{|X_{i\Delta}|>\bar{\psi}_{n}\}}\bar{e}_{\Delta,i},
\end{flalign*}
according to the Lipschitz property of $b $, $|b (X_{u})-b (X_{i\Delta})|\bar{e}_{\Delta,i}\leq C$ almost surely for sufficiently small $\Delta$. Note that $h=O(n^{-\theta}\log n)$ and $\Delta=O(n^{-\kappa})$.
set
$\bar{\psi}_{n}=\left(\frac{n\Delta }{h(\log n)^{3}}\right)^{1/q}$, we have
$V_{1}=O_{P}(\sqrt{(\log n)^{3}/(nh)})$.
Moreover, if
$q\geq \frac{2(1+\theta-\kappa)}{\kappa}$,  we have
$\Delta \bar{\psi}_{n}\leq \sqrt{\Delta}\log n$ as $\Delta, h\to 0$ and $n\to\infty$.
%%%%%%%%%%%%%%%%%%%%%%%%%%%%%%%%%%%%%%%%%%%%%%%%%%%%%%%%
To find the rate of $V_{2}$, for each $h>0$, we define the function class
$\mathcal{K}(h)=\left\{K((\cdot-x)/h)\vert x\in \mathbb{R}\right\}$. With  Lemma \ref{lemma-covering number}, we can cover $\mathcal{K}(h)$ with a finite collection   $\left\{\mathcal{K}_{k}(h)\right\}_{k=1}^{\breve{v}(h)}$ such that: $\mathcal{K}(h)\subset \bigcup_{k=1}^{\breve{v}(h)}\mathcal{K}_k(h)$; each  $\mathcal{K}_k(h)$  is centered at   $\breve{g}_{k}\left(\cdot\right)=K\left((\cdot-x_k)/h\right)$;  for any $\varepsilon\in(0,1)$, $\breve{g}\in \mathcal{K}_k(h)$, and any probability measure $Q$,
\begin{equation}
\label{proof-covering-number}
\int\vert \breve{g}-\breve{g}_{k}\vert dQ\leq \varepsilon 8\breve{K},\ \ 
\breve{v}(h)\leq \Lambda \varepsilon^{-4}.
\end{equation}
Then, $V_{2}$ can be bounded by,
\begin{flalign*}
V_{2}&\leq \max_{k\in \{1,\cdots, \breve{v}(h)\}}\sup_{\breve{g}\in \mathcal{K}_k(h)}\left\vert \frac{1}{Th} \sum_{i=1}^{n-1}\left[\breve{g}_{k}(X_{i\Delta})-\breve{g}(X_{i\Delta})\right]
\int_{i\Delta}^{(i+1) \Delta}  \tilde{\rho}_{t,i} dW_{t}\right\vert  \\
& \quad
+\max_{k\in \{1,\cdots, \breve{v}(h)\}}\left\vert\frac{1}{Th}\sum_{i=1}^{n-1} \breve{g}_{k}(X_{i\Delta})\int_{i\Delta}^{(i+1) \Delta} \tilde{\rho}_{t,i}dW_{t} \right\vert  \\
&=: \mathcal{V}_{21} + \mathcal{V}_{22}.
\end{flalign*}
For $\mathcal{V}_{21}$, it is easy to obtain that,
\begin{flalign*}
\mathcal{V}_{21}
&\leq
\frac{n-1}{Th}\left\{\max_{k\in \{1,\cdots, \breve{v}(h)\}}\sup_{g\in \mathcal{K}_k(h)} \frac{1}{n-1} \sum_{i=1}^{n-1}\left|\breve{g}_{k}(X_{i\Delta})-\breve{g}(X_{i\Delta})\right|^{2}\right\}^{\frac{1}{2}}    %\\
%&\quad\times
\left\{\frac{1}{n-1}\sum_{i=1}^{n-1}
\left\vert\int_{i\Delta}^{(i+1) \Delta}  \tilde{\rho}_{t,i} dW_{t}\right\vert^{2}\right\}^{\frac{1}{2}}   \\
&\leq \frac{n-1}{Th} \varepsilon 8\breve{K}\times O_{P}(\Delta\log(1/\Delta))    \\
&= O_{P}\left(\frac{\log(1/\Delta) \varepsilon }{h}\right) ,
\end{flalign*}
where the first and second inequalities follow from the H{\"o}lder and BDG inequalities, respectively.
Let
\begin{flalign*}
\varepsilon=\sqrt{h\log n/n} ,
\end{flalign*}
we can get $\mathcal{V}_{21}=O_{P}(\sqrt{(\log n)^{3}/(nh)})$.
For $\mathcal{V}_{22}$, define
\begin{flalign*}
\check{M}_{T}:=\sum_{i=1}^{n-1} \breve{g}_{k}(X_{i\Delta})\int_{i\Delta}^{(i+1) \Delta} \tilde{\rho}_{t,i} dW_{t},
\end{flalign*}
note that $\check{M}_{T}$ can be represented as a continuous martingale with quadratic variation
\begin{flalign*}
\langle \check{M} \rangle_{T}=\sum_{i=1}^{n-1} \breve{g}_{k}^{2}(X_{i\Delta}) \int_{i\Delta}^{(i+1) \Delta}\tilde{\rho}_{t,i}^{2}dt.
\end{flalign*}
Denote
\begin{flalign*}
Z_{n,i}(x_k,h)
:=\breve{g}_{k}^{2}(X_{i\Delta}) \int_{i\Delta}^{(i+1) \Delta}\tilde{\rho}_{t,i}^{2}dt
-\mathbb{E}\left[\breve{g}_{k}^{2}(X_{i\Delta}) \int_{i\Delta}^{(i+1) \Delta}\tilde{\rho}_{t,i}^{2}dt\right].
\end{flalign*}
According to the boundedness of $K$, we have,
\begin{flalign*}
|Z_{n,i}(x_k,h)|&\leq C_{1}\Delta^{2}(\log(1/\Delta))^{2}, \\
\mathbb{E}\left[\left(\sum_{i=1}^{m}Z_{n,i}(x_k,h)\right)^{2}\right]&\leq C_{2}m^{2}\Delta^{4}(\log(1/\Delta))^{4}h, \quad 0<m\leq (n-1),
\end{flalign*}
where $C_{1}$, $C_{2}>0$.
Note that
\begin{flalign*}
\langle \check{M} \rangle_{T}
&=\sum_{i=1}^{n-1}Z_{n,i}(x_k,h)
+\sum_{i=1}^{n-1}\mathbb{E}\left[\breve{g}_{k}^{2}(X_{i\Delta}) \int_{i\Delta}^{(i+1) \Delta}\tilde{\rho}_{t,i}^{2}dt\right] \\
&\leq \sum_{i=1}^{n-1}Z_{n,i}(x_k,h)+C_{3}Th\Delta\log(1/\Delta)^{2}  ,
\end{flalign*}
where $C_{3}>0$.
For sufficiently large $a>0$, we can get,
\begin{flalign*}
& P\left[\mathcal{V}_{22}\geq a \sqrt{(\log n)^{3}/(nh)} \right]        \\
&\leq \Lambda \varepsilon^{-4}\left\{
P\left[|\check{M}_{T}|\geq aTh \sqrt{(\log n)^{3}/(nh)},
\langle \check{M} \rangle_{T}\leq 2aTh\Delta(\log n)^{2} \right]  \right.       \\
&\quad\left.+P\left[\sum_{i=1}^{n-1}|Z_{n,i}(x_k,h)|
+C_{3}Th\Delta\log(1/\Delta)^{2}> 2aTh\Delta(\log n)^{2}\right]
\right\}      \\
&\leq \Lambda \varepsilon^{-4}\left\{
P\left[|\check{M}_{T}|\geq aTh \sqrt{(\log n)^{3}/(nh)},
\langle \check{M} \rangle_{T}\leq 2aTh\Delta(\log n)^{2} \right]  %\right.       \\
+P\left[\sum_{i=1}^{n-1}|Z_{n,i}(x_k,h)|> aTh\Delta(\log n)^{2}\right]
\right\}      \\
&\leq \Lambda \varepsilon^{-4}\left[
\exp \left(-\frac{a \log n}{4}\right)+4\breve{A}nm^{-1-\breve{\beta}}\Delta^{-\breve{\beta}} \right.\\
&\quad\left.+4\exp
\left(-\frac{a^{2}T^{2}h^{2}\Delta^{2}(\log n)^{4} }{64n
\left(C_{2}m\Delta^{4}(\log(1/\Delta))^{4}h\right)  +\frac{8}{3}C_{1}\Delta^{2}(\log(1/\Delta))^{2}aTh\Delta(\log n)^{2} m }\right)    \right]      \\
&\leq \Lambda \varepsilon^{-4}\left[
n^{-\frac{a}{4} }
+4\breve{A}n^{-\breve{\beta}}h^{-1-\breve{\beta}}(\log n)^{1+\breve{\beta}}\Delta^{-\breve{\beta}}
+4\exp\left(-\frac{a\log n}{64C_{2}/a+(8/3)C_{1}} \right)
\right]\\
&= O\left((\log n)^{-2}n^{2}h^{-2}
(n^{-\frac{a}{4} }+4n^{-C_{4} a}+4n^{-\breve{\beta}+(1+\breve{\beta})\theta+\breve{\beta}\kappa} )  \right)     \\
&=O\left((\log n)^{-4}n^{2+2\theta}
(n^{-\frac{a}{4} }+4n^{-C_{4} a}+4n^{-\breve{\beta}+(1+\breve{\beta})\theta+\breve{\beta}\kappa} )  \right),
\end{flalign*}
where $C_{4}=1/(64C_{2}/a+(8/3)C_{1})$.
The third inequality directly follows from the exponential inequality of continuous martingales and the Bernstein inequality applied to strong mixing arrays.
The fourth inequality holds with
$m=\frac{nh}{\log n}$.
It is worth mentioning that for sufficiently large
$a,n$ and sufficiently small $h, \Delta$, $m$ satisfies,
\begin{flalign*}
m\leq \min\left\{n-1,
\frac{aTh\Delta(\log n)^{2}}{C_{1}\Delta^{2}(\log(1/\Delta))^{2}}\right\}.
\end{flalign*}
If
$2+2\theta-\breve{\beta}+(1+\breve{\beta})\theta+\breve{\beta}\kappa\leq 0$,
i.e.,
$\breve{\beta}\geq\frac{2+3\theta}{1-\theta-\kappa}$,
we could get $\mathcal{V}_{22}=O_{P}(\sqrt{(\log n)^{3}/(nh)})$.
Consequently, $V_{2}=O_{P}(\sqrt{(\log n)^{3}/(nh)})$.
\end{proof}
%%%%%%%%%%%%%%%%%%%%%%%%%%%%%%%%%%%%%%%%%%%%%%%%%%%%%%%% lemma 3
\begin{lemma}
\label{lemmafdwtmu}
Suppose Assumption $\ref{assumption-stationary and mixing}$ and condition $(1)$ of Assumption $\ref{assumption-kernel and damping}$ hold.
$\sup_{x}|\pi(x)|<\infty$. As $\Delta, h \to 0$, $T\rightarrow \infty$, we have $\Delta^{-1}=O(T^{\bar{\kappa}})$  for some $\bar{\kappa}>0$, and $(\log T)/(T^{\bar{\theta}}h) \to 0$ 
for some constant $\bar{\theta}\in (0,1)$. 
Moreover, $1-(1+\frac{4}{q})\bar{\theta}-\frac{2\bar{\kappa}}{q}>0$.
If
$\breve{\beta}\geq\frac{ \frac{3}{2}+\bar{\theta}+\bar{\kappa}}{ 1-\bar{\theta}(1+\frac{4}{q})-\frac{2\bar{\kappa}}{q}}-2,$
then as $n, T \to\infty$ and $\Delta, h \to 0$,
\begin{equation*}
\label{lemmarhodw}
\sup\limits_{x\in \mathbb{R}}
\bigg|\frac{1}{T} \sum_{i=1}^{n-1} K_{h}\left(X_{(i-1)\Delta}-x\right)
\int_{i\Delta}^{(i+1)\Delta}\sigma(X_{t})dW_{t}\bigg|
=O_{P}(\sqrt{(\log T)/(T^{\bar{\theta}}h)}).
\end{equation*}
\end{lemma}
%%%%%%%%%%%%%%%%%%%%%%%%%%%%%%%%%%%%%%%%%%%%%%%%%%%%%%%% lemma 6
\begin{lemma}
\label{lemmafef}
Suppose Assumptions $\ref{assumption-stationary and mixing}$ and $\ref{assumption-kernel and damping}$ hold. Assume that
$\sup_{x}|\pi(x)|<\infty$, $\sup_{x,y}|\pi_{t,t+s}(x,y)|<\infty$.
Let $\Delta^{-1}=O(n^{\kappa})$, and as $n\rightarrow \infty$ and $\Delta, h \to 0$, $(\log n)/(n^{\theta}h) \to 0$, where $\theta\in(0,1)$ and
$\kappa\in(0,1/2)$ are constants.    \\
$(1)$ If $\sigma$ is uniformly bounded. Moreover, $\breve{\beta}>\frac{2+\frac{1}{2+q}}{1-2\kappa}$, then as $n, T \rightarrow \infty$ and $\Delta, h \rightarrow 0$, it follows that
\begin{align*}
&\sup\limits_{x\in \mathbb{R}}
\left|\frac{1}{n} \sum_{i=1}^{n-1}\left[
K_{h}\left(X_{i\Delta}-x\right)\left(\sigma^{2}(X_{i\Delta})-\sigma^{2}(x)\right)
-\mathbb{E}\left[K_{h}\left(X_{i\Delta}-x\right)\left(\sigma^{2}(X_{i\Delta})-\sigma^{2}(x)\right) \right] \right]\right|      %\\
%&
=O_{P}(\sqrt{(\log n)/(nh)}).
\end{align*}
$(2)$ If $\breve{\beta}>\frac{(2+\frac{1}{2+q})\theta}{1-\theta-\kappa}$ with $1-\theta-\kappa>0$, then as $n, T \rightarrow \infty$ and $\Delta, h \rightarrow 0$, we obtain,
\begin{align*}
&\sup\limits_{x\in \mathbb{R}}
\left|\frac{1}{n} \sum_{i=1}^{n-1}\left[
K_{h}\left(X_{i\Delta}-x\right)
-\mathbb{E}\left[K_{h}\left(X_{i\Delta}-x\right) \right] \right]\right|
=O_{P}(\sqrt{(\log n)/(n^{\theta}h)}).
\end{align*}
\end{lemma}
\begin{proof}
Denote
\begin{align*}
H_{i}(x)=K\left((X_{i\Delta}-x)/h\right)\left(\sigma^{2}(X_{i\Delta})-\sigma^{2}(x)\right).
\end{align*}
Consider a compact interval $[-\tilde{b}_{n},\tilde{b}_{n}]\subset\mathbb{R}$, where $\tilde{b}_{n}\to \infty$ at a specified growth rate. This interval is covered by a finite collection of closed balls $\{I_{k}\}_{k=1}^{\bar{v}(n)}$, such that $[-\tilde{b}_{n},\tilde{b}_{n}]\subset\cup_{k=1}^{\bar{v}(n)}I_{k}$. Each ball $I_{k}$ is centered at  $x_k$ and has radius $\tilde{r}_{n}$, with $ \bar{v}(n)= \tilde{b}_{n}/\tilde{r}_{n}$. Then, we can obtain that,
\begin{align*}
&\sup\limits_{x\in \mathbb{R}} \left|
\frac{1}{nh} \sum_{i=1}^{n-1}
\left[H_{i}(x)-\mathbb{E}[H_{i}(x)]\right] \right|  \\
&\leq \sup\limits_{|x|>\tilde{b}_{n}}
\frac{1}{nh} \sum_{i=1}^{n-1}
\left|H_{i}(x)-\mathbb{E}[H_{i}(x)] \right|    %%% 1
+\max\limits_{k\in\{1,\cdots,\bar{v}(n)\}}\sup\limits_{x\in I_{k}}
\frac{1}{nh} \sum_{i=1}^{n-1}\left[
\left|H_{i}(x)-H_{i}(x_{k})\right|+\left|\mathbb{E}[H_{i}(x)]-\mathbb{E}[H_{i}(x_{k})] \right|\right]   \\ %%% 2
&~~+\max\limits_{k\in\{1,\cdots,\bar{v}(n)\}}
\left|\frac{1}{nh} \sum_{i=1}^{n-1}\left[
H_{i}(x_{k})-\mathbb{E}[H_{i}(x_{k})] \right]\right|        \\ %%% 3
&:=R_{1}+R_{2} +R_{3}.
\end{align*}
For $R_{1}$, we have,
\begin{align*}
R_{1}&\leq \sup\limits_{|x|>\tilde{b}_{n}}
\frac{1}{nh} \sum_{i=1}^{n-1}
\left|H_{i}(x)\mathbb{I}_{\{|X_{i\Delta}|>\frac{\tilde{b}_{n}}{2}\}}
-\mathbb{E}\left[H_{i}(x)\mathbb{I}_{\{|X_{i\Delta}|>\frac{\tilde{b}_{n}}{2}\}}\right] \right|   \\
&~~+ \sup\limits_{|x|>\tilde{b}_{n}}
\frac{1}{nh} \sum_{i=1}^{n-1}
\left|H_{i}(x)\mathbb{I}_{\{|X_{i\Delta}|\leq\frac{\tilde{b}_{n}}{2}\}}
-\mathbb{E}\left[H_{i}(x)\mathbb{I}_{\{|X_{i\Delta}|\leq\frac{\tilde{b}_{n}}{2}\}}\right] \right|    \\
&:=R_{11}+R_{12}.
\end{align*}
Note that $(X_{i\Delta}-x)/h\geq \tilde{b}_{n}/(2h)$, if $|x|>\tilde{b}_{n}$ and $|X_{i\Delta}|\leq \tilde{b}_{n}/2$. Then, $K\left((X_{i\Delta}-x)/h\right)=0$, if $\tilde{b}_{n}/(2h)\geq \bar{c}_{K}$ with sufficiently large $n$. Therefore, the convergence rate of $R_{1}$ is denominated by $R_{11}$.
Applying the mean-value theorem, we have,
\begin{align*}
&\sup\limits_{|x|>\tilde{b}_{n}}
\frac{1}{nh} \sum_{i=1}^{n-1}
\left|H_{i}(x)\mathbb{I}_{\{|X_{i\Delta}|>\frac{\tilde{b}_{n}}{2}\}} \right|   \\
&\leq  \frac{\breve{K}}{nh} \sum_{i=1}^{n-1}\mathbb{I}_{\{|y-x|\leq \bar{c}_{K}h\}}
|\sigma^{2}(y)-\sigma^{2}(x)|\mathbb{I}_{\{|X_{i\Delta}|>\frac{\tilde{b}_{n}}{2}\}}       \\
&\leq  \frac{\breve{K}}{nh} \sum_{i=1}^{n-1}\sup\limits_{|y-x|\leq \bar{c}_{K}h}
\left|(\sigma^{2})'\left(y+\bar{\lambda}(x-y)\right)(y-x)\right|\frac{|X_{i\Delta}|^{2+q}}{(\tilde{b}_{n}/2)^{2+q}}      \\
&=O_{P}\left(1/(\tilde{b}_{n}^{2+q})\right),
\end{align*}
where $\bar{\lambda}\in[0,1]$.
The last equality follows from the boundedness of $(\sigma^{2})'$ and the moment condition of $X_t$.
Similarly, we could get
$$\sup\limits_{|x|>\tilde{b}_{n}}
\frac{1}{nh} \sum_{i=1}^{n-1}
\left|\mathbb{E}\left[H_{i}(x)\mathbb{I}_{\{|X_{i\Delta}|>\frac{\tilde{b}_{n}}{2}\}}\right] \right|=O\left(1/(\tilde{b}_{n}^{2+q})\right).  $$
In conclusion, $R_{1}=O_{P}\left(1/(\tilde{b}_{n}^{2+q})\right)$.
For $R_{2}$, define an event $D_{n,i}(x,x_k)$ in $\Omega$ as follows for each $(n,i,x,x_k)$,
\begin{align*}
D_{n,i}(x,x_k):=\{\max\{|X_{i\Delta}-x|,|X_{i\Delta}-x_k|\}\leq\bar{c}_{K}h+\tilde{r}_{n}  \}.
\end{align*}
Note that for any $x\in I_{k}$, $|x-x_{k}|\leq \tilde{r}_{n}$, we can get,
\begin{align*}
&\{|X_{i\Delta}-x|\leq \bar{c}_{K}h\}
\subset D_{n,i}(x,x_k),     \ \ \text{and},
\{|X_{i\Delta}-x_k|\leq\bar{c}_{K}h\}\subset D_{n,i}(x,x_k).
\end{align*}
For any $x\in I_{k}$,
\begin{align*}
&|H_{i}(x)-H_{i}(x_{k})|     \\
&\leq \sigma^{2}(X_{i\Delta})\left| K\left(\frac{X_{i\Delta}-x}{h}\right)-K\left(\frac{X_{i\Delta}-x_k}{h}\right)\right|
+K\left(\frac{X_{i\Delta}-x}{h}\right)\mathbb{I}_{D_{n,i}(x,x_k)}|\sigma^{2}(x)-\sigma^{2}(x_k)|    \\
&~~+\sigma^{2}(x_k)\mathbb{I}_{D_{n,i}(x,x_k) }
\left| K\left(\frac{X_{i\Delta}-x}{h}\right)-K\left(\frac{X_{i\Delta}-x_k}{h}\right)\right|     \\
&\leq \tilde{K}|x_{k}-x|/h
+\breve{K}\mathbb{I}_{D_{n,i}(x,x_k)}|(\sigma^{2})'(x+\lambda(x_{k}-x))||x-x_k|
+\tilde{K}\mathbb{I}_{D_{n,i}(x,x_k)}|x-x_k|/h      \\
&\leq (\tilde{K}_{1}\tilde{r}_{n})/h,
\end{align*}
where $\lambda\in[0,1]$, $\tilde{K}$, $\tilde{K}_{1}>0$. The first inequality uses the triangle inequality, the second inequality holds with the boundedness of $K$, $\sigma^{2}$, and Lipschitz property of $K$.
Therefore, $R_{2}=O_{P}(\tilde{r}_{n}/h^{2})$.
Let
\begin{align*}
\tilde{b}_{n}=\left(\frac{nh}{\log n}\right)^{1/(2(2+q))},\quad \tilde{r}_{n}=\sqrt{\frac{h^{3}\log n}{n}},
\end{align*}
we can get $R_{1}$ and $R_{2}$ are all of order $\sqrt{(\log n)/(nh)}$.
For $R_{3}$,
define
\begin{align*}
\Sigma_{H,m}^{2}&:=\mathbb{E}\left[\left|\sum_{i=1}^{m}[
H_{i}(x_{k})-\mathbb{E}[H_{i}(x_{k})]]\right|^{2}\right],\ \text{where} \
m\leq n-1.
\end{align*}
In the following, we will show that if $\breve{\beta}>1$, it can be obtained that,
\begin{align}
\label{prooflemmafefvariance}
\Sigma_{H,m}^{2}\leq\bar{C}_{2}m\Delta^{-1}h^{4-\frac{2}{\breve{\beta}}},
\end{align}
where $\bar{C}_{2}$ is some constant.  
Note that,
\begin{align}
\Sigma_{H,m}^{2}
&=\mathbb{E}\left[\sum_{i=1}^{m}
\left[H_{i}(x_{k})-\mathbb{E}[H_{i}(x_{k})]\right]^{2}\right]
+2\sum_{1\leq i< j\leq m}\mathbb{E}\left[
\left[H_{i}(x_{k})-\mathbb{E}[H_{i}(x_{k})]\right]
\left[H_{j}(x_{k})-\mathbb{E}[H_{j}(x_{k})]\right]\right]       \nonumber\\
&\leq \sum_{i=1}^{m}\mathbb{E}\left[H_{i}^{2}(x_{k})\right]
-\sum_{i=1}^{m}\left[\mathbb{E}[H_{i}(x_{k})]\right]^{2}
+2\sum_{1\leq i< j\leq m}\left[\mathbb{E}\left[H_{i}(x_{k})H_{j}(x_{k})\right]
-\mathbb{E}\left[H_{i}(x_{k})\right]\mathbb{E}\left[H_{j}(x_{k})\right]\right]      \nonumber\\
&\leq \sum_{i=1}^{m}\mathbb{E}\left[H_{i}^{2}(x_{k})\right]
+2\sum_{1\leq i< j\leq m}\mathbb{E}\left[|H_{i}(x_{k})||H_{j}(x_{k})|\right]    \nonumber\\
&\leq \bar{C}mh^{3}+2\sum_{1\leq i< j\leq m}\mathbb{E}\left[|H_{i}(x_{k})||H_{j}(x_{k})|\right], \label{prooflemmafeffirst}
\end{align}
where $\bar{C}>0$. The last inequality uses change-of variable and the boundedness of $(\sigma^{2})'$.
To establish the bound for the second term in (\ref{prooflemmafeffirst}), we bound $\mathbb{E}\left[|H_{1}(x_{k})||H_{j}(x_{k})|\right]$ using two approaches.
At first, for $j>1$,
\begin{align*}
&\mathbb{E}\left[|H_{1}(x_{k})||H_{j}(x_{k})|\right]        \\
&=\int_{-\infty}^{\infty}\int_{-\infty}^{\infty}h^{2}
K(u)|\sigma^{2}(uh+x_{k})-\sigma^{2}(x_{k})|K(v)|\sigma^{2}(vh+x_{k})-\sigma^{2}(x_{k})|
\pi_{1,j}(uh+x_{k},vh+x_{k})dudv    \\
&\leq \bar{C}_{1}h^{4},
\end{align*}
where $\bar{C}_{1}>0$.
The last inequality holds with the boundedness of $\pi_{1,j}$ and $(\sigma^{2})'$. Second, it is easy to obtain that $|H_{i}(x_{k})|\leq C_{H} h$ for some positive constant $C_{H}$, where $i\geq 1$. Using the Billingsley inequality (Corollary 1.1 in \cite{Bosq1998Nonparametric}), we have,
\begin{align*}
&\mathbb{E}\left[|H_{1}(x_{k})||H_{j}(x_{k})|\right]
\leq 4\alpha(j\Delta)\Vert  H_{1}(x_{k}) \Vert_{\infty}^{2}
\leq 4A C_{H}^{2}j^{-\breve{\beta}}\Delta^{-\breve{\beta}} h^{2},\quad j>1.
\end{align*}
Therefore,
\begin{align}
2\sum_{1\leq i< j\leq m}\mathbb{E}\left[|H_{i}(x_{k})||H_{j}(x_{k})|\right]
&\leq 2m\left[ \sum_{1< j\leq \Delta^{-1}h^{-2/\breve{\beta}}}\bar{C}_{1}h^{4}
+\sum_{j>\Delta^{-1}h^{-2/\breve{\beta}} }4A C_{H}^{2}j^{-\breve{\beta}}\Delta^{-\breve{\beta}} h^{2}\right] \nonumber  \\
&\leq 2\bar{C}_{1}m\Delta^{-1}h^{-2/\breve{\beta}}h^{4}
+8A C_{H}^{2}m\Delta^{-1}h^{-2/\breve{\beta}}h^{4} /(\breve{\beta}-1)    \nonumber  \\
&= \bar{C}_{2}m\Delta^{-1}h^{4-(2/\breve{\beta})},\label{prooflemmafefsecond}
\end{align}
where $\bar{C}_{2}=2\bar{C}_{1}+8A C_{H}^{2}/(\breve{\beta}-1)$. It is worth mentioning that $\breve{\beta}>1$, and $\Delta^{-1}h^{-2/\breve{\beta}}\geq 2$ for small $\Delta$ and $h$. The second inequality holds with the application of
\begin{align*}
\sum_{j>\Delta^{-1}h^{-2/\breve{\beta}} }j^{-\breve{\beta}}\leq
\int_{\Delta^{-1}h^{-2/\breve{\beta}}}^{\infty} x^{-\breve{\beta}}dx
=\frac{1}{\breve{\beta}-1}\Delta^{\breve{\beta}-1}h^{2-2/\breve{\beta}}.
\end{align*}
Combined (\ref{prooflemmafeffirst}) and (\ref{prooflemmafefsecond}), we could get (\ref{prooflemmafefvariance}). Therefore, for large enough $a$, we have,
\begin{align*}
&P [R_{3}\geq a\sqrt{(\log n)/(nh)}]        \\
&\leq \sum_{k=1}^{\bar{v}(n)}P \left[\left|\sum_{i=1}^{n}\left[
H_{i}(x_{k})-\mathbb{E}[H_{i}(x_{k})]\right]\right|\geq anh\sqrt{(\log n)/(nh)}\right]        \\
&\leq \bar{v}(n)\Bigg[
4\exp\left(-\frac{(a^{2}n^{2}h^{2}\log n)/(nh)}{64\bar{C}_{2}n\Delta^{-1}h^{4-2/\breve{\beta}}
+(8/3)am nh\sqrt{(\log n)/(nh)} C_H h }\right)
+4Anm^{-1-\breve{\beta}}\Delta^{-\breve{\beta}}\Bigg]      \\
&\leq \frac{\tilde{b}_{n}}{\tilde{r}_{n}}\Bigg[
4\exp\left(-\frac{a^{2}\log n}{64\bar{C}_{2}\Delta^{-1}h^{3-2/\breve{\beta}}
+(8/3)aC_H }\right)
+4An\left(\frac{n^{2-\theta}}{h\log n}\right)^{(-1-\breve{\beta})/2}\Delta^{-\breve{\beta}}\Bigg]      \\
&\leq \frac{4\tilde{b}_{n}}{\tilde{r}_{n}}n^{-\frac{a^{2}}{64\bar{C}_{2}+(8/3)aC_H}}
+4A\frac{\tilde{b}_{n}}{\tilde{r}_{n}}\left(\frac{n}{h\log n}\right)^{(-1-\breve{\beta})/2}n\Delta^{-\breve{\beta}},
\end{align*}
where the second inequality is because of the exponential inequality of strong mixing arrays, the third inequality holds with
\begin{align*}
m=\frac{\sqrt{n}}{\sqrt{h(\log n)}},\quad
\left(m\leq \max\left\{n-1,\frac{anh\sqrt{\log n}}{\sqrt{nh}4C_H h}\right\}~~\text{for~large}~ a\right).
\end{align*}
The last inequality holds with
\begin{align*}
\Delta^{-1}h^{3-(2/\breve{\beta})}
=(\log n)^{3-2/\breve{\beta}}n^{\kappa-(3-2/\breve{\beta})\theta}\leq 1,
\end{align*}
as $\Delta, h\to 0$ if $\breve{\beta}>1$.
Note that,
\begin{align*}
\frac{\tilde{b}_{n}}{\tilde{r}_{n}}=(\log n)^{-\frac{1}{2(2+q)}-\frac{1}{2}}h^{\frac{1}{2(2+q)}-\frac{3}{2}}
n^{\frac{1}{2(2+q)}+\frac{1}{2}},
\end{align*}
then, the second term of the last inequality is of the order
\begin{align*}
(\log n)^{-\frac{1}{2(2+q)}+\frac{\breve{\beta}}{2}}
h^{\frac{1}{2(2+q)}+\frac{\breve{\beta}}{2}-1}
n^{\frac{1}{2(2+q)}+1-\frac{\breve{\beta}}{2}+\breve{\beta}\kappa}.
\end{align*}
If
$\breve{\beta}>\frac{2+\frac{1}{2+q}}{1-2\kappa}$, we have
$R_{3}=O_{P}(\sqrt{(\log n)/(nh)})$.
The proof for the second part is analogous to that of the first part, for simplicity, we omit it here.
\end{proof}
%%%%%%%%%%%%%%%%%%%%%%%%%%%%%%%%%%%%%%%%%%%%%%%%%%%%%%%% lemma 6
\begin{lemma}
\label{lemmafefmu}
Assume that Assumptions $\ref{assumption-stationary and mixing}$ and $\ref{assumption-kernel and damping}$ hold.
$\sup_{x}|\pi(x)|<\infty$. $\sup_{x,y}|\pi_{t,t+s}(x,y)|<\infty$. As $\Delta, h \to 0$ and $T\rightarrow \infty$, let
$\Delta^{-1}=O(T^{\bar{\kappa}})$ for some $\bar{\kappa}>0$, and $(\log T)/(T^{\bar{\theta}}h) \to 0$, where $\bar{\theta}\in(0,1)$. \\
$(1)$ If $\breve{\beta}>\frac{(2+\frac{1}{2+q})\bar{\theta}}{1-\bar{\theta}}$,
then as $\Delta, h \rightarrow 0$ and $n, T \rightarrow \infty$,
\begin{align*}
&\sup\limits_{x\in \mathbb{R}}
\left|\frac{1}{n} \sum_{i=1}^{n-1}\left[
K_{h}\left(X_{i\Delta}-x\right)
-\mathbb{E}\left[K_{h}\left(X_{i\Delta}-x\right) \right] \right]\right|
=O_{P}(\sqrt{(\log T)/(T^{\bar{\theta}}h)}).
\end{align*}
$(2)$ If $b $ is uniformly bounded and $\breve{\beta}>\frac{(2+\frac{1}{2+q})\bar{\theta}}{2-\bar{\theta}}$,
then as $\Delta, h \rightarrow 0$ and $n, T \rightarrow \infty$,
\begin{align*}
&\sup\limits_{x\in \mathbb{R}}
\left|\frac{1}{n} \sum_{i=1}^{n-1}\left[
K_{h}\left(X_{i\Delta}-x\right)\left(b (X_{i\Delta})-b (x)\right)
-\mathbb{E}\left[K_{h}\left(X_{i\Delta}-x\right)\left(b (X_{i\Delta})-b (x)\right) \right] \right]\right|
=O_{P}(\sqrt{(\log T)/(T^{\bar{\theta}}h)}).
\end{align*}
\end{lemma}
%%%%%%%%%%%%%%%%%%%%%%%%%%%%%%%%%%%%%%%%%%%%%%%%%%%%%%%% lemma 4
\begin{lemma}
\label{lemmafsfi}
Under the conditions stated in Lemmas $\ref{lemmafydw}$ and $\ref{lemmafef}$, suppose that
$|(\sigma^{2})''(x)|=O(|x|)$ as $|x|\to\infty$.
Then, as $\Delta, h \rightarrow 0$ and $n, T \rightarrow \infty$, we can obtain,
\begin{align*}
&\sup\limits_{x\in \mathbb{R}}
\bigg|\frac{1}{T} \sum_{i=1}^{n-1} K_{h}\left(X_{i\Delta}-x\right)
\int_{i\Delta}^{(i+1)\Delta} \left( \sigma^{2}(X_{t})-\sigma^{2}(X_{i\Delta})\right) dt \bigg|
=O_{P}(\Delta h^{-\frac{1}{q+1}}+\sqrt{(\log n)^{3}/(nh)}).
\end{align*}
\end{lemma}
\begin{proof}
Applying It$\mathrm{\hat{o}}$ formula, we can get,
\begin{align*}
\sigma^{2}(X_{(i+1)\Delta})-\sigma^{2}(X_{i\Delta})
=\int_{i\Delta}^{(i+1)\Delta}\mathcal{A}\sigma(X_{t})dt
+\int_{i\Delta}^{(i+1)\Delta}(\sigma (\sigma^{2})')(X_{t}) dW_{t},
\end{align*}
where $\mathcal{A}\sigma=(\sigma^{2})'b +\frac{1}{2}(\sigma^{2})''\sigma^{2}$. Then, we have,
\begin{align*}
&\sup\limits_{x\in \mathbb{R}}
\bigg|\frac{1}{T} \sum_{i=1}^{n-1} K_{h}\left(X_{i\Delta}-x\right)
\int_{i\Delta}^{(i+1)\Delta}  \left( \sigma^{2}(X_{t})-\sigma^{2}(X_{i\Delta})\right)  dt \bigg|      \\
&\leq\sup\limits_{x\in \mathbb{R}}
\bigg|\frac{1}{T} \sum_{i=1}^{n-1} K_{h}\left(X_{i\Delta}-x\right)
\int_{i\Delta}^{(i+1)\Delta}\int_{i\Delta}^{t}\mathcal{A}\sigma(X_{s})dsdt \bigg|    \\
&~~+\sup\limits_{x\in \mathbb{R}}
\bigg|\frac{1}{T} \sum_{i=1}^{n-1} K_{h}\left(X_{i\Delta}-x\right)
\int_{i\Delta}^{(i+1)\Delta}\int_{i\Delta}^{t}(\sigma (\sigma^{2})')(X_{s}) dW_{s}dt \bigg|  \\
&:=\mathcal{Q}_{1}+\mathcal{Q}_{2}.
\end{align*}
Note that
\begin{align}
&\sup\limits_{x\in \mathbb{R}}
\left|\frac{1}{n} \sum_{i=1}^{n-1} K_{h}\left(X_{i\Delta}-x\right)\right|         \nonumber   \\
&\leq \sup\limits_{x\in \mathbb{R}}\left|\frac{1}{n} \sum_{i=1}^{n-1}
\left[K_{h}\left(X_{i\Delta}-x\right)-\mathbb{E}\left[K_{h}\left(X_{i\Delta}-x\right)\right]  \right]   \right|
+\sup\limits_{x\in \mathbb{R}}\left|\mathbb{E}\left[K_{h}\left(X_{i\Delta}-x\right)\right]     \right|  \label{prooflemmaK}  \\
&=O_{P}(1),     \nonumber
\end{align}
where the first term in (\ref{prooflemmaK}) is $o_{P}(1)$, as verified in Lemma \ref{lemmafef}.
The second term is $O_{P}(1)$, which follows from the boundedness of $\pi$.
Then, we have,
\begin{align*}
\mathcal{Q}_{1}&\leq \sup\limits_{x\in \mathbb{R}}
\bigg|\frac{\Delta}{T} \sum_{i=1}^{n-1} K_{h}\left(X_{i\Delta}-x\right)
\int_{i\Delta}^{(i+1)\Delta}|\mathcal{A}\sigma(X_{s})|ds\bigg|       \\
&\leq\sup\limits_{x\in \mathbb{R}}
\bigg|\frac{\Delta}{T} \sum_{i=1}^{n-1} K_{h}\left(X_{i\Delta}-x\right)
\int_{i\Delta}^{(i+1)\Delta}|\mathcal{A}\sigma(X_{s})|\mathbb{I}_{\{|X_{s}|\leq \bar{\phi}_{n}\}}ds\bigg|       \\
&~~+\sup\limits_{x\in \mathbb{R}}
\bigg|\frac{\Delta}{T} \sum_{i=1}^{n-1} K_{h}\left(X_{i\Delta}-x\right)
\int_{i\Delta}^{(i+1)\Delta}|\mathcal{A}\sigma(X_{s})|\mathbb{I}_{\{|X_{s}|> \bar{\phi}_{n}\}}ds\bigg| \\
&=O_{P}(\Delta\bar{\phi}_{n})+O_{P}(\Delta/h\bar{\phi}_{n}^{q}),
\end{align*}
where $\bar{\phi}_{n}$ is a sequence of positive real numbers.
Let $\bar{\phi}_{n}=h^{-\frac{1}{q+1}}$,  $\mathcal{Q}_{1}=O_{P}(\Delta h^{-\frac{1}{q+1}})$.
Applying the method of proving Lemma $\ref{lemmafydw}$, we can get $\mathcal{Q}_{2}=O_{P}(\sqrt{(\log n)^{3}/(nh)})$.
\end{proof}

%%%%%%%%%%%%%%%%%%%%%%%%%%%%%%%%%%%%%%%%%%%%%%%%%%%%%%%% lemma 5
\begin{lemma}
\label{lemmafyds}
Suppose that $(\ref{betaconditiondiff})$ and the conditions $(1)$-$(3)$ in Lemma $\ref{lemmafef}$ hold.
$|b ''(x)|=O(|x|)$ as $|x|\to\infty$.
Then, as $n, T \rightarrow \infty$ and $\Delta, h \rightarrow 0$,
\begin{align*}
&\sup\limits_{x\in \mathbb{R}}
\left|\frac{1}{T\Delta} \sum_{i=1}^{n-1} K_{h}\left(X_{i\Delta}-x\right)
\left[\int_{i\Delta}^{(i+1)\Delta}
b (X_{s})(Y_{s}-Y_{i\Delta})ds -\frac{\Delta^{2}}{2}X_{i\Delta}b (X_{i\Delta})\right] \right|     \\
&=O_{P}(\Delta h^{-\frac{1}{q+1}}+\sqrt{(\log n)^{3}/(nh)}).
\end{align*}
\end{lemma}
\begin{proof}
Applying It$\mathrm{\hat{o}}$ formula, we have,
\begin{align*}
b (X_{s})(Y_{s}-Y_{i\Delta})
&=\int_{i\Delta}^{s}X_{u}b (X_{u})du
+\int_{i\Delta}^{s}(b b ')(X_{u})(Y_{u}-Y_{i\Delta})du    \\
&~~+\frac{1}{2}\int_{i\Delta}^{s}(\sigma^{2}b '')(X_{u})(Y_{u}-Y_{i\Delta})du
+\int_{i\Delta}^{s}(\sigma b ')(X_{u})(Y_{u}-Y_{i\Delta})dW_{u},
\end{align*}
then,
\begin{align*}
&\sup\limits_{x\in \mathbb{R}}
\left|\frac{1}{T\Delta} \sum_{i=1}^{n-1} K_{h}\left(X_{i\Delta}-x\right)
\left[\int_{i\Delta}^{(i+1)\Delta}
 b (X_{s})(Y_{s}-Y_{i\Delta})ds -\frac{\Delta^{2}}{2}X_{i\Delta} b (X_{i\Delta})\right] \right|     \\
&\leq \sup\limits_{x\in \mathbb{R}}
\left|\frac{1}{T\Delta} \sum_{i=1}^{n-1} K_{h}\left(X_{i\Delta}-x\right)
\int_{i\Delta}^{(i+1)\Delta}\int_{i\Delta}^{s}\left(X_{u} b (X_{u})-X_{i\Delta} b (X_{i\Delta})\right)du  ds  \right|    \\
&~~+\sup\limits_{x\in \mathbb{R}}
\left|\frac{1}{T\Delta} \sum_{i=1}^{n-1} K_{h}\left(X_{i\Delta}-x\right)
\int_{i\Delta}^{(i+1)\Delta}
\int_{i\Delta}^{s}( b  b ')(X_{u})(Y_{u}-Y_{i\Delta})duds  \right|    \\
&~~+\sup\limits_{x\in \mathbb{R}}
\left|\frac{1}{2T\Delta} \sum_{i=1}^{n-1} K_{h}\left(X_{i\Delta}-x\right)
\int_{i\Delta}^{(i+1)\Delta}
\int_{i\Delta}^{s}(\sigma^{2} b '')(X_{u})(Y_{u}-Y_{i\Delta})du ds  \right|    \\
&~~+\sup\limits_{x\in \mathbb{R}}
\left|\frac{1}{T\Delta} \sum_{i=1}^{n-1} K_{h}\left(X_{i\Delta}-x\right)
\int_{i\Delta}^{(i+1)\Delta}
\int_{i\Delta}^{s}(\sigma b ')(X_{u})(Y_{u}-Y_{i\Delta})dW_{u}ds  \right|    \\
&:=\bar{\mathcal{Q}}_{1}+\bar{\mathcal{Q}}_{2}+\bar{\mathcal{Q}}_{3}+\bar{\mathcal{Q}}_{4}.
\end{align*}
Employing the method in Lemma $\ref{lemmafsfi}$ , we have $\bar{\mathcal{Q}}_{1}=O_{P}(\Delta h^{-\frac{1}{q+1}})+O_{P}(\sqrt{(\log n)^{3}/(nh)})$. Denote $\bar{f}= b  b '$, $\bar{\mathcal{Q}}_{2}$ can be bounded by
\begin{align*}
\bar{\mathcal{Q}}_{2}
&\leq \sup\limits_{x\in \mathbb{R}}
\left|\frac{1}{T} \sum_{i=1}^{n-1} K_{h}\left(X_{i\Delta}-x\right)
\int_{i\Delta}^{(i+1)\Delta}|\bar{f}(X_{t})|dt
\int_{i\Delta}^{(i+1)\Delta}|X_{t}|dt   \right|     \\
&\leq \sup\limits_{x\in \mathbb{R}}
\left|\frac{1}{T} \sum_{i=1}^{n-1} K_{h}\left(X_{i\Delta}-x\right)
\int_{i\Delta}^{(i+1)\Delta}C(1+|X_{t}|)dt
\int_{i\Delta}^{(i+1)\Delta}|X_{t}|dt   \right|     \\
&\leq \sup\limits_{x\in \mathbb{R}}
\left|\frac{\tilde{C}}{T} \sum_{i=1}^{n-1} K_{h}\left(X_{i\Delta}-x\right)
\left(\int_{i\Delta}^{(i+1)\Delta}|X_{t}|dt \right) ^{2}  \right| \\
&\leq \sup\limits_{x\in \mathbb{R}}
\left|\frac{\tilde{C}\Delta}{T} \sum_{i=1}^{n-1} K_{h}\left(X_{i\Delta}-x\right)
\int_{i\Delta}^{(i+1)\Delta}|X_{t}|^{2}dt \right|    \\
&\leq \sup\limits_{x\in \mathbb{R}}
\left|\frac{\tilde{C}\Delta}{T} \sum_{i=1}^{n-1} K_{h}\left(X_{i\Delta}-x\right)
\int_{i\Delta}^{(i+1)\Delta}|X_{t}|^{2}\mathbb{I}_{\{|X_{t}|^{2}\leq \tilde{\phi}_{n}\}} dt \right|    \\
&~~+ \sup\limits_{x\in \mathbb{R}}
\left|\frac{\tilde{C}\Delta}{T} \sum_{i=1}^{n-1} K_{h}\left(X_{i\Delta}-x\right)
\int_{i\Delta}^{(i+1)\Delta}|X_{t}|^{2}\mathbb{I}_{\{|X_{t}|^{2}> \tilde{\phi}_{n}\}} dt \right| \\
&=O_{P}(\Delta h^{-\frac{1}{q+1}}),
\end{align*}
where $\tilde{\phi}_{n}$ is a sequence of positive real numbers. $\tilde{C},C$ are some constants.
The proof of the last equality is similar to $\mathcal{Q}_{1}$ in Lemma $\ref{lemmafsfi}$, we omit it here.
Moreover, it is easy to obtain that $\bar{\mathcal{Q}}_{3}=O_{P}(\Delta h^{-\frac{1}{q+1}}) $,
and $\bar{\mathcal{Q}}_{4}=O_{P}(\sqrt{(\log n)^{3}/(nh)})$.
\end{proof}

%%%%%%%%%%%%%%%%%%%%%%%%%%%%%%%%%%%%%%%%%
\section{Proofs of the Main Theorems}
\begin{proof}[\textbf{Proof of Theorem \ref{thm-diffusion}}]
Denote
\begin{flalign}
&\hat{\sigma}^{2}(x)=\frac{ \frac{1}{T} \sum_{i=1}^{n-1} K_{h}(\breve{X}_{(i-1)\Delta}-x)
\frac{3}{2}
(\breve{X}_{(i+1)\Delta}-\breve{X}_{i\Delta})^{2}  }
{ \frac{\Delta}{T} \sum_{i=1}^{n-1} K_{h}(\breve{X}_{(i-1)\Delta}-x)}
:=\frac{C_{n}(x)}{\breve{\Pi}(x)}.
\end{flalign}
At first, we will show that
\begin{equation}
\label{prooftheoremsigmaup}
\sup\limits_{x\in \mathbb{R}}\left|C_{n}(x)-\breve{\Pi}(x)\sigma^{2}(x)\right| =O_{P}(a_{n,T}).
\end{equation}
Observe that,
\begin{flalign*}
&\Delta^{2}(\breve{X}_{(i+1) \Delta}-\breve{X}_{i \Delta})^{2}
=(Y_{(i+1)\Delta}-Y_{i\Delta})^{2}
+(Y_{i\Delta}-Y_{(i-1)\Delta})^{2}-2(Y_{(i+1)\Delta}-Y_{i\Delta})(Y_{i\Delta}-Y_{(i-1)\Delta}).
\end{flalign*}
Then,
\begin{flalign*}
&\sup\limits_{x\in \mathbb{R}}
\left|C_{n}(x)-\breve{\Pi}(x)\sigma^{2}(x)\right| \leq \bar{H}_{1}+\bar{H}_{2}+\bar{H}_{3}+\bar{H}_{4},
\end{flalign*}
where
\begin{flalign*}
\bar{H}_{1}&:=\sup\limits_{x\in \mathbb{R}}
\left| \sum_{i=1}^{n-1}
\left(K_{h}(\breve{X}_{(i-1)\Delta}-x)
-K_{h}(X_{i\Delta}-x) \right) \left[ \frac{3}{2T\Delta^{2}} \left[(Y_{(i+1)\Delta}-Y_{i\Delta})^{2}-\Delta^{2}X_{i\Delta}^{2}
\right. \right.\right. \\
&~~\left.  \left.\left.  +2\Delta X_{i\Delta}(Y_{i\Delta}-Y_{(i-1)\Delta})
-2(Y_{(i+1)\Delta}-Y_{i\Delta})(Y_{i\Delta}-Y_{(i-1)\Delta})  \right]
-\frac{1}{2n} \sigma^{2}(x)\right]  \right|,     \\  %%%% C1
\bar{H}_{2}&:=\sup\limits_{x\in \mathbb{R}}
\left| \sum_{i=1}^{n-1}
K_{h}(X_{i\Delta}-x) \left[ \frac{3}{2T\Delta^{2}} \left[(Y_{(i+1)\Delta}-Y_{i\Delta})^{2}-\Delta^{2}X_{i\Delta}^{2}
\right. \right.\right. \\
&~~\left.  \left.\left.  +2\Delta X_{i\Delta}(Y_{i\Delta}-Y_{(i-1)\Delta})
-2(Y_{(i+1)\Delta}-Y_{i\Delta})(Y_{i\Delta}-Y_{(i-1)\Delta})  \right]
-\frac{1}{2n} \sigma^{2}(x)\right]  \right| ,    \\ %%%% C2
\bar{H}_{3}&:=\sup\limits_{x\in \mathbb{R}}
\left| \sum_{i=1}^{n-1}
\left(K_{h}(\breve{X}_{(i-1)\Delta}-x)
-K_{h}(X_{(i-1)\Delta}-x) \right)
\left[\frac{3}{2T\Delta^{2}}(Y_{i\Delta}-Y_{(i-1)\Delta}-\Delta X_{i \Delta})^{2}
-\frac{1}{2n} \sigma^{2}(x)\right]  \right|,        \\   %%%% C3
\bar{H}_{4}&:=\sup\limits_{x\in \mathbb{R}}
\left|\sum_{i=1}^{n-1}
K_{h}(X_{(i-1)\Delta}-x)\left[\frac{3}{2T\Delta^{2}}(Y_{i\Delta}-Y_{(i-1)\Delta}-\Delta X_{i \Delta})^{2}
-\frac{1}{2n} \sigma^{2}(x)\right] \right| .          %%%% C4
\end{flalign*}

We analyze $\bar{H}_{2}$ at first,
applying It$\mathrm{\hat{o}}$ formula, we have,
\begin{flalign*}
&(Y_{(i+1)\Delta}-Y_{i\Delta})^{2} - \Delta^{2} X_{i\Delta}^{2}    \\
&=\int_{i\Delta}^{(i+1)\Delta}\int_{i\Delta}^{t}\int_{i\Delta}^{s}
\left(4X_{u} b (X_{u})+2\sigma^{2}(X_{u}) \right)dudsdt
+\int_{i\Delta}^{(i+1)\Delta}\int_{i\Delta}^{t}
2 b (X_{s})(Y_{s}-Y_{i\Delta})dsdt  \\
&~~+\int_{i\Delta}^{(i+1)\Delta}\int_{i\Delta}^{t}\int_{i\Delta}^{s}
4X_{u}\sigma(X_{u})dW_{u}dsdt
+\int_{i\Delta}^{(i+1)\Delta}\int_{i\Delta}^{t}
2\sigma(X_{s})(Y_{s}-Y_{i\Delta})dW_{s}dt ,
\end{flalign*}
and,
\begin{flalign*}
&2\left(Y_{(i+1)\Delta}-Y_{i\Delta}\right)
\left(Y_{i\Delta}-Y_{(i-1)\Delta}\right)-2\Delta X_{i\Delta}\left(Y_{i\Delta}-Y_{(i-1)\Delta}\right)\\
&=\int_{i\Delta}^{(i+1)\Delta}\int_{i\Delta}^{t}
2 b (X_{s}) dsdt
(Y_{i\Delta}-Y_{(i-1)\Delta})
+\int_{i\Delta}^{(i+1)\Delta}\int_{i\Delta}^{t}
2\sigma(X_{s}) (Y_{i\Delta}-Y_{(i-1)\Delta})  dW_{s}dt.
\end{flalign*}
Therefore,
\begin{flalign*}
\bar{H}_{2}\leq \bar{H}_{21}+\bar{H}_{22}+\bar{H}_{23}+\bar{H}_{24},
\end{flalign*}
where
\begin{flalign*}
\bar{H}_{21}&:=\sup\limits_{x\in \mathbb{R}}
\left|\frac{3}{2T\Delta^{2}} \sum_{i=1}^{n-1}
K_{h}(X_{i\Delta}-x)\left[
\left(\int_{i\Delta}^{(i+1)\Delta}\int_{i\Delta}^{t}\int_{i\Delta}^{s}
4X_{u} b (X_{u})dudsdt -\frac{2\Delta^{3}}{3}X_{i\Delta} b (X_{i\Delta})\right) \right.  \right.    \\
&~~\left.\left.
+\left(\int_{i\Delta}^{(i+1)\Delta}\int_{i\Delta}^{t}
2 b (X_{s})(Y_{s}-Y_{i\Delta})dsdt-\frac{\Delta^{3}}{3}X_{i\Delta} b (X_{i\Delta})\right) \right.\right. \\
&~~\left.\left.-\left(\int_{i\Delta}^{(i+1)\Delta}\int_{i\Delta}^{t}
2 b (X_{s})(Y_{i\Delta}-Y_{(i-1)\Delta})dsdt-\Delta^{3}X_{i\Delta} b (X_{i\Delta})\right) \right.\right. \\
&~~\left.\left.+\left(\int_{i\Delta}^{(i+1)\Delta}\int_{i\Delta}^{t}\int_{i\Delta}^{s}
2\sigma^{2}(X_{u})dudsdt-\frac{\Delta^{3}}{3}\sigma^{2}(X_{i\Delta})\right) \right] \right|,     \\ %% H21
\bar{H}_{22}&:=\sup\limits_{x\in \mathbb{R}}
\left|\frac{3}{2T\Delta^{2}} \sum_{i=1}^{n-1}
K_{h}(X_{i\Delta}-x)\left[
\int_{i\Delta}^{(i+1)\Delta}\int_{i\Delta}^{t}\int_{i\Delta}^{s}
4(X_{u}\sigma(X_{u})-X_{i\Delta}\sigma(X_{i\Delta})) dW_{u}dsdt \right.  \right.    \\
&~~\left.\left.+\int_{i\Delta}^{(i+1)\Delta}\int_{i\Delta}^{t}
2(\sigma(X_{s})(Y_{s}-Y_{i\Delta})-(s-i\Delta)X_{i\Delta}\sigma(X_{i\Delta}))  dW_{s} dt \right.  \right. \\
&~~\left.\left.-\int_{i\Delta}^{(i+1)\Delta}\int_{i\Delta}^{t}
2(\sigma(X_{s})(Y_{i\Delta}-Y_{(i-1)\Delta})-
\Delta X_{i\Delta}\sigma(X_{i\Delta}) )  dW_{s} dt \right] \right|,  \\ %% H22
\bar{H}_{23}&:=\sup\limits_{x\in \mathbb{R}}
\left|\frac{\Delta}{2T} \sum_{i=1}^{n-1}
\left[K_{h}(X_{i\Delta}-x)
\left(\sigma^{2}(X_{i\Delta})-\sigma^{2}(x)\right)
-\mathbb{E}\left[K_{h}(X_{i\Delta}-x)\left(\sigma^{2}(X_{i\Delta})-\sigma^{2}(x)\right)\right]\right]\right|, \\ %% H23
\bar{H}_{24}&:=\sup\limits_{x\in \mathbb{R}}
\left|\frac{\Delta}{2T} \sum_{i=1}^{n-1}
\mathbb{E}\left[K_{h}(X_{i\Delta}-x)\left(\sigma^{2}(X_{i\Delta})-\sigma^{2}(x)\right)\right]\right|.  %% H24
\end{flalign*}
Applying Lemma $\ref{lemmafsfi}$ and \ref{lemmafyds}, we can get $\bar{H}_{21}=O_{P}(\Delta h^{-\frac{1}{q+1}})+O_{P}(\sqrt{(\log n)^{3}/(nh)})$. With Lemma \ref{lemmafydw}, we have $\bar{H}_{22}=O_{P}(\sqrt{(\log n)^{3}/(nh)})$. Moreover, $\bar{H}_{23}=O_{P}(\sqrt{(\log n)^{3}/(nh)})$, which follows from Lemma \ref{lemmafef}.
$\bar{H}_{24}=O(h^{2})$ by simple calculation.
For $\bar{H}_{4}$, note that
\begin{flalign*}
&2 \Delta X_{i\Delta}(Y_{i\Delta}-Y_{(i-1)\Delta})   \\
&=2\Delta^{2} X_{(i-1)\Delta}^{2}
+\Delta\int_{(i-1)\Delta}^{i\Delta}\int_{(i-1)\Delta}^{t}
\left(4X_{s} b (X_{s})+2\sigma^{2}(X_{s})\right)dsdt
+\Delta\int_{(i-1)\Delta}^{i\Delta}\int_{(i-1)\Delta}^{t}
4X_{s}\sigma(X_{s})  dW_{s}dt           \\
&~~+\Delta\int_{(i-1)\Delta}^{i\Delta}
2 b (X_{t})(Y_{t}-Y_{(i-1)\Delta})dt
+\Delta\int_{(i-1)\Delta}^{i\Delta}
2\sigma(X_{t})(Y_{t}-Y_{(i-1)\Delta})dW_{t}.
\end{flalign*}
Furthermore, it is easy to obtain
\begin{flalign*}
&(Y_{i\Delta}-Y_{(i-1)\Delta}-\Delta X_{i\Delta})^{2}     \\
&=\left(\Delta^{2} \int_{(i-1)i\Delta}^{i\Delta}2X_{t} b (X_{t})dt-2\Delta^{3}X_{(i-1)\Delta} b (X_{(i-1)\Delta})\right)   \\
&~~+\left(\int_{(i-1)i\Delta}^{i\Delta}\int_{(i-1)\Delta}^{t}\int_{(i-1)\Delta}^{s}
4X_{u} b (X_{u})dudsdt -\frac{2\Delta^{3}}{3}X_{(i-1)\Delta} b (X_{(i-1)\Delta})\right) \\
&~~+\left(\int_{(i-1)\Delta}^{i\Delta}\int_{(i-1)\Delta}^{t}
2 b (X_{s})(Y_{s}-Y_{(i-1)\Delta})dsdt -\frac{\Delta^{3}}{3}X_{(i-1)\Delta} b (X_{(i-1)\Delta})  \right)   \\
&~~-\left(\Delta\int_{(i-1)\Delta}^{i\Delta}\int_{(i-1)\Delta}^{t}
4X_{s} b (X_{s})dsdt -2\Delta^{3}X_{(i-1)\Delta} b (X_{(i-1)\Delta}) \right) \\
&~~-\left(\Delta\int_{(i-1)\Delta}^{i\Delta}
2 b (X_{t})(Y_{t}-Y_{(i-1)\Delta})dt- \Delta^{3}X_{(i-1)\Delta} b (X_{(i-1)\Delta}) \right)  \\
&~~+\left(\Delta^{2}
\int_{(i-1)i\Delta}^{i\Delta}\sigma^{2}(X_{t})dt -\Delta^{3} \sigma^{2}(X_{(i-1)\Delta})\right)
-\left(\Delta\int_{(i-1)\Delta}^{i\Delta}\int_{(i-1)\Delta}^{t}
2\sigma^{2}(X_{s})dsdt-\Delta^{3} \sigma^{2}(X_{(i-1)\Delta}) \right)      \\
&~~+\left(\int_{(i-1)i\Delta}^{i\Delta}\int_{(i-1)\Delta}^{t}\int_{(i-1)\Delta}^{s}
2\sigma^{2}(X_{u})dudsdt -\frac{\Delta^{3}}{3}\sigma^{2}(X_{(i-1)\Delta}) \right)   \\
&~~+\Delta^{2} \int_{(i-1)i\Delta}^{i\Delta}
2\left(X_{t}\sigma(X_{t})-X_{(i-1)\Delta}\sigma(X_{(i-1)\Delta})\right) dW_{t}      \\
&~~+\int_{(i-1)\Delta}^{i\Delta}\int_{(i-1)\Delta}^{t}\int_{(i-1)\Delta}^{s}
4\left(X_{u}\sigma(X_{u})-X_{(i-1)\Delta}\sigma(X_{(i-1)\Delta})\right)dW_{u}dsdt      \\
&~~+\int_{(i-1)\Delta}^{i\Delta}\int_{(i-1)\Delta}^{t}
2\left(\sigma(X_{s})(Y_{s}-Y_{(i-1)\Delta})- (s-(i-1)\Delta)X_{(i-1)\Delta}\sigma(X_{(i-1)\Delta})\right) dW_{s}dt   \\
&~~-\Delta\int_{(i-1)\Delta}^{i\Delta}\int_{(i-1)\Delta}^{t}
4\left(X_{s}\sigma(X_{s})-X_{(i-1)\Delta}\sigma(X_{(i-1)\Delta})\right)  dW_{s}dt       \\
&~~-\Delta\int_{(i-1)\Delta}^{i\Delta}
2\left(\sigma(X_{t})(Y_{t}-Y_{(i-1)\Delta})-(t-(i-1)\Delta)X_{(i-1)\Delta}\sigma(X_{(i-1)\Delta})\right) dW_{t}.
\end{flalign*}
Similar as the proof of $\bar{H}_{2}$, we have
$\bar{H}_{4}=O_{P}(\Delta h^{-\frac{1}{q+1}}+\sqrt{(\log n)^{3}/(nh)}+h^{2})$.
For $\bar{H}_{1}$ and $\bar{H}_{3}$,
using mean-value theorem, we have,
\begin{flalign*}
&K_{h}(\breve{X}_{(i-1)\Delta}-x)-K_{h}(X_{i\Delta}-x)
=\frac{1}{h^{2}}
K'\left(\frac{X_{i\Delta}-x}{h}+\frac{\bar{\lambda}(\breve{X}_{(i-1)\Delta}-X_{i\Delta})}{h}\right)
\left(\breve{X}_{(i-1)\Delta}-X_{i\Delta}\right) ,
\end{flalign*}
where $0\leq\bar{\lambda}\leq 1$.  According to the conditions on the kernel function, we can construct a function $K^{*}(u)=K_{1}\mathbb{I}_{\{|u|\leq \bar{c}_{K}+\bar{\upsilon}\}}$ such that
$\sup_{|\upsilon|\leq \bar{\upsilon}}|K'(u+\upsilon)|\leq K^{*}(u)$, where $K_{1}>0$, $\bar{\upsilon}>0$, $\upsilon$  are some constants. Then,
\begin{flalign*}
\left|K'\left(\frac{X_{i\Delta}-x}{h}+\frac{\bar{\lambda}(\breve{X}_{(i-1)\Delta}-X_{i\Delta})}{h}\right)\right|
\mathbb{I}_{\{|(\bar{\lambda}(\breve{X}_{(i-1)\Delta}-X_{i\Delta}))/h|\leq\bar{\upsilon}\}}
\leq K^{*}\left(\frac{X_{i\Delta}-x}{h}\right).
\end{flalign*}
Note that $(\breve{X}_{(i-1)\Delta}-X_{i\Delta})/h=O_{a.s.}(\Delta^{\frac{1}{2}-\frac{1}{2+q}}/h)$.
Following from the proof of $\bar{H}_{2}$ and $\bar{H}_{4}$, we can get $\bar{H}_{1}$ and $\bar{H}_{3}$ are all of the order of $O_{P}(a_{n,T})$.
For the denominator $\breve{\Pi}(x)$ of $\hat{\sigma}^{2}(x)$, applying Lemma \ref{lemmafef}, it is easy to obtain,
\begin{flalign*}
&\sup\limits_{x\in \mathbb{R}}
\left|\breve{\Pi}(x)-\pi(x)\right|    \\
&\leq\sup\limits_{x\in \mathbb{R}}
\left|\frac{1}{n} \sum_{i=1}^{n-1} \left[
K_{h}\left(X_{(i-1)\Delta}-x\right)-\mathbb{E}\left[K_{h}\left(X_{(i-1)\Delta}-x\right)\right]\right]\right|
+\sup\limits_{x\in \mathbb{R}}
\left|\mathbb{E}\left[K_{h}\left(X_{(i-1)\Delta}-x\right)\right]-\pi(x)\right|   \\
&~~+\sup\limits_{x\in \mathbb{R}}
\left|\left[\frac{1}{n}  \sum_{i=1}^{n-1}\left( K_{h}\left(\breve{X}_{(i-1)\Delta}-x\right)
- K_{h}\left(X_{(i-1)\Delta}-x\right)\right)\right]\right|      \\
&=O_{P}(\sqrt{(\log n)/(n^{\theta}h)})+O_{P}(h^{2})+O_{P}(\Delta^{\frac{1}{2}-\frac{1}{2+q}}/h).
\end{flalign*}
Therefore,
\begin{flalign}
&\sup\limits_{|x|\leq \bar{b}_{n,T}}\left|\frac{\breve{\Pi}(x)}{\pi(x)}-1\right|
=\sup\limits_{|x|\leq \bar{b}_{n,T}}\left|\frac{\breve{\Pi}(x)-\pi(x)}{\pi(x)}\right|
\leq \frac{O_{P}(\sqrt{(\log n)/(n^{\theta}h)}+h^{2}+\Delta^{\frac{1}{2}-\frac{1}{2+q}}/h)}{\delta_{n,T}} .  \label{prooftheorempi}
\end{flalign}
Combining (\ref{prooftheoremsigmaup}) and (\ref{prooftheorempi}), we have,
\begin{flalign*}
\hat{\sigma}^{2}(x)-\sigma^{2}(x)
=\frac{(C_{n}(x)-\breve{\Pi}(x)\sigma^{2}(x))/\pi(x)}{\breve{\Pi}(x)/\pi(x)}
=\frac{O_{P}(a_{n,T}/\delta_{n,T})}{1+o_{P}(1)}
=O_{P}\left(\frac{a_{n,T}}{\delta_{n,T}}\right),
\end{flalign*}
uniformly over $|x|\leq \bar{b}_{n,T}$.
\end{proof}

\begin{proof}[\textbf{Proof of Theorem \ref{thm-drift}}]
Denote
\begin{flalign*}
\hat{ b }(x)=\frac{ \frac{1}{T} \sum_{i=1}^{n-1} K_{h}(\breve{X}_{(i-1)\Delta}-x)
(\breve{X}_{(i+1)\Delta}-\breve{X}_{i\Delta})  }
{ \frac{\Delta}{T} \sum_{i=1}^{n-1} K_{h}(\breve{X}_{(i-1)\Delta}-x)  }
:=\frac{A_{n}(x)}{\breve{\Pi}(x)}.
\end{flalign*}
To get the result, we need to show that
\begin{equation*}
\sup\limits_{x\in \mathbb{R}}|A_{n}(x)-\breve{\Pi}(x) b (x)|=O_{P}(\Delta^{\frac{1}{2}-\frac{1}{2+q}}+h^{2}+\sqrt{(\log T)/(T^{\bar{\theta}}h)}).
\end{equation*}
It is easy to obtain
\begin{equation*}
\sup\limits_{x\in \mathbb{R}}|A_{n}(x)-\breve{\Pi}(x) b (x)|\leq \tilde{A}_{1}+\tilde{A}_{2},
\end{equation*}
where
\begin{flalign*}
\tilde{A}_{1}&:=\sup\limits_{x\in \mathbb{R}}\left|
\frac{1}{T} \sum_{i=1}^{n-1}K_{h}(X_{(i-1)\Delta}-x)\left( (\breve{X}_{(i+1)\Delta}-\breve{X}_{i\Delta})-\Delta b (x)\right) \right|,                    \\
\tilde{A}_{2}&:=\sup\limits_{x\in \mathbb{R}}\left|
\frac{1}{T} \sum_{i=1}^{n-1}\left( K_{h}(\breve{X}_{(i-1)\Delta}-x)-K_{h}(X_{(i-1)\Delta}-x) \right)
\left( (\breve{X}_{(i+1)\Delta}-\breve{X}_{i\Delta})-\Delta b (x)\right)  \right|.
\end{flalign*}

We analyze $\tilde{A}_{1}$ in the first step.
Using It\^o formula, we have,
\begin{flalign*}
&\breve{X}_{(i+1) \Delta}-\breve{X}_{i \Delta}  \\
&=\frac{1}{\Delta}\int_{i\Delta}^{(i+1)\Delta}(X_{t}-X_{i\Delta})dt
-\frac{1}{\Delta}\int_{(i-1)\Delta}^{i\Delta}(X_{t}-X_{(i-1)\Delta})dt
+(X_{i\Delta}-X_{(i-1)\Delta})      \\
&=\int_{(i-1)\Delta}^{i\Delta} b (X_{t})dt
+\frac{1}{\Delta}  \int_{i\Delta}^{(i+1)\Delta}\int_{i\Delta}^{t} b (X_{s})dsdt
-\frac{1}{\Delta} \int_{(i-1)\Delta}^{i\Delta}\int_{(i-1)\Delta}^{t} b (X_{s})dsdt \\
&~~+\int_{(i-1)\Delta}^{i\Delta}\sigma(X_{t})dW_{t}
+\frac{1}{\Delta}  \int_{i\Delta}^{(i+1)\Delta}\int_{i\Delta}^{t}\sigma(X_{s})dW_{s}dt
-\frac{1}{\Delta} \int_{(i-1)\Delta}^{i\Delta}\int_{(i-1)\Delta}^{t}\sigma(X_{s})dW_{s}dt.
\end{flalign*}
Then, we can get,
\begin{flalign*}
\tilde{A}_{1}\leq \tilde{A}_{11}+\tilde{A}_{12}+\tilde{A}_{13}+\tilde{A}_{14},
\end{flalign*}
where
\begin{flalign*}
\tilde{A}_{11}&:=\sup\limits_{x\in \mathbb{R}}\left|
\frac{1}{T} \sum_{i=1}^{n-1}K_{h}(X_{(i-1)\Delta}-x)
\left(\int_{(i-1)\Delta}^{i\Delta} b (X_{t})dt-\Delta  b (X_{(i-1)\Delta})\right)\right.\\
&~~\left. +\left(\frac{1}{\Delta}  \int_{i\Delta}^{(i+1)\Delta}\int_{i\Delta}^{t} b (X_{s})dsdt
-\frac{\Delta}{2}  b (X_{i\Delta})\right)
-\left(\frac{1}{\Delta} \int_{(i-1)\Delta}^{i\Delta}\int_{(i-1)\Delta}^{t} b (X_{s})dsdt
-\frac{\Delta}{2}  b (X_{i\Delta})\right)\right|,           \\ %%% A11
\tilde{A}_{12}&:=\sup\limits_{x\in \mathbb{R}}\left|
\frac{1}{T} \sum_{i=1}^{n-1}K_{h}(X_{(i-1)\Delta}-x)
\int_{(i-1)\Delta}^{i\Delta}\sigma(X_{t})dW_{t} \right.\\
&~~\left.+\frac{1}{\Delta}  \int_{i\Delta}^{(i+1)\Delta}
\int_{i\Delta}^{t}\sigma(X_{s})dW_{s}dt
-\frac{1}{\Delta} \int_{(i-1)\Delta}^{i\Delta}\int_{(i-1)\Delta}^{t}
\sigma(X_{s})dW_{s}dt\right|,           \\  %%% A12
\tilde{A}_{13}&:=\sup\limits_{x\in \mathbb{R}}\left|
\frac{\Delta}{T} \sum_{i=1}^{n-1}
\left[K_{h}(X_{(i-1)\Delta}-x)( b (X_{(i-1)\Delta})- b (x) )
-\mathbb{E}[K_{h}(X_{(i-1)\Delta}-x)
( b (X_{(i-1)\Delta})- b (x) )]\right] \right|,    \\   %%% A13
\tilde{A}_{14}&:=\sup\limits_{x\in \mathbb{R}}\left|
\frac{\Delta}{T} \sum_{i=1}^{n-1}\mathbb{E}\left[K_{h}(X_{(i-1)\Delta}-x)
\left( b (X_{(i-1)\Delta})- b (x) \right)\right]\right|. %%% A14
\end{flalign*}
Following from the lipschitz condition of $ b $,
$\tilde{A}_{11}$ can be bounded by,
\begin{flalign*}
\frac{\tilde{C}_{ b }}{n}
\sup\limits_{x\in \mathbb{R}}\left|
\sum_{i=1}^{n-1}K_{h}(X_{(i-1)\Delta}-x)\right|\times
\sup_{|s-t|\in[0,\Delta],s,t\in[0,\infty)}|X_{s}-X_{t}|
=O_{P}(\Delta^{\frac{1}{2}-\frac{1}{2+q}}),
\end{flalign*}
where $\tilde{C}_{ b }>0$.
Applying Lemma \ref{lemmafdwtmu} and \ref{lemmafefmu},  we could get $\tilde{A}_{12}=O_{P}(\sqrt{(\log T)/(T^{\bar{\theta}}h)})$,
and $\tilde{A}_{13}=O_{P}(\sqrt{(\log T)/(T^{\bar{\theta}}h)})$.
$\tilde{A}_{14}=O(h^{2})$ by simple calculation.
Similar to the proof of $\bar{H}_{1}$ in Theorem \ref{thm-diffusion}, we could get $\bar{A}_{2}=O_{P}(a^*_{n,T})$.
$(\ref{conclusionthmdrift})$ can be obtained as $(\ref{conclusionthmdiffusion})$,
we omit it here for simplicity.
\end{proof}
\section{Simulation}
In this section, we investigate the finite-sample performance of the N-W estimators for the diffusion coefficients in two types of diffusion process: the Cox-Ingersoll-Ross (CIR) process and the Ornstein-Uhlenbeck (OU) process. The models are specified as follows.
%%%%%%%%%%%%%%%%%%%%%%%%%%%%%%%%%%%% CIR
\begin{example}[$\textbf{CIR process}$]
\begin{align*}
dY_{t}&=X_{t}dt,        \\
dX_{t}&=\kappa(\theta-X_{t})dt+\sigma\sqrt{X_{t}}dW_{t},
\end{align*}
where the parameters $(\kappa,\theta,\sigma)$ considered are $(0.85837, 0.085711, 0.15660)$, following Chapman and Pearson \cite{chapman2000short}.
\end{example}
%%%%%%%%%%%%%%%%%%%%%%%%%%%%%%%%%%%% OU
\begin{example}[$\textbf{OU process}$]
\begin{align*}
dY_{t}&=X_{t}dt,        \\
dX_{t}&=\kappa(\theta-X_{t})dt+\sigma dW_{t},
\end{align*}
where the parameters are set to $\kappa=0.5$, $\theta = -2.75$, and $\sigma= 0.43$, consistent with Stanton \cite{stanton1997nonparametric}.
\end{example}
For both models, we assume that the process $\{Y_{t}\}$ is observed at discrete time points $t_{i}=i\Delta,i=0,1,\cdots,n$. To assess the performance of the diffusion coefficient estimator, we compare the estimator based on direct observations of $X$ (defined in (\ref{estimatorsigmainitial})) with the estimator constructed from discrete observations of $Y$ (defined in (\ref{estimator-sigma}))). The comparison is conducted under various choices of observation interval $\Delta$ and bandwidth $h$.
For nonparametric estimation, we employ the Epanechnikov kernel
$K(u)=\frac{3}{4}(1-u^{2})\mathbb{I}_{\{|u|\leq 1\}}$.
To evaluate the accuracy of the estimators, we use the maximum absolute error (MAAE), defined as
\begin{align*}
&\text{MAAE}=\frac{1}{L}\sum_{k=1}^{L}\max_{i=1,\cdots,N}\{
|\hat{\sigma}^{2}(x_{i}^{k})-\sigma^{2}(x_{i}^{k})|\},
\end{align*}
where $\{x_{i}\}_{i=1}^{N}$ are equidistant points within the range of $X$, set as $[0.078, 0.09]$ for the CIR process and $[-2.79, -2.7]$ for the OU process. We set  $N=50$.  The Monte Carlo simulations are conducted with $L=1000$ replications, and trajectories are generated using Euler discretization.

%%%%%%%%%%%%%%%%%%%%%%%%%%%%
Tables \ref{resulttableCIR} and \ref{resulttableOU} present the estimation errors, where MAAE$^{1}$ corresponds to the estimator (\ref{estimatorsigmainitial}), and MAAE$^{2}$ corresponds to the estimator (\ref{estimator-sigma}). The results demonstrate that the estimation error decreases as the sample size $n$ increases, confirming the consistency of the estimator. Moreover, the choice of bandwidth significantly affects estimation accuracy, highlighting its role in the estimation accuracy.
Figures \ref{resultestimatorCIR} and \ref{resultestimatorOU} depict the mean estimated diffusion coefficient at different evaluation points, computed from $1000$ sample paths  under different values of $\Delta$ and $h$. The "true" curve represents the true diffusion coefficient function, while the estimated curves labeled "$n = a, 1$" and "$n = a, 2$" represent the performance of estimators (\ref{estimatorsigmainitial}) and (\ref{estimator-sigma}), respectively, for sample sizes
$a=1000,5000,9000$. As $n$ increases, the estimated curves align more closely with the true function.

To compare the simulation results with theoretical convergence rates, Figures \ref{comprisontheosimuCIR} and \ref{comprisontheosimuOU} illustrate the relationship between the theoretical error rate $((\log n)^{3}/n)^{2/5}$ (horizontal axis) and the simulated error MAAE (vertical axis) for different $\Delta$ and $h$. The dashed line represents the MAAE of the estimator (\ref{estimatorsigmainitial}), while the solid line corresponds to estimator (\ref{estimator-sigma}). The nearly linear trend between the theoretical rate and the observed error further validates the theoretical uniform convergence rate.
%%%%%%%%%%%%%%%%%%%%%%%%%%%%%%%%%%%%%
\begin{table}[H]
\renewcommand{\arraystretch}{1.2}
  \centering
\begin{tabular}{cccccccc}
\hline
\multicolumn{1}{c}{$\Delta$}            & h                     & Error                   & n=1000 & n=3000 & n=5000 & n=7000 & n=9000 \\ \hline
\multirow{6}{*}{$0.002$} & \multirow{2}{*}{0.03} & $\text{MAAE}^{1}\times 10^{3}$ & 0.2931 & 0.1538 & 0.1156 & 0.0993 & 0.0903 \\
                                &                       & $\text{MAAE}^{2}\times 10^{3}$ & 1.0745 & 1.0208 & 1.013  & 1.0105 & 1.0083 \\ \cline{2-8}
                                & \multirow{2}{*}{0.12} & $\text{MAAE}^{1}\times 10^{3}$ & 0.5959 & 0.4348 & 0.3597 & 0.3213 & 0.297  \\
                                &                       & $\text{MAAE}^{2}\times 10^{3}$ & 1.1923 & 1.1139 & 1.0961 & 1.0885 & 1.0873 \\ \cline{2-8}
                                & \multirow{2}{*}{0.3}  & $\text{MAAE}^{1}\times 10^{3}$ & 0.6461 & 0.4942 & 0.4156 & 0.3756 & 0.3488 \\
                                &                       & $\text{MAAE}^{2}\times 10^{3}$ & 1.254  & 1.1972 & 1.1875 & 1.181  & 0.1832 \\ \hline
\multirow{6}{*}{$0.004$} & \multirow{2}{*}{0.03} & $\text{MAAE}^{1}\times 10^{3}$ & 0.2112 & 0.1207 & 0.098  & 0.0848 & 0.0782 \\
                                &                       & $\text{MAAE}^{2}\times 10^{3}$ & 0.3036 & 0.2383 & 0.2296 & 0.2272 & 0.2241 \\ \cline{2-8}
                                & \multirow{2}{*}{0.12} & $\text{MAAE}^{1}\times 10^{3}$ & 0.4955 & 0.3448 & 0.2884 & 0.2589 & 0.2385 \\
                                &                       & $\text{MAAE}^{2}\times 10^{3}$ & 0.5613 & 0.4178 & 0.3756 & 0.3572 & 0.3484 \\ \cline{2-8}
                                & \multirow{2}{*}{0.3}  & $\text{MAAE}^{1}\times 10^{3}$ & 0.5573 & 0.3998 & 0.3388 & 0.3069 & 0.2843 \\
                                &                       & $\text{MAAE}^{2}\times 10^{3}$ & 0.6328 & 0.4919 & 0.4518 & 0.4336 & 0.4264 \\ \hline
\multirow{6}{*}{$0.008$} & \multirow{2}{*}{0.03} & $\text{MAAE}^{1}\times 10^{3}$ & 0.1757 & 0.104  & 0.0882 & 0.0801 & 0.0754 \\
                                &                       & $\text{MAAE}^{2}\times 10^{3}$ & 0.1788 & 0.1018 & 0.0801 & 0.0686 & 0.0622 \\ \cline{2-8}
                                & \multirow{2}{*}{0.12} & $\text{MAAE}^{1}\times 10^{3}$ & 0.4022 & 0.272  & 0.2331 & 0.2126 & 0.2016 \\
                                &                       & $\text{MAAE}^{2}\times 10^{3}$ & 0.408  & 0.2756 & 0.2367 & 0.2162 & 0.207  \\ \cline{2-8}
                                & \multirow{2}{*}{0.3}  & $\text{MAAE}^{1}\times 10^{3}$ & 0.4605 & 0.3209 & 0.2762 & 0.2508 & 0.2398 \\
                                &                       & $\text{MAAE}^{2}\times 10^{3}$ & 0.4709 & 0.3341 & 0.2932 & 0.272  & 0.2622 \\ \hline
\multirow{6}{*}{$0.01$}  & \multirow{2}{*}{0.03} & $\text{MAAE}^{1}\times 10^{3}$ & 0.1665 & 0.1037 & 0.0865 & 0.0807 & 0.0761 \\
                                &                       & $\text{MAAE}^{2}\times 10^{3}$ & 0.1647 & 0.0993 & 0.0777 & 0.0684 & 0.0623 \\ \cline{2-8}
                                & \multirow{2}{*}{0.12} & $\text{MAAE}^{1}\times 10^{3}$ & 0.369  & 0.2558 & 0.2187 & 0.2044 & 0.1931 \\
                                &                       & $\text{MAAE}^{2}\times 10^{3}$ & 0.3705 & 0.2601 & 0.2204 & 0.2041 & 0.1918 \\ \cline{2-8}
                                & \multirow{2}{*}{0.3}  & $\text{MAAE}^{1}\times 10^{3}$ & 0.4223 & 0.3018 & 0.2591 & 0.2412 & 0.2273 \\
                                &                       & $\text{MAAE}^{2}\times 10^{3}$ & 0.429  & 0.3134 & 0.2699 & 0.2534 & 0.2393  \\ \hline
\end{tabular}
  \caption{MAAE of the diffusion function estimators for CIR process}
  \label{resulttableCIR}
\end{table}
%%%%%%%%%%%%%%%%%%%%%%%%%%%%%%%%%%%%%%% CIR-2
\begin{figure}[htbp]
  %\centering
  \subfigure{
  \begin{minipage}{0.32\linewidth}
  \centering
    \includegraphics[width=2.3in,height=1.9in]{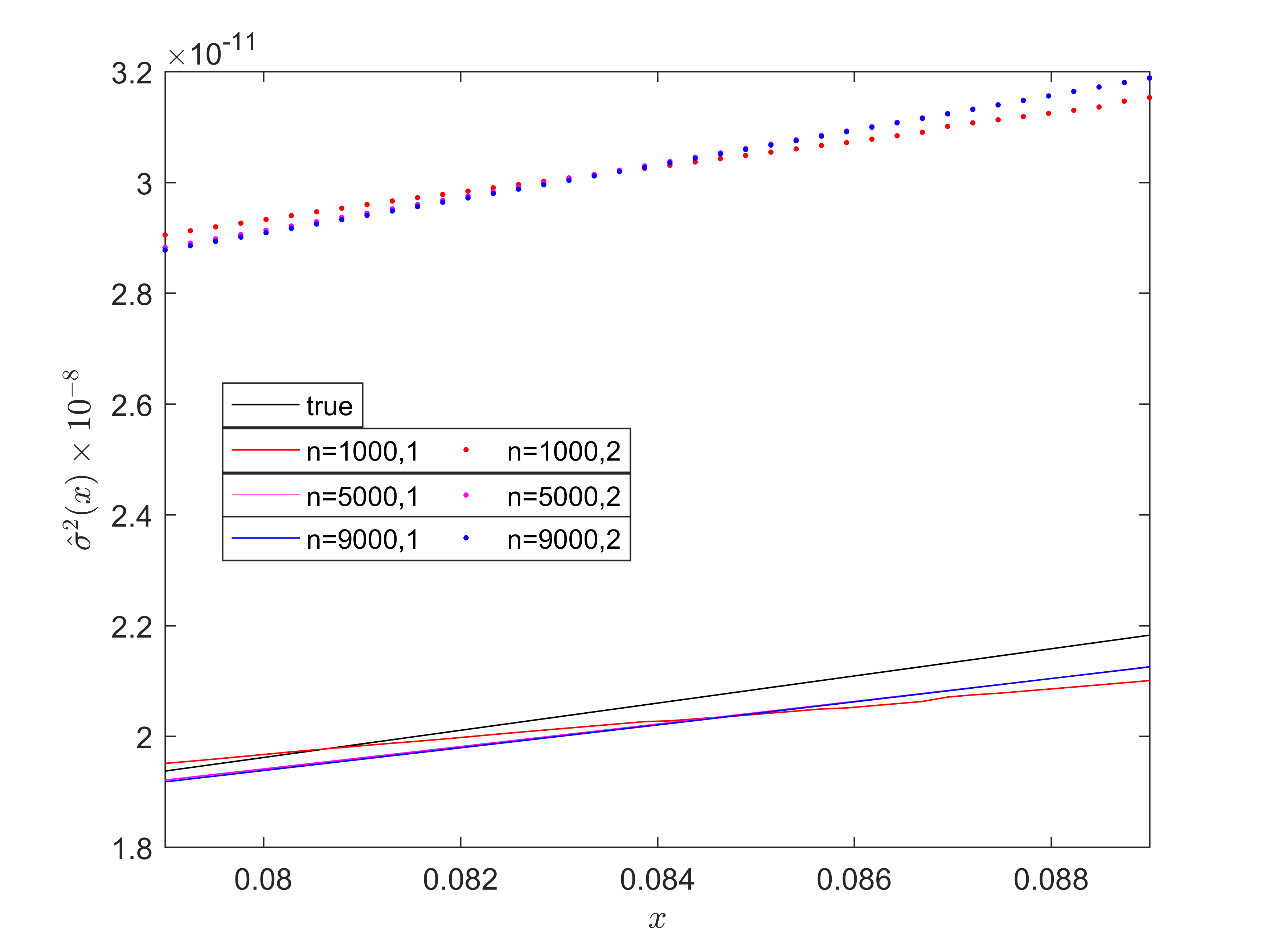}
    \caption*{$\Delta=0.002$, $h=0.03$}
  \end{minipage}
  }
  \subfigure{
    \begin{minipage}{0.32\linewidth}
    \includegraphics[width=2.3in,height=1.9in]{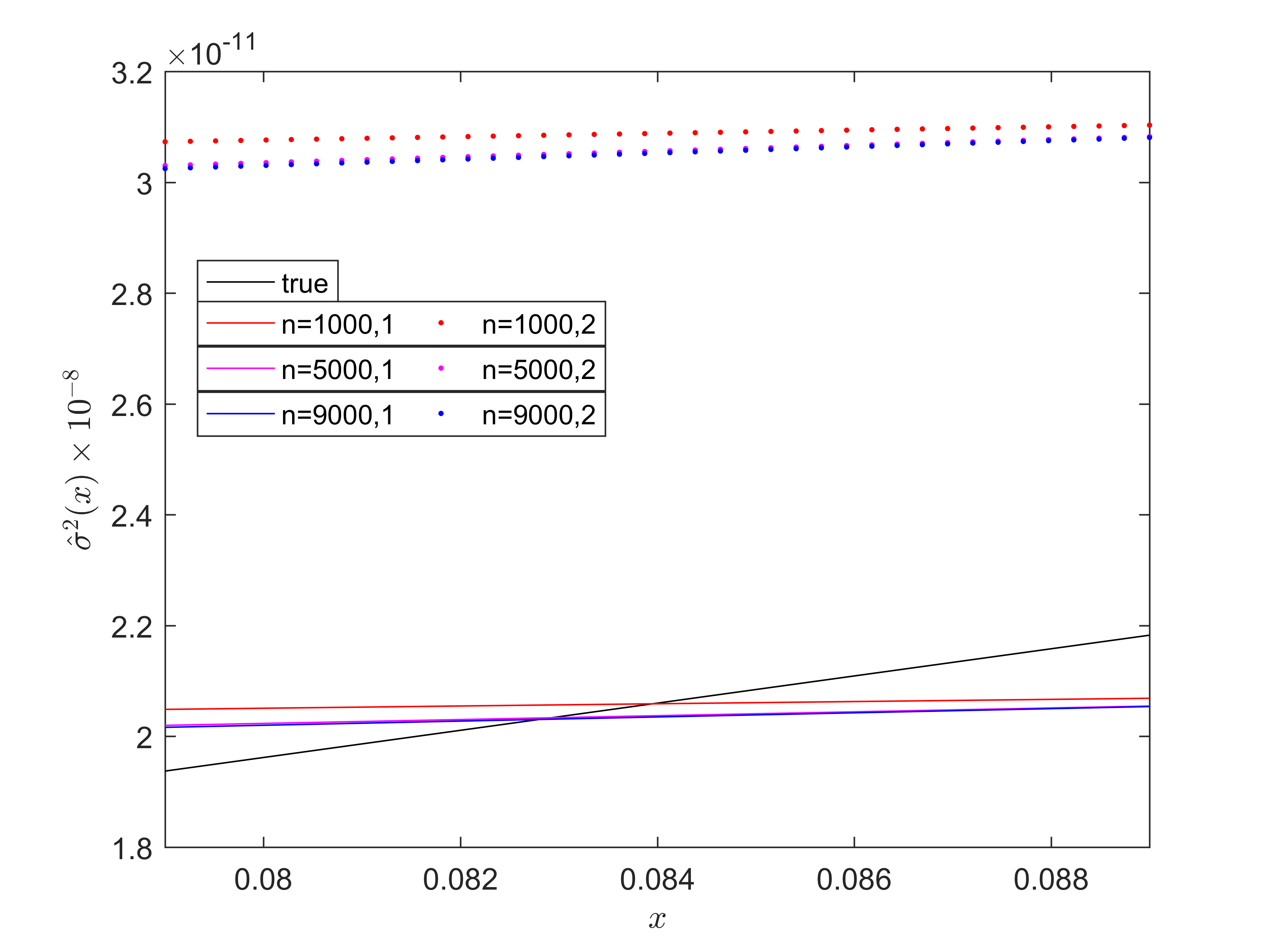}
    \caption*{$\Delta=0.002$, $h=0.12$}
  \end{minipage}
  }
  \subfigure{
    \begin{minipage}{0.32\linewidth}
    \includegraphics[width=2.3in,height=1.9in]{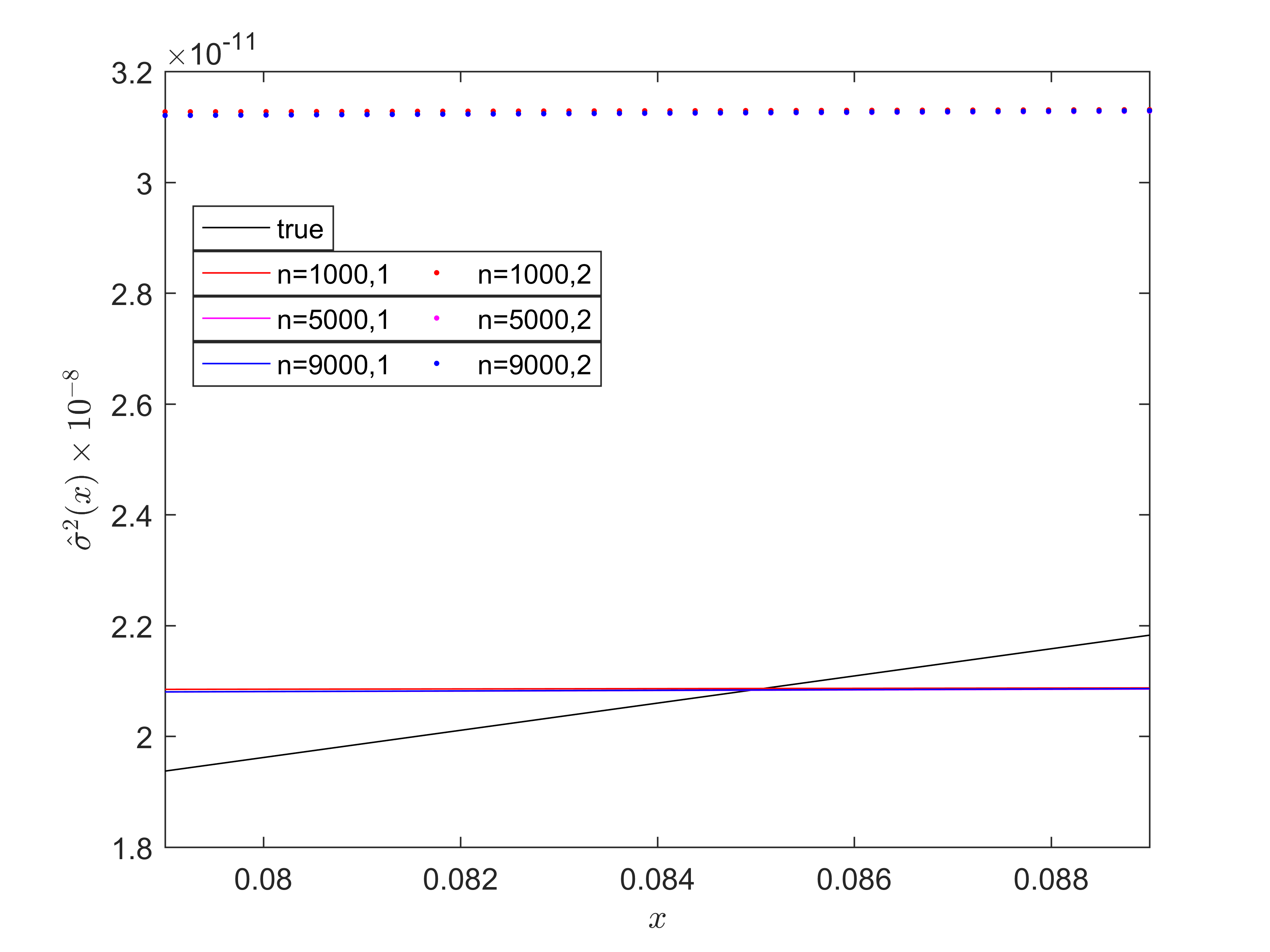}
    \caption*{$\Delta=0.002$, $h=0.3$}
  \end{minipage}
  }

  \subfigure{
  \begin{minipage}{0.32\linewidth}
  \centering
    \includegraphics[width=2.3in,height=1.9in]{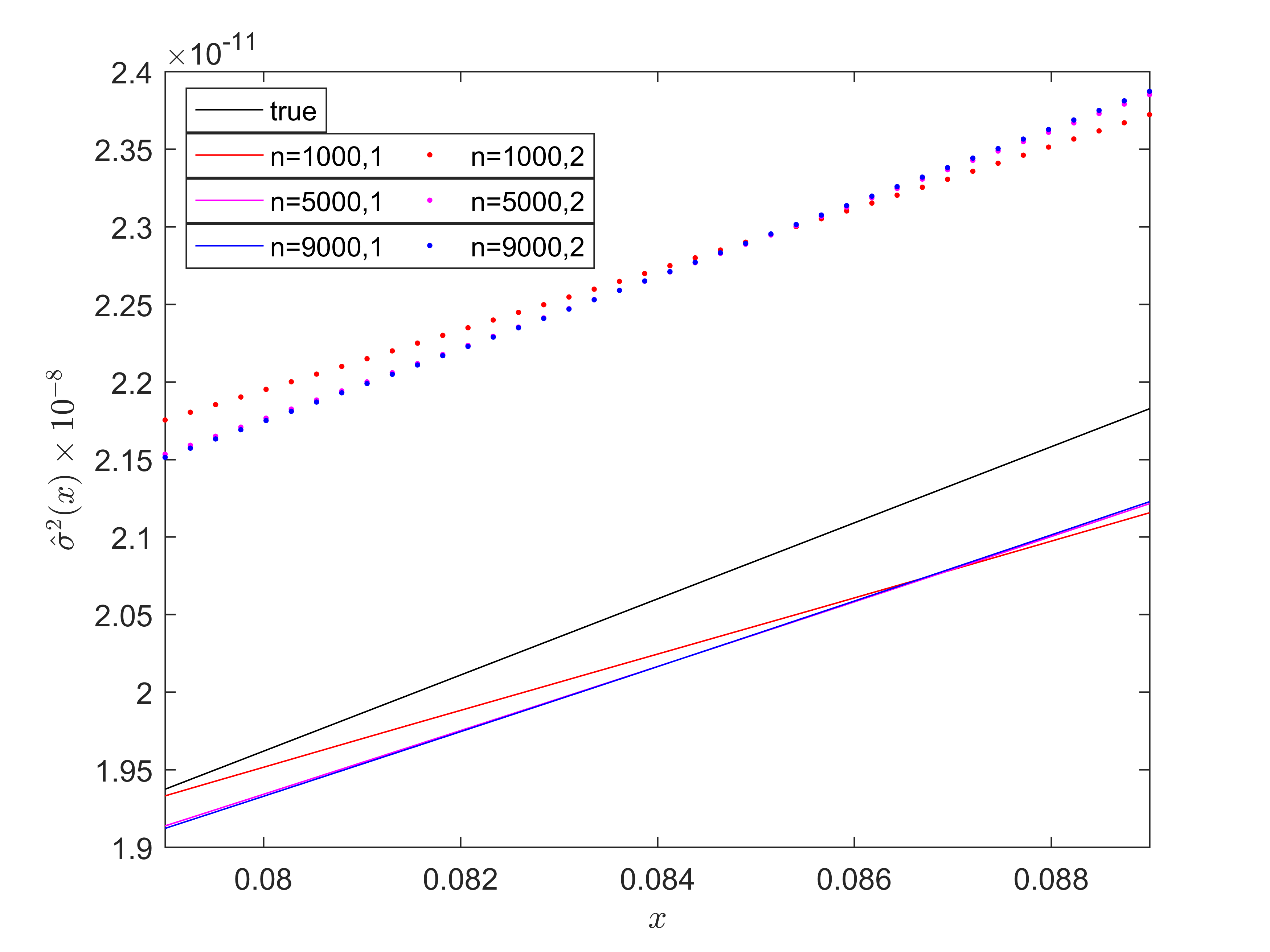}
    \caption*{$\Delta=0.004$, $h=0.03$}
  \end{minipage}
  }
  \subfigure{
    \begin{minipage}{0.32\linewidth}
    \includegraphics[width=2.3in,height=1.9in]{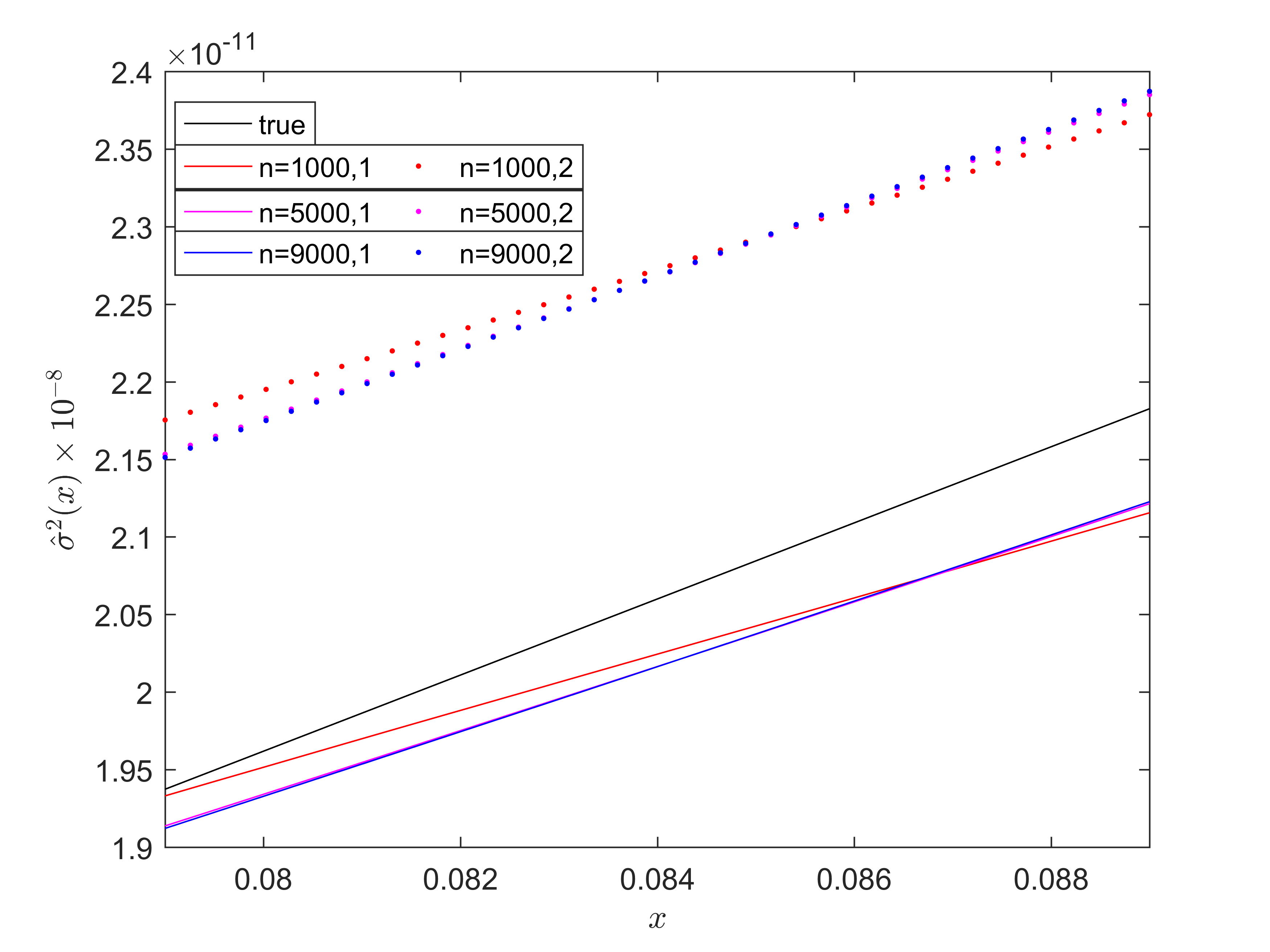}
    \caption*{$\Delta=0.004$, $h=0.12$}
  \end{minipage}
  }
  \subfigure{
    \begin{minipage}{0.32\linewidth}
    \includegraphics[width=2.3in,height=1.9in]{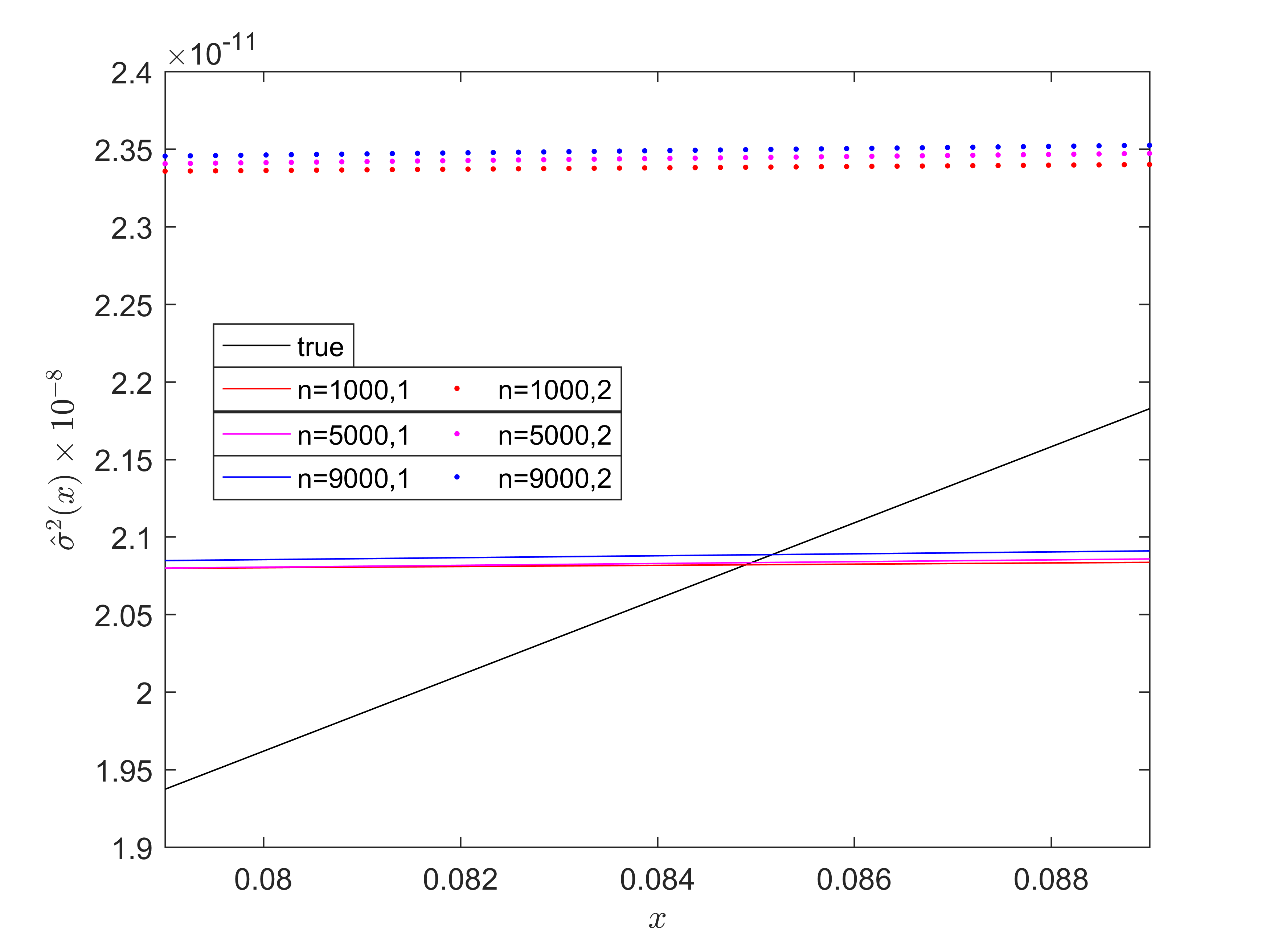}
    \caption*{$\Delta=0.004$, $h=0.3$}
  \end{minipage}
  }

  \subfigure{
    \begin{minipage}{0.32\linewidth}
  \centering
    \includegraphics[width=2.3in,height=1.9in]{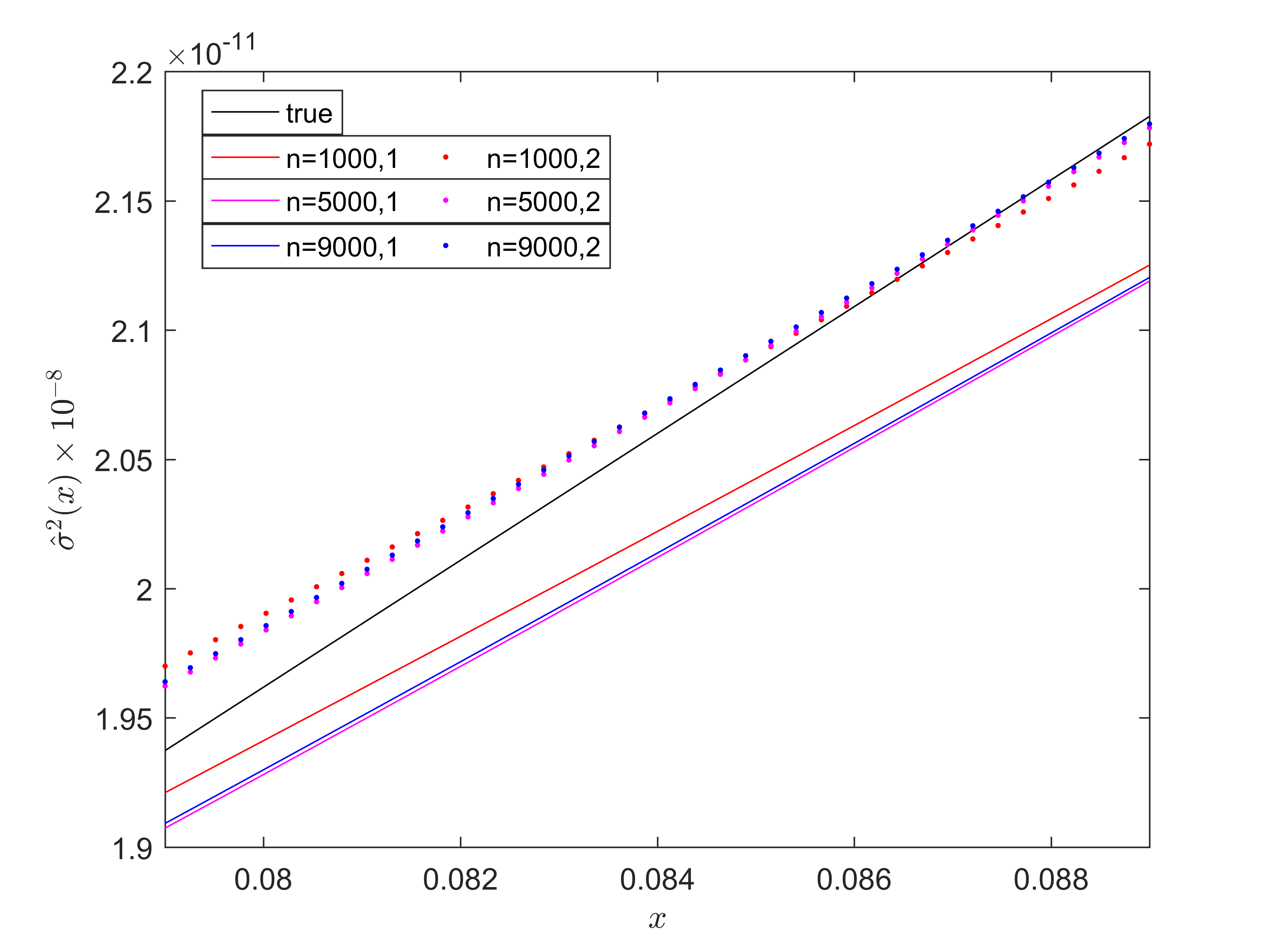}
    \caption*{$\Delta=0.008$, $h=0.03$}
  \end{minipage}
  }
  \subfigure{
    \begin{minipage}{0.32\linewidth}
    \includegraphics[width=2.3in,height=1.9in]{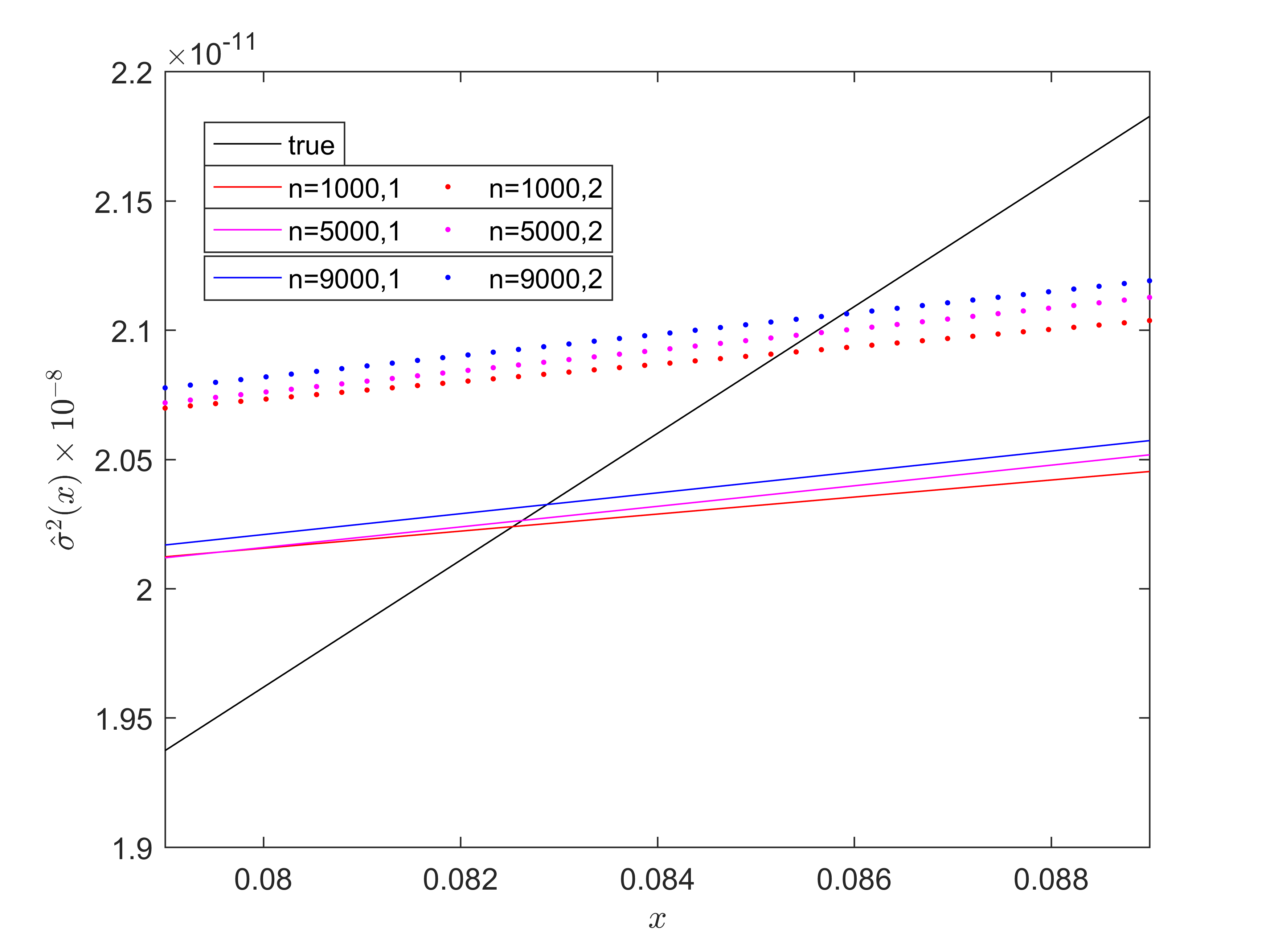}
    \caption*{$\Delta=0.008$, $h=0.12$}
  \end{minipage}
  }
  \subfigure{
    \begin{minipage}{0.32\linewidth}
    \includegraphics[width=2.3in,height=1.9in]{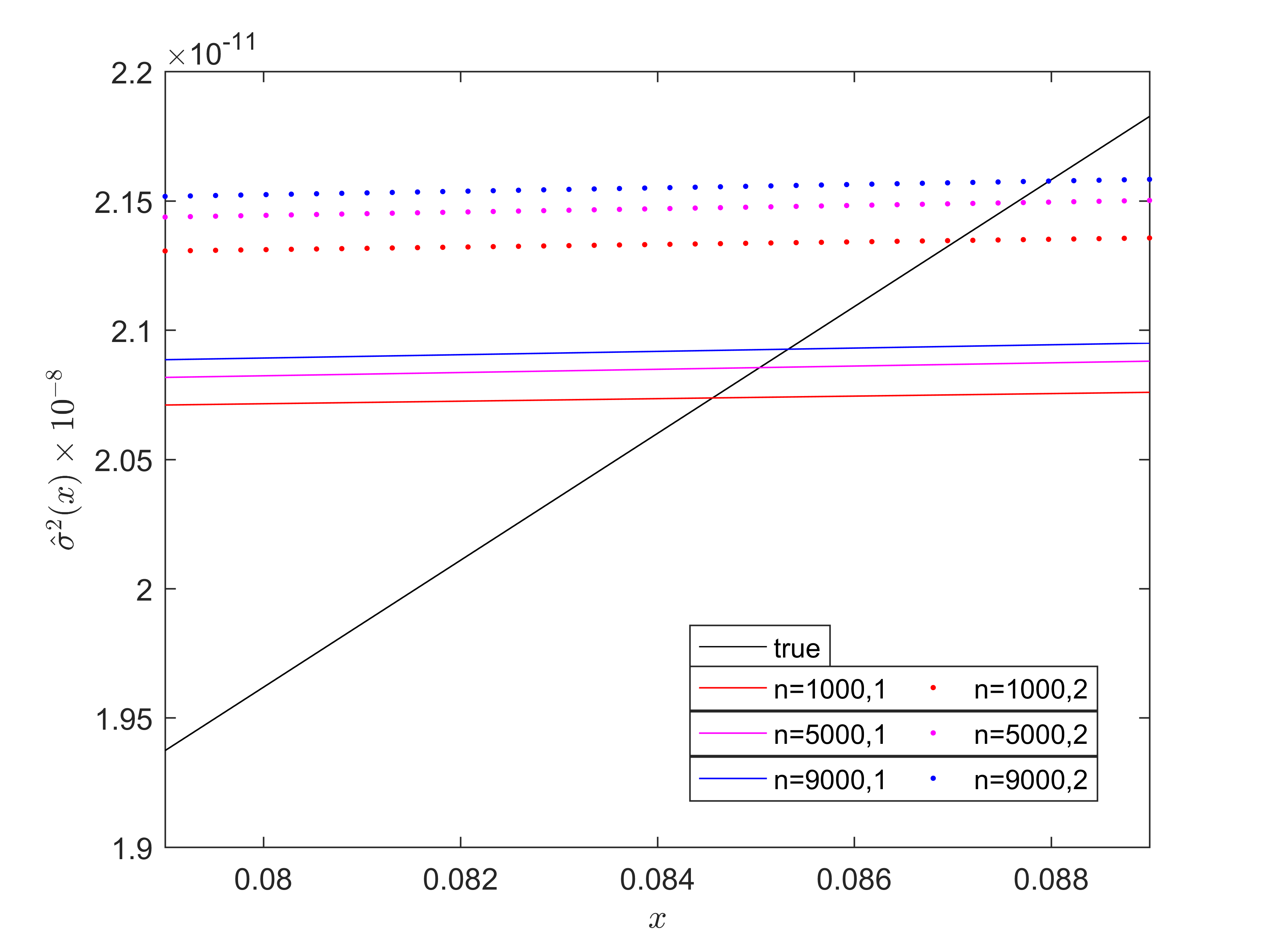}
    \caption*{$\Delta=0.008$, $h=0.3$}
  \end{minipage}
  }
  \caption{Mean Estimated Diffusion Coefficient Under Different $\Delta$ and $h$ for the CIR process}
  \label{resultestimatorCIR}
\end{figure}
%%%%%%%%%%%%%%%%%%%%%%%%%%%%%%%%%%%%%%% CIR-1
\begin{figure}[htbp]
  %\centering
  \begin{minipage}{0.32\linewidth}
  \centering
    \includegraphics[width=2.3in,height=1.9in]{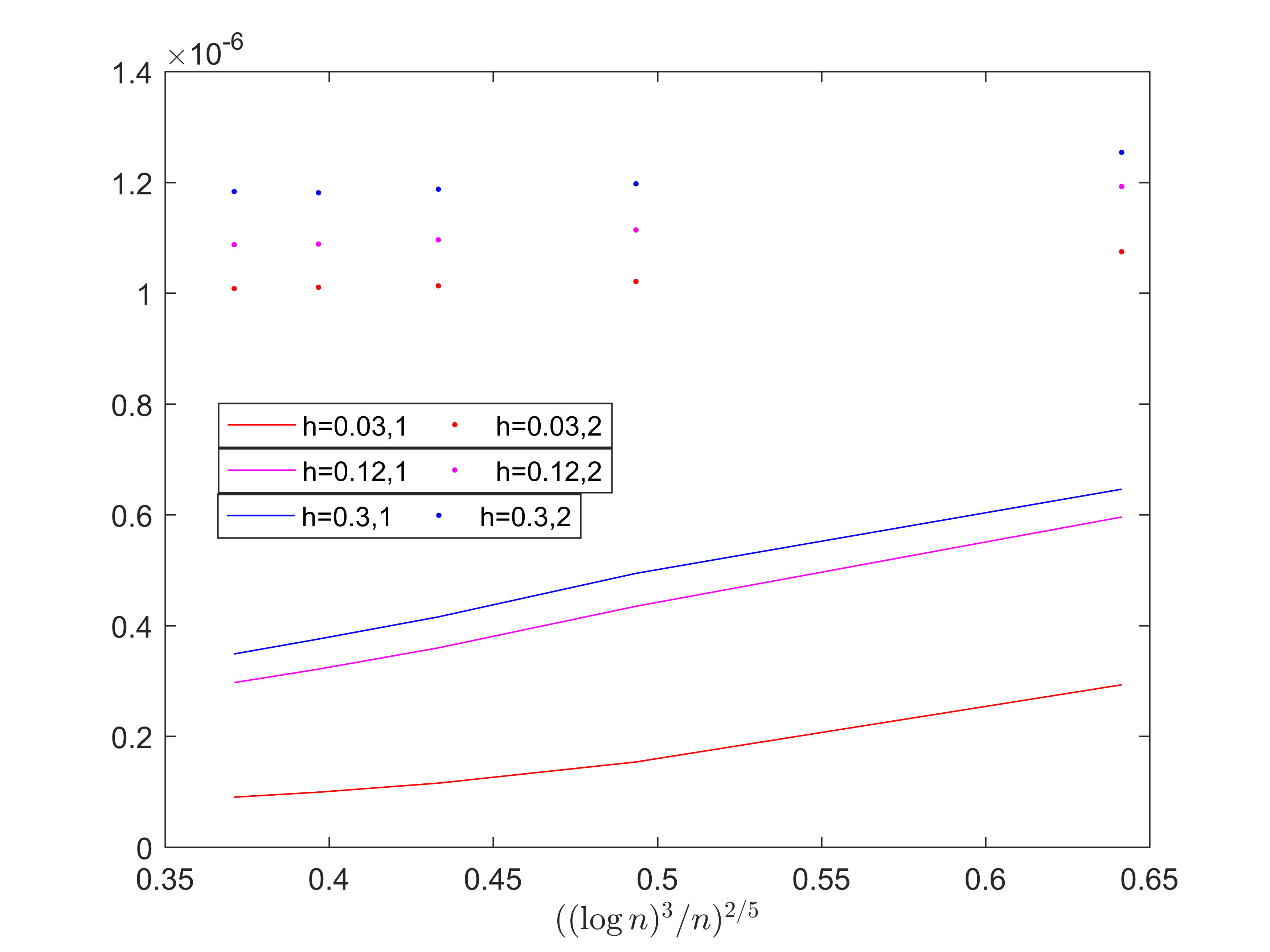}
    \caption*{$\Delta=0.002$}
  \end{minipage}
    \begin{minipage}{0.32\linewidth}
    \includegraphics[width=2.3in,height=1.9in]{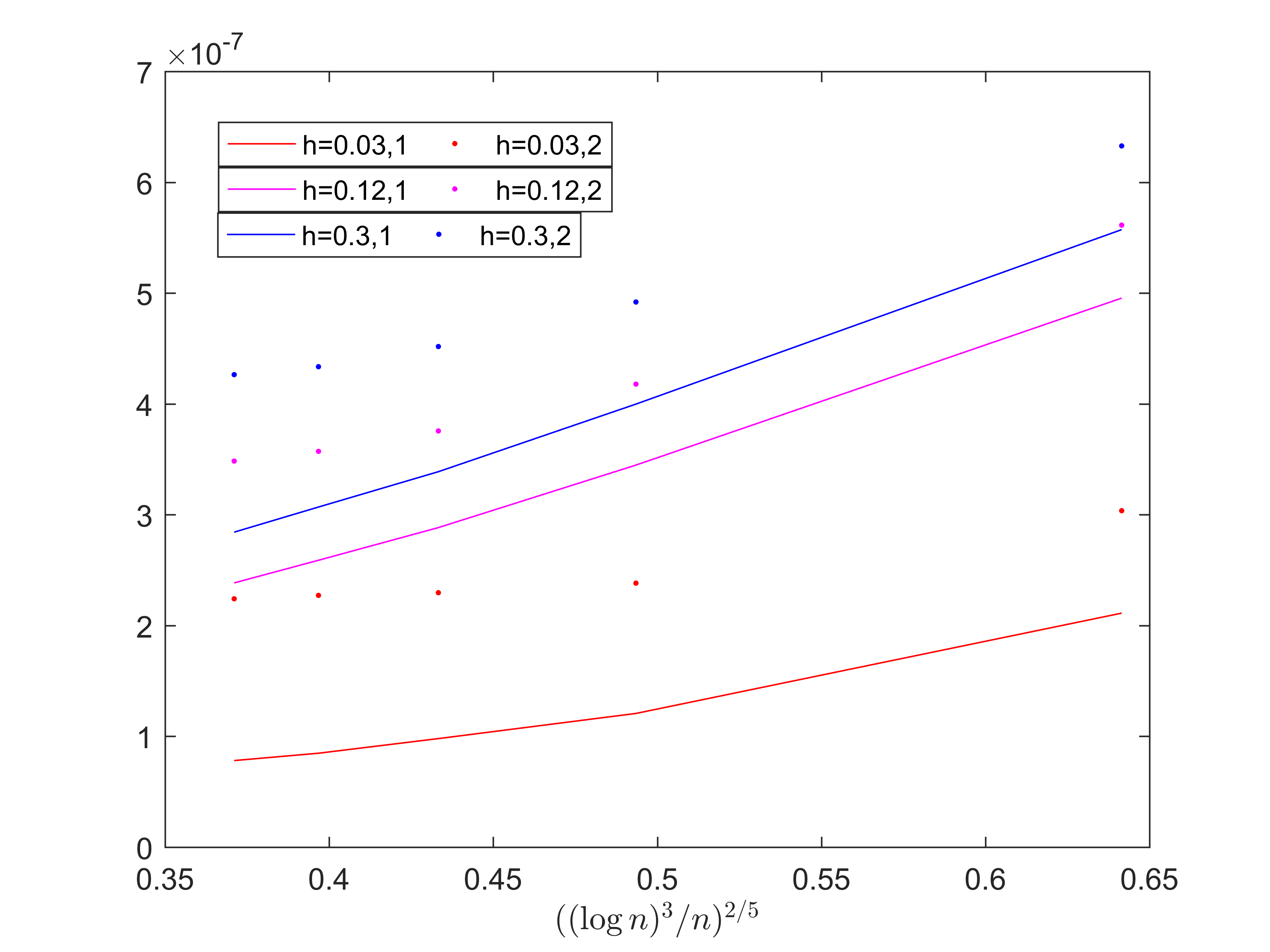}
    \caption*{$\Delta=0.004$}
  \end{minipage}
    \begin{minipage}{0.32\linewidth}
    \includegraphics[width=2.3in,height=1.9in]{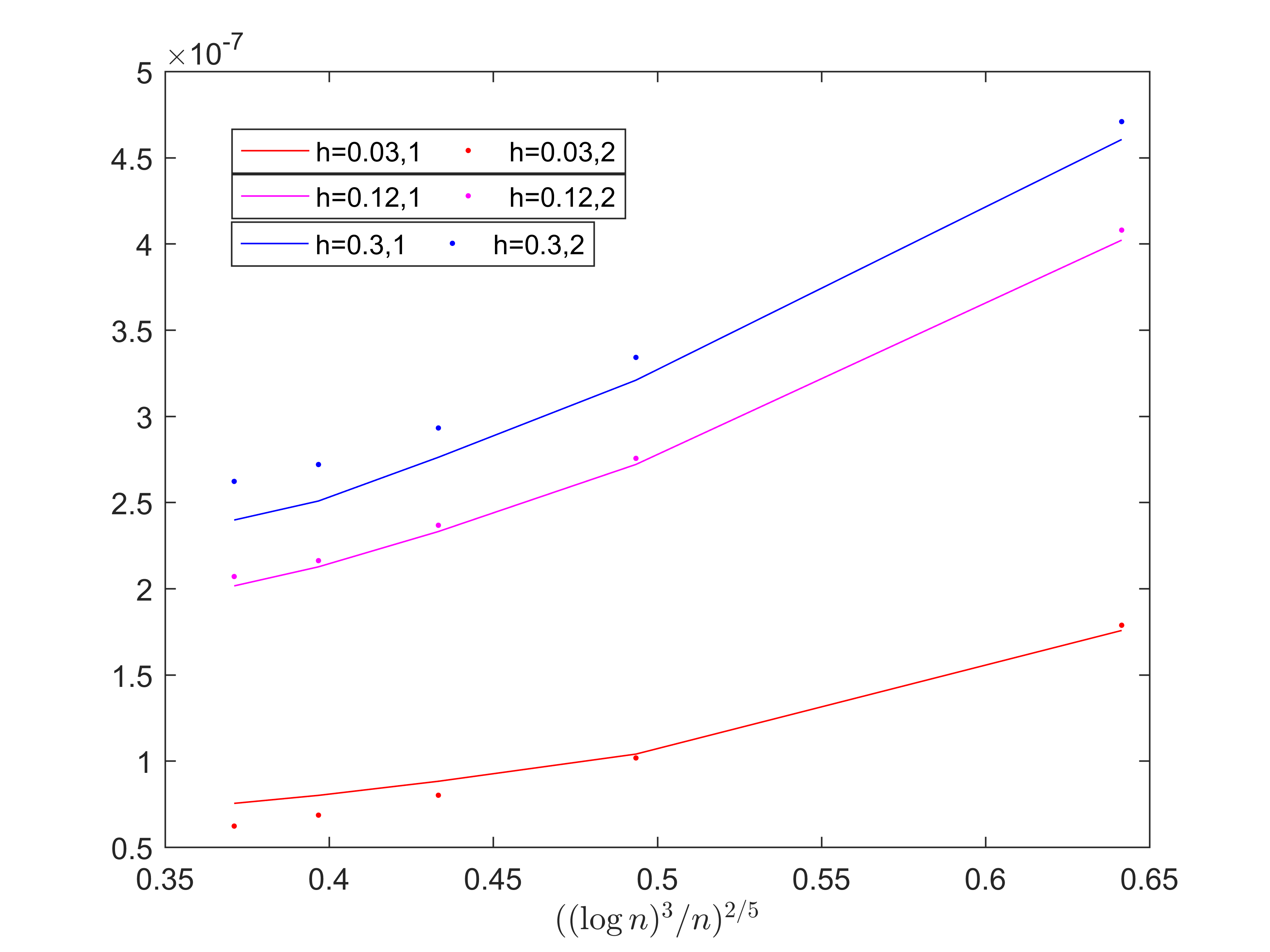}
    \caption*{$\Delta=0.008$}
  \end{minipage}
  \caption{Theoretical vs. Simulated Error: $((\log n)^{3}/n)^{2/5}$ vs. MAAE ($\times 10^{-3}$) for the CIR process}
  \label{comprisontheosimuCIR}
\end{figure}

%%%%%%%%%%%%%%%%%%%%%%%%%%%%%%%%%%%%% OU
\begin{table}[H]
\renewcommand{\arraystretch}{1.2}
  \centering
\begin{tabular}{cccccccc}
\hline
\multicolumn{1}{c}{$\Delta$}            & h                     & Error           & n=1000 & n=3000 & n=5000 & n=7000 & n=9000 \\ \hline
\multirow{6}{*}{$0.002$}                             &  \multirow{2}{*}{0.1218} & $\text{MAAE}^{1}\times 10^{2}$ & 7.0589  & 2.1952  & 1.1668  & 0.936  & 0.7929 \\
                                                      &                           & $\text{MAAE}^{2}\times 10^{2}$ & 13.7814 & 11.0112 & 10.0617 & 9.8511 & 9.7601 \\ \cline{2-8}
                                                      &   \multirow{2}{*}{0.3046}& $\text{MAAE}^{1}\times 10^{2}$ & 2.9353  & 0.8459  & 0.5703  & 0.484  & 0.4167 \\
                                                      &                           & $\text{MAAE}^{2}\times 10^{2}$ & 11.2319 & 9.6431  & 9.4678  & 9.4036 & 9.3779 \\ \cline{2-8}
                                                      &  \multirow{2}{*}{0.6091} & $\text{MAAE}^{1}\times 10^{2}$ & 1.0493  & 0.4915  & 0.3709  & 0.3233 & 0.2881 \\
                                                      &                           & $\text{MAAE}^{2}\times 10^{2}$ & 9.5579  & 9.3369  & 9.3164  & 9.2878 & 9.2895 \\  \hline
\multirow{6}{*}{$0.004$}                             & \multirow{2}{*}{0.1218}  & $\text{MAAE}^{1}\times 10^{2}$ & 3.9868  & 1.4056  & 1.0644  & 0.8837 & 0.7616 \\
                                                      &                           & $\text{MAAE}^{2}\times 10^{2}$ & 5.5051  & 3.0479  & 2.7937  & 2.689  & 2.6351 \\ \cline{2-8}
                                                      &  \multirow{2}{*}{0.3046} & $\text{MAAE}^{1}\times 10^{2}$ & 1.626   & 0.722   & 0.5533  & 0.4603 & 0.4045 \\
                                                      &                            & $\text{MAAE}^{2}\times 10^{2}$ & 3.0606  & 2.4568  & 2.4034  & 2.3774 & 2.3723 \\ \cline{2-8}
                                                      &  \multirow{2}{*}{0.6091}  & $\text{MAAE}^{1}\times 10^{2}$ & 0.86    & 0.4919  & 0.3769  & 0.3191 & 0.2796 \\
                                                      &                            & $\text{MAAE}^{2}\times 10^{2}$ & 2.3825  & 2.3091  & 2.3069  & 2.305  & 2.3057 \\   \hline
\multirow{6}{*}{$0.008$}                             &  \multirow{2}{*}{0.1218}  & $\text{MAAE}^{1}\times 10^{2}$ & 2.8146  & 1.3818  & 1.0232  & 0.8583 & 0.7508 \\
                                                      &                            & $\text{MAAE}^{2}\times 10^{2}$ & 3.232   & 1.5209  & 1.1899  & 1.0336 & 0.9353 \\ \cline{2-8}
                                                      &  \multirow{2}{*}{0.3046}  & $\text{MAAE}^{1}\times 10^{2}$ & 1.3349  & 0.7257  & 0.5499  & 0.4567 & 0.4008 \\
                                                      &                           & $\text{MAAE}^{2}\times 10^{2}$ & 1.5015  & 0.8824  & 0.7383  & 0.6662 & 0.6194 \\ \cline{2-8}
                                                      &  \multirow{2}{*}{0.6091} & $\text{MAAE}^{1}\times 10^{2}$ & 0.8583  & 0.4878  & 0.3826  & 0.3191 & 0.2834 \\
                                                      &                          & $\text{MAAE}^{2}\times 10^{2}$ & 0.9744  & 0.6651  & 0.594   & 0.5575 & 0.5426 \\  \hline
\multirow{6}{*}{$0.01$}                              &  \multirow{2}{*}{0.1218} & $\text{MAAE}^{1}\times 10^{2}$ & 2.5946  & 1.322   & 1.0081  & 0.8365 & 0.0578 \\
                                                      &                         & $\text{MAAE}^{2}\times 10^{2}$ & 2.7102  & 1.4415  & 1.0931  & 0.9285 & 0.8168 \\ \cline{2-8}
                                                      &  \multirow{2}{*}{0.3046}  & $\text{MAAE}^{1}\times 10^{2}$ & 1.2865  & 0.7179  & 0.5572  & 0.454  & 0.0434 \\
                                                      &                         & $\text{MAAE}^{2}\times 10^{2}$ & 1.3902  & 0.7829  & 0.6204  & 0.527  & 0.4759 \\ \cline{2-8}
                                                      &  \multirow{2}{*}{0.6091}  & $\text{MAAE}^{1}\times 10^{2}$ & 0.8304  & 0.4793  & 0.3856  & 0.3184 & 0.2747 \\
                                                      &                         & $\text{MAAE}^{2}\times 10^{2}$ & 0.9082  & 0.5486  & 0.4566  & 0.4054 & 0.3742 \\  \hline
\end{tabular}
  \caption{MAAE of the diffusion function estimators for OU process}
  \label{resulttableOU}
\end{table}
%%%%%%%%%%%%%%%%%%%%%%%%%%%%%%%%%%%%%%% OU-2
\begin{figure}[htbp]
  %\centering
  \subfigure{
  \begin{minipage}{0.32\linewidth}
  \centering
    \includegraphics[width=2.3in,height=1.9in]{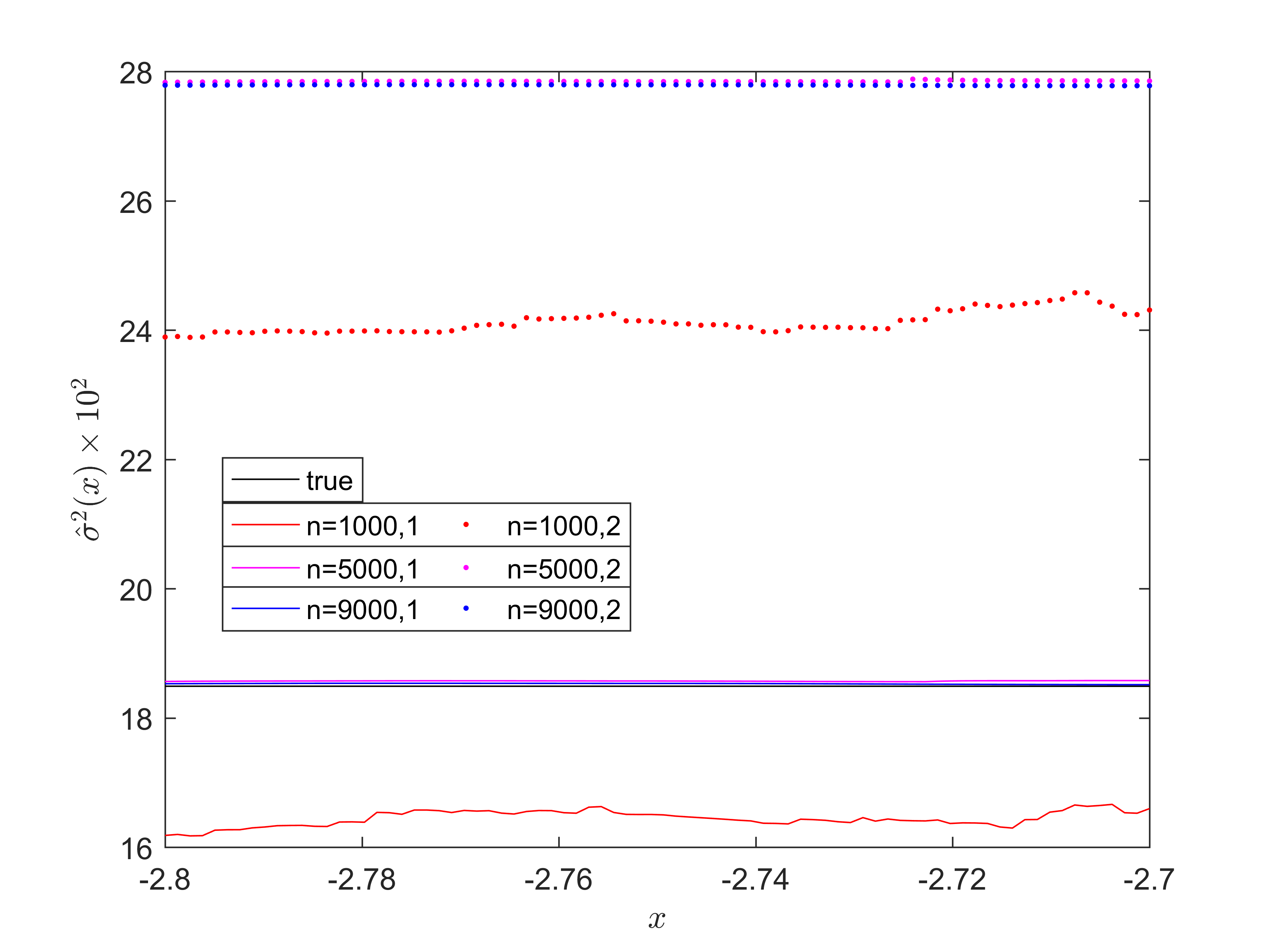}
    \caption*{$\Delta=0.002$, $h=0.1218$}
  \end{minipage}
  }
  \subfigure{
    \begin{minipage}{0.32\linewidth}
    \includegraphics[width=2.3in,height=1.9in]{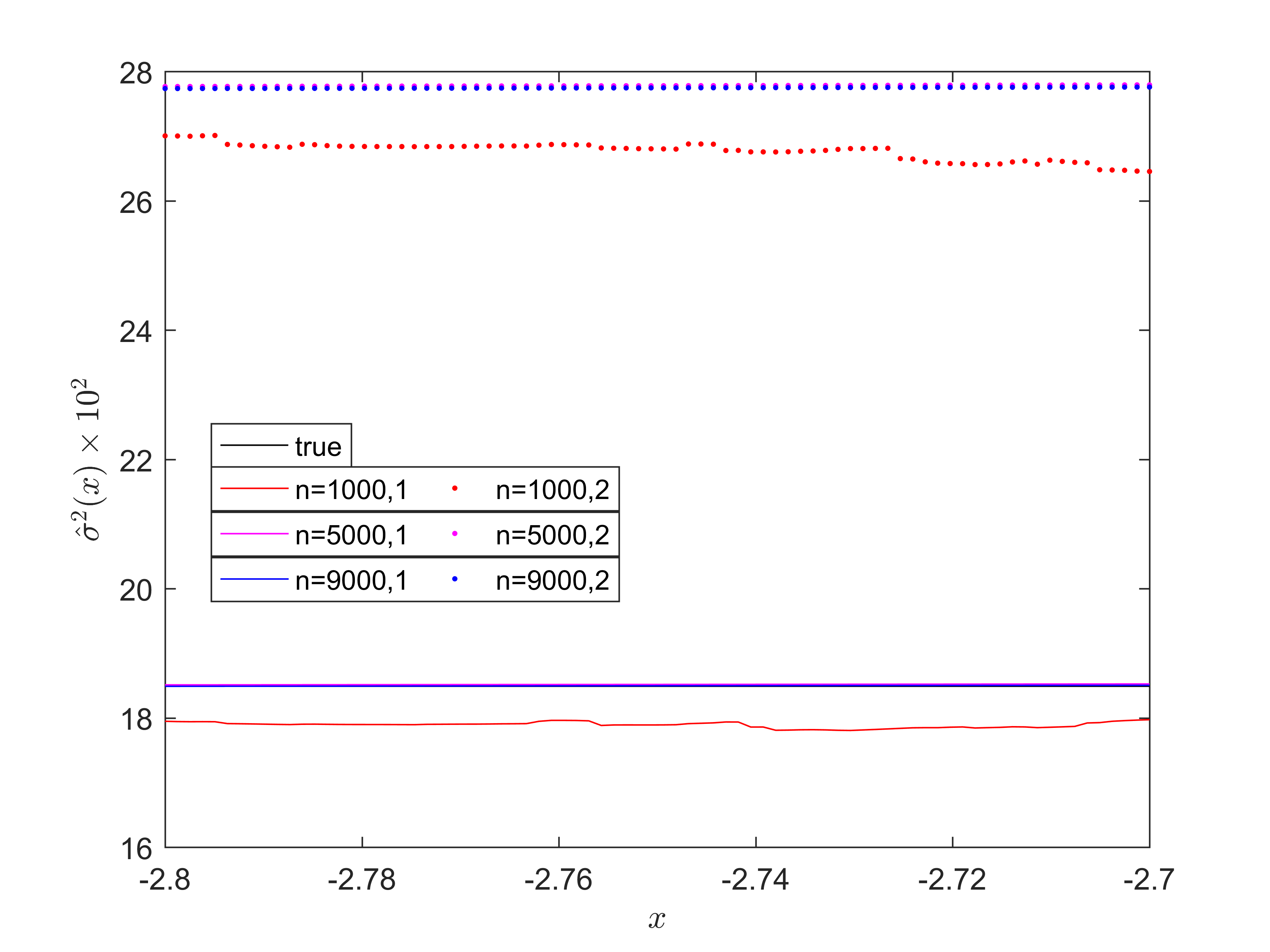}
    \caption*{$\Delta=0.002$, $h=0.3046$}
  \end{minipage}
  }
  \subfigure{
    \begin{minipage}{0.32\linewidth}
    \includegraphics[width=2.3in,height=1.9in]{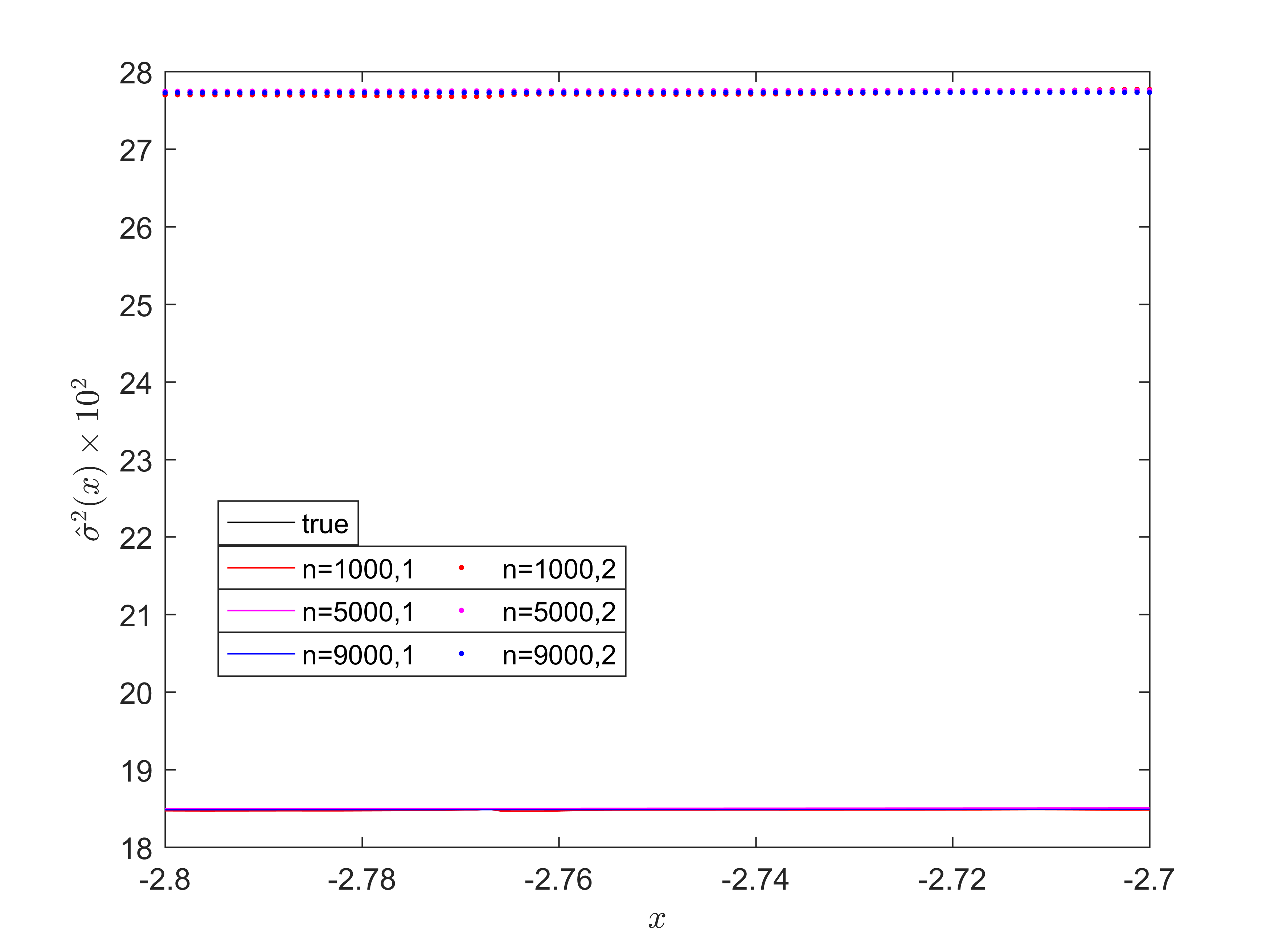}
    \caption*{$\Delta=0.002$, $h=0.6091$}
  \end{minipage}
  }

  \subfigure{
  \begin{minipage}{0.32\linewidth}
  \centering
    \includegraphics[width=2.3in,height=1.9in]{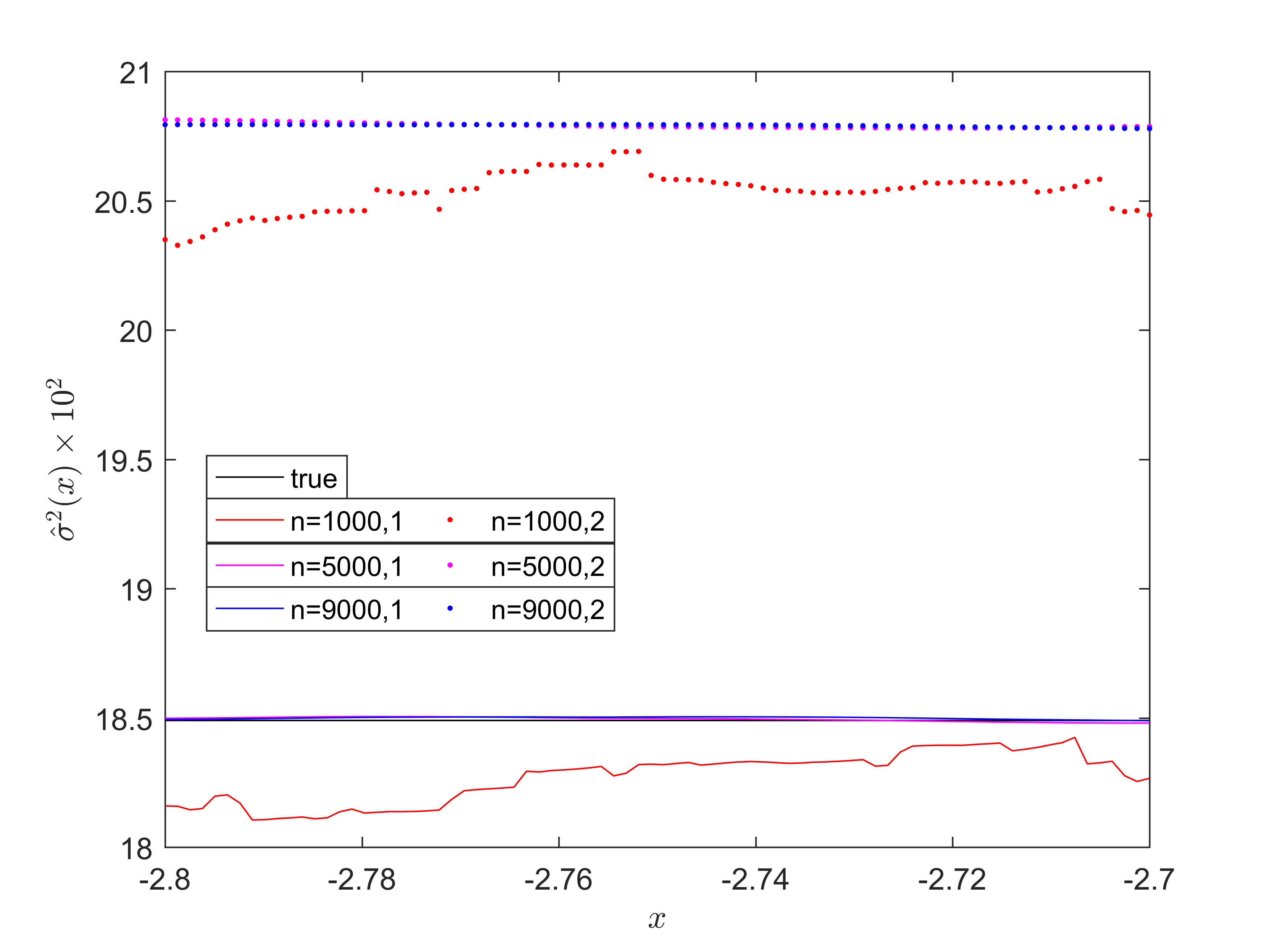}
    \caption*{$\Delta=0.004$, $h=0.1218$}
  \end{minipage}
  }
  \subfigure{
    \begin{minipage}{0.32\linewidth}
    \includegraphics[width=2.3in,height=1.9in]{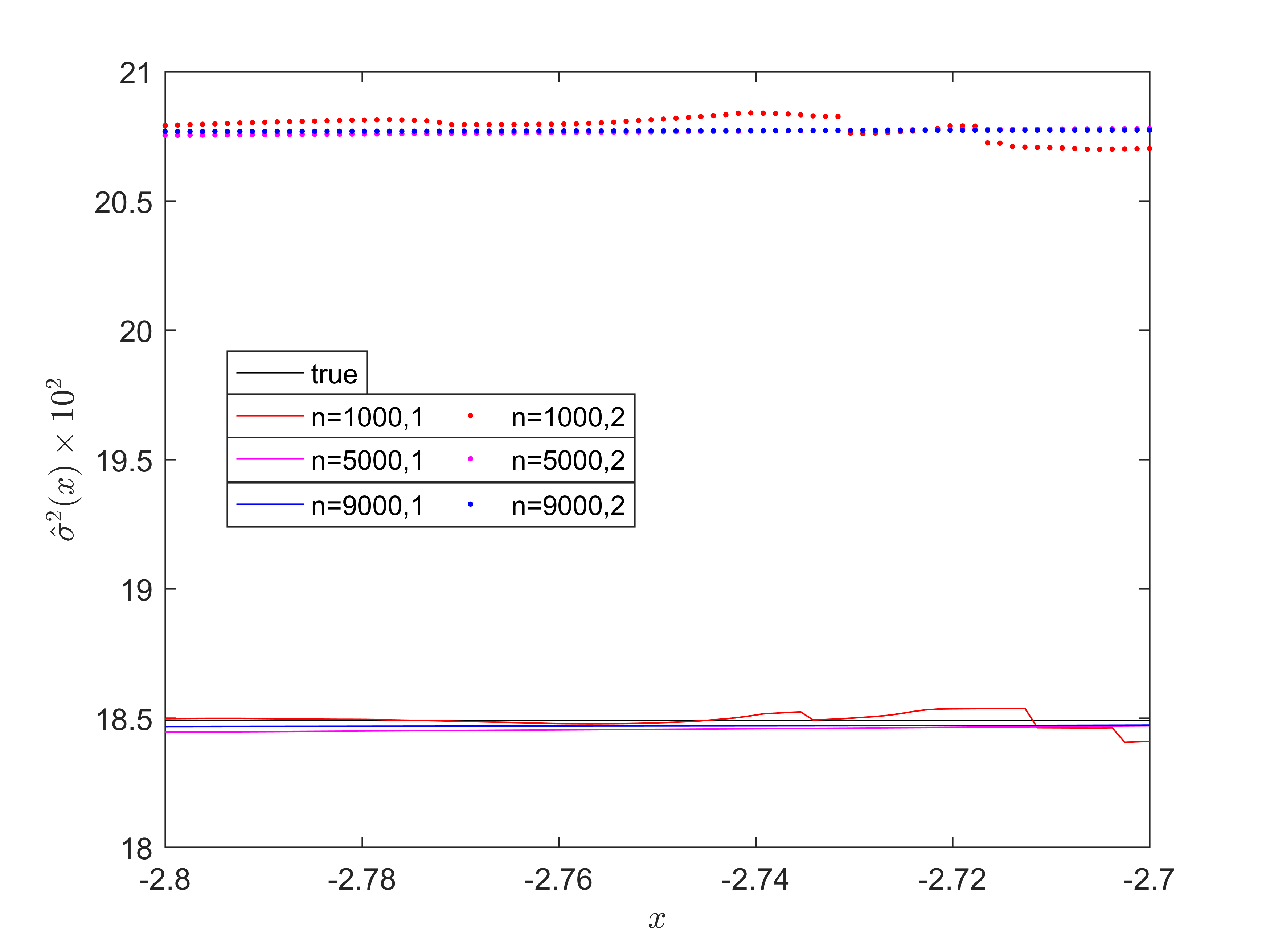}
    \caption*{$\Delta=0.004$, $h=0.3046$}
  \end{minipage}
  }
  \subfigure{
    \begin{minipage}{0.32\linewidth}
    \includegraphics[width=2.3in,height=1.9in]{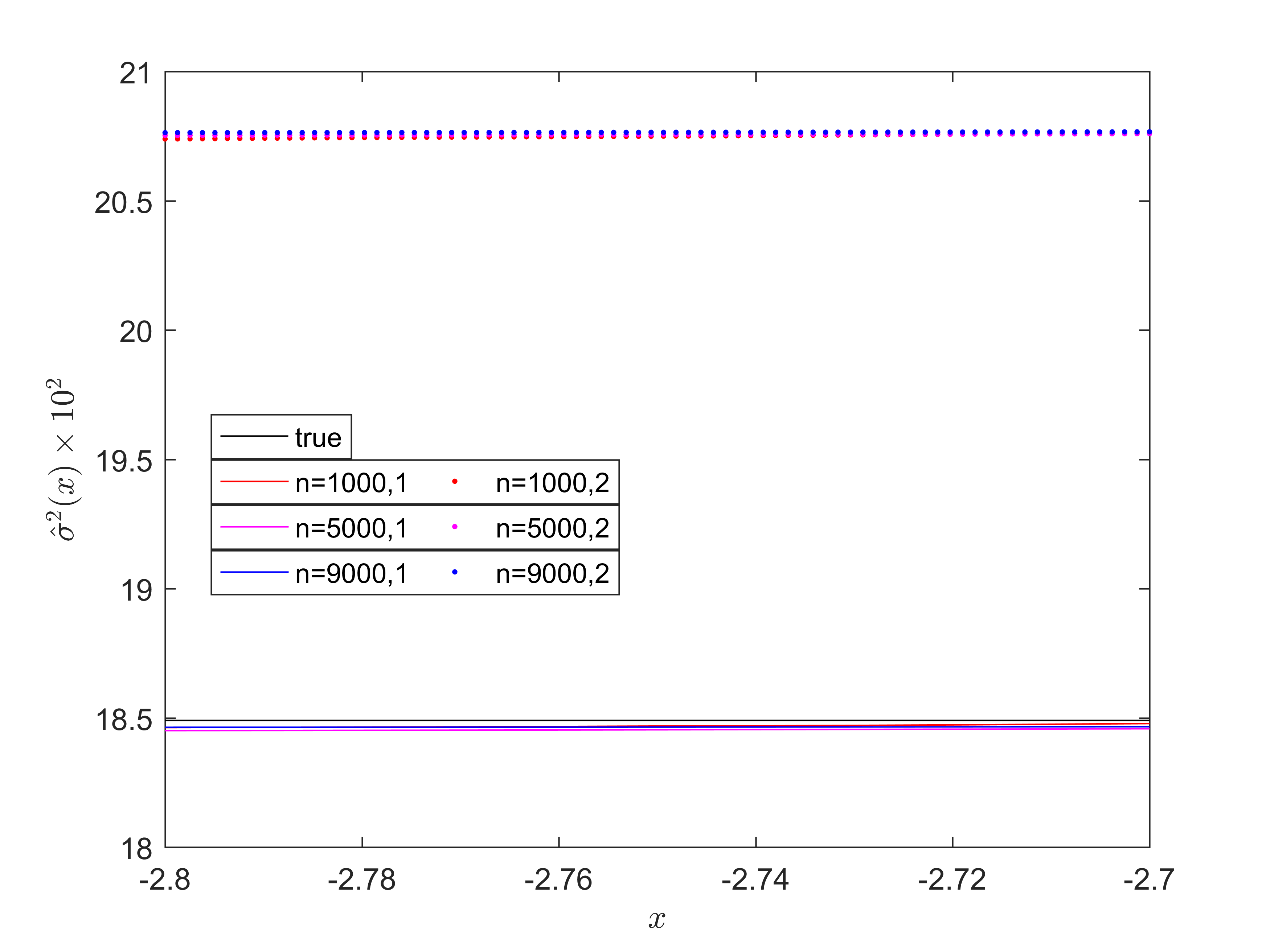}
    \caption*{$\Delta=0.004$, $h=0.6091$}
  \end{minipage}
  }

  \subfigure{
  \begin{minipage}{0.32\linewidth}
  \centering
    \includegraphics[width=2.3in,height=1.9in]{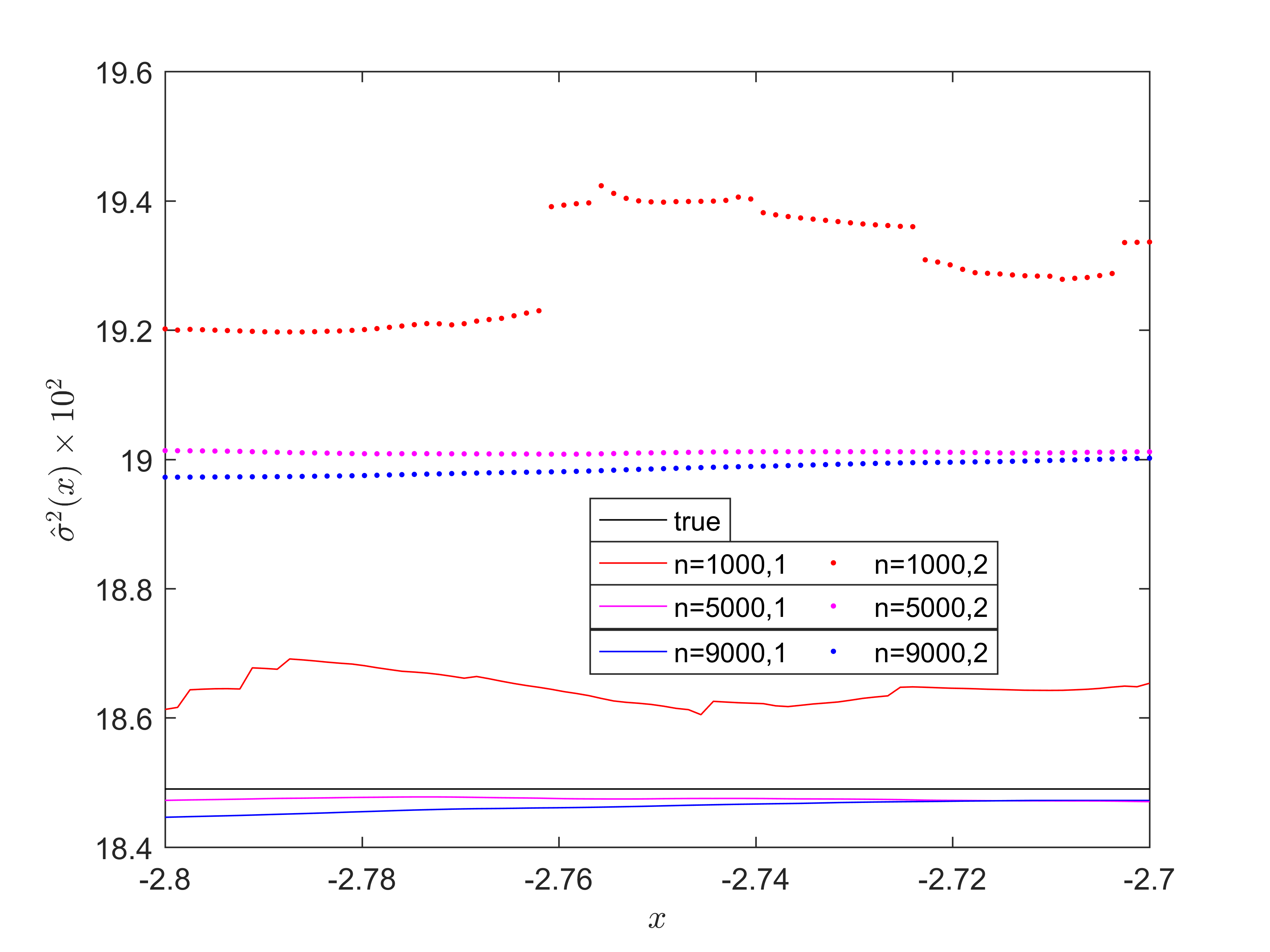}
    \caption*{$\Delta=0.008$, $h=0.1218$}
  \end{minipage}
  }
  \subfigure{
    \begin{minipage}{0.32\linewidth}
    \includegraphics[width=2.3in,height=1.9in]{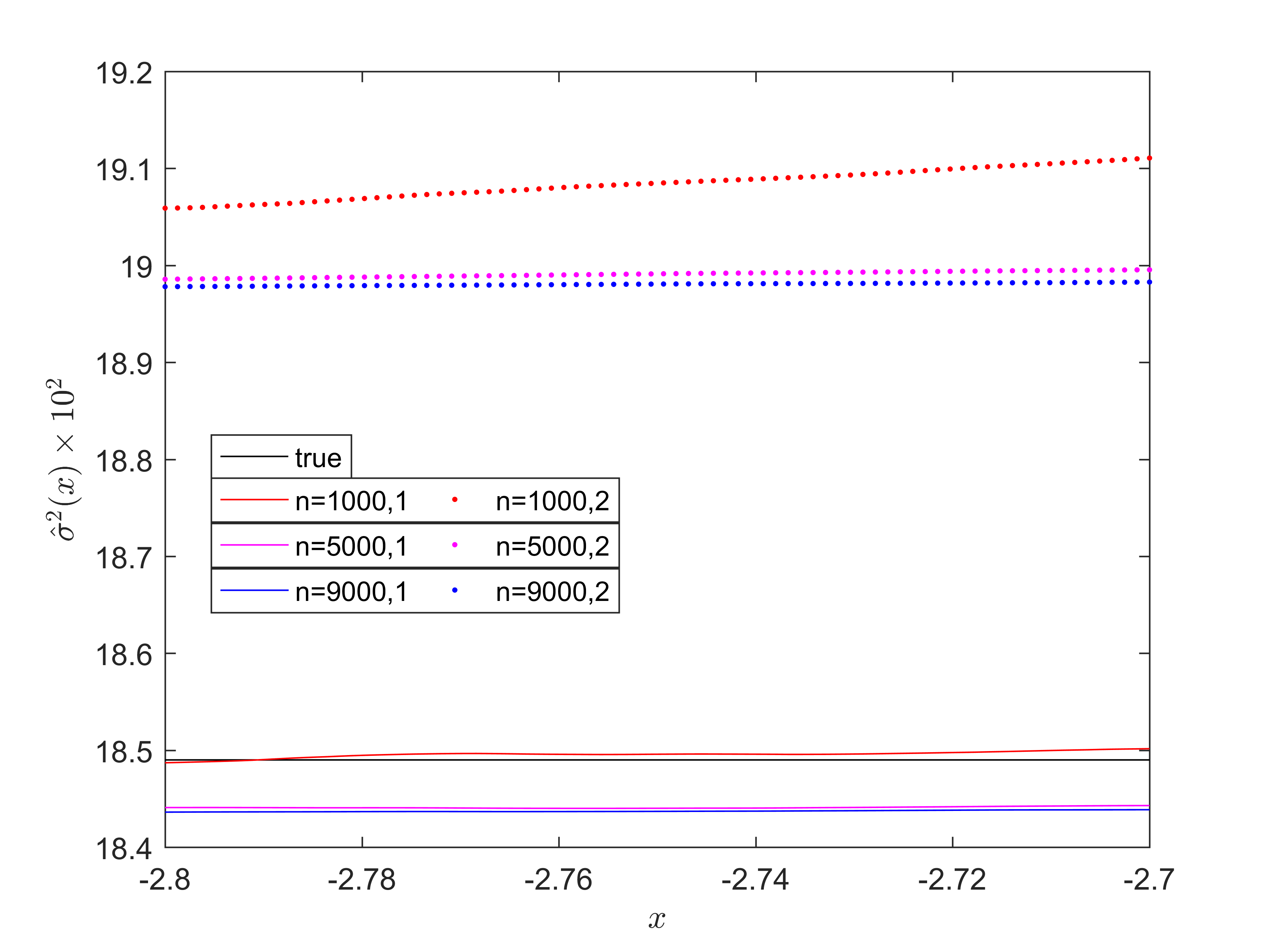}
    \caption*{$\Delta=0.008$, $h=0.3046$}
  \end{minipage}
  }
  \subfigure{
    \begin{minipage}{0.32\linewidth}
    \includegraphics[width=2.3in,height=1.9in]{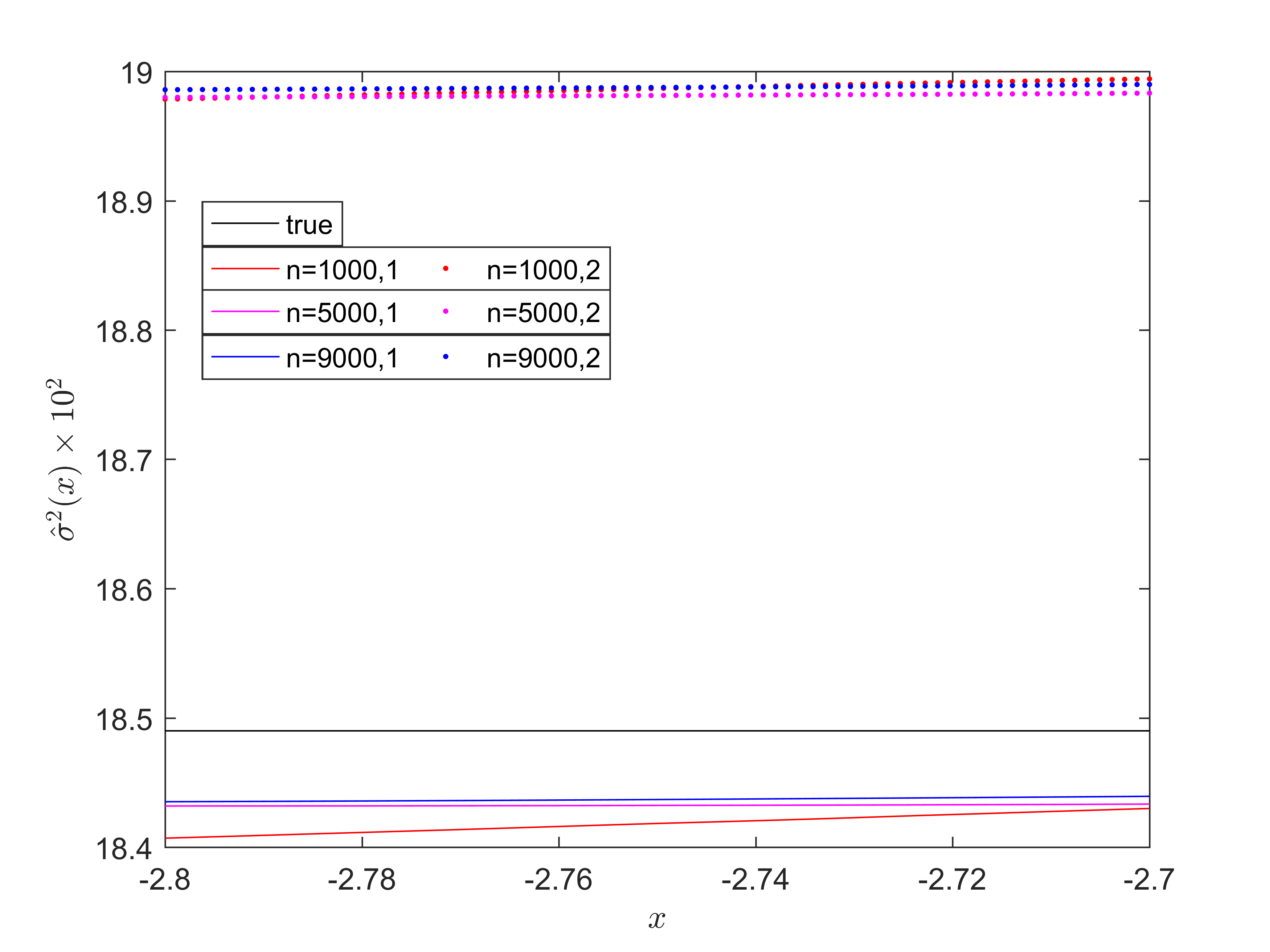}
    \caption*{$\Delta=0.008$, $h=0.6091$}
  \end{minipage}
  }
  \caption{Mean Estimated Diffusion Coefficient Under Different $\Delta$ and $h$ for the OU process}
  \label{resultestimatorOU}
\end{figure}
%%%%%%%%%%%%%%%%%%%%%%%%%%%%%%%%%%%%%%% OU-1
\begin{figure}[htbp]
  %\centering
  \begin{minipage}{0.32\linewidth}
  \centering
    \includegraphics[width=2.3in,height=1.9in]{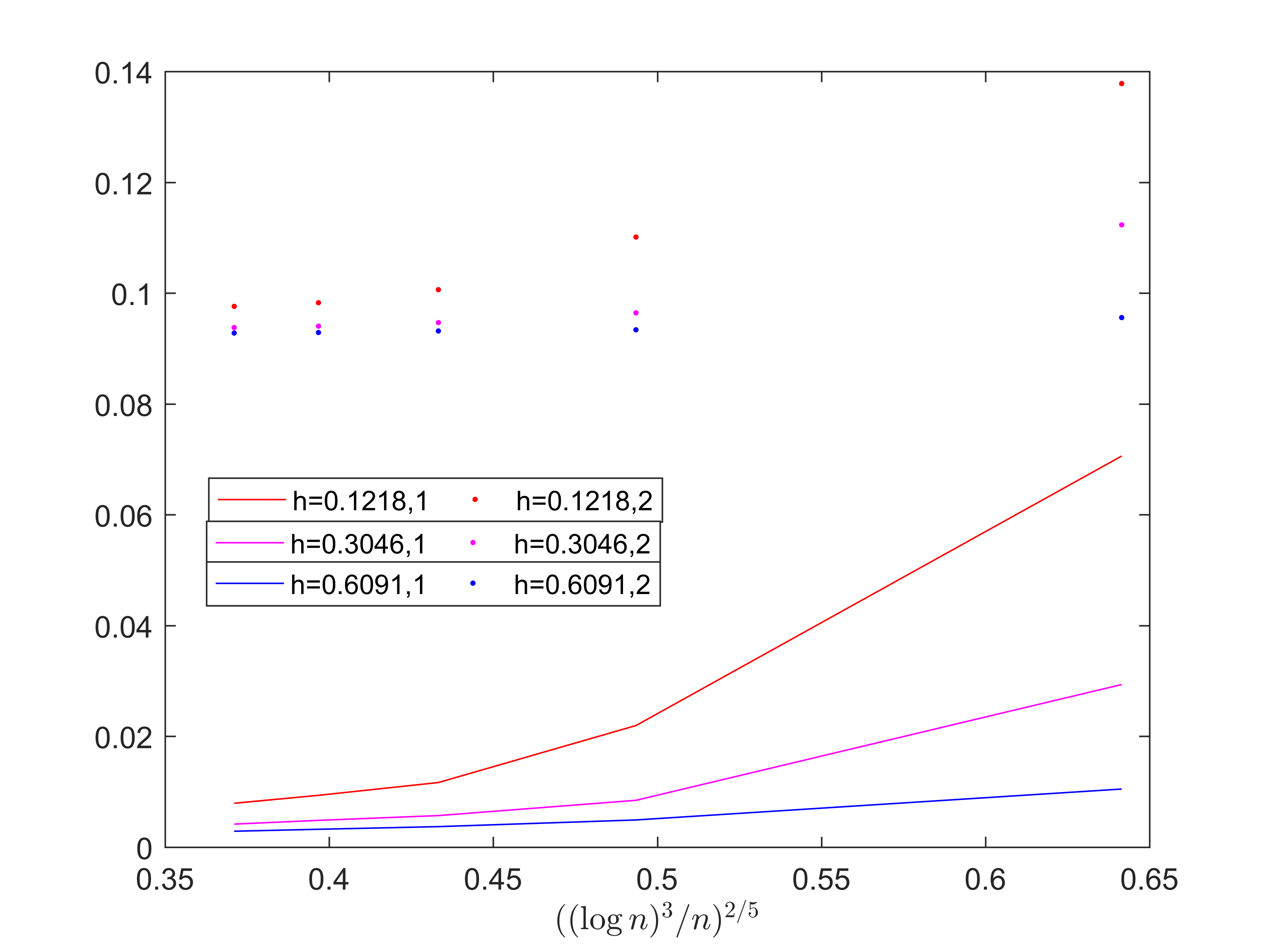}
    \caption*{$\Delta=0.002$}
  \end{minipage}
    \begin{minipage}{0.32\linewidth}
    \includegraphics[width=2.3in,height=1.9in]{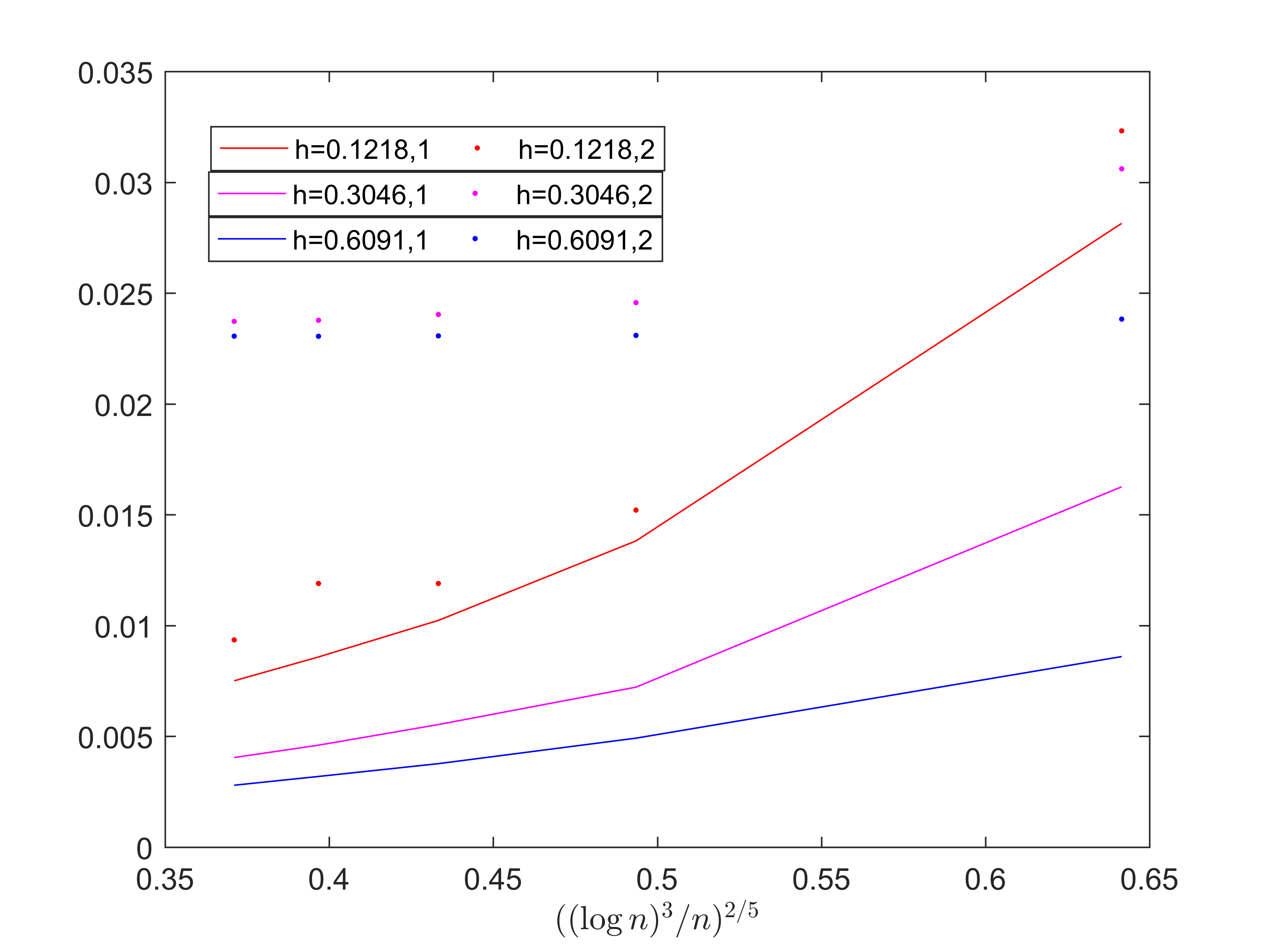}
    \caption*{$\Delta=0.004$}
  \end{minipage}
    \begin{minipage}{0.32\linewidth}
    \includegraphics[width=2.3in,height=1.9in]{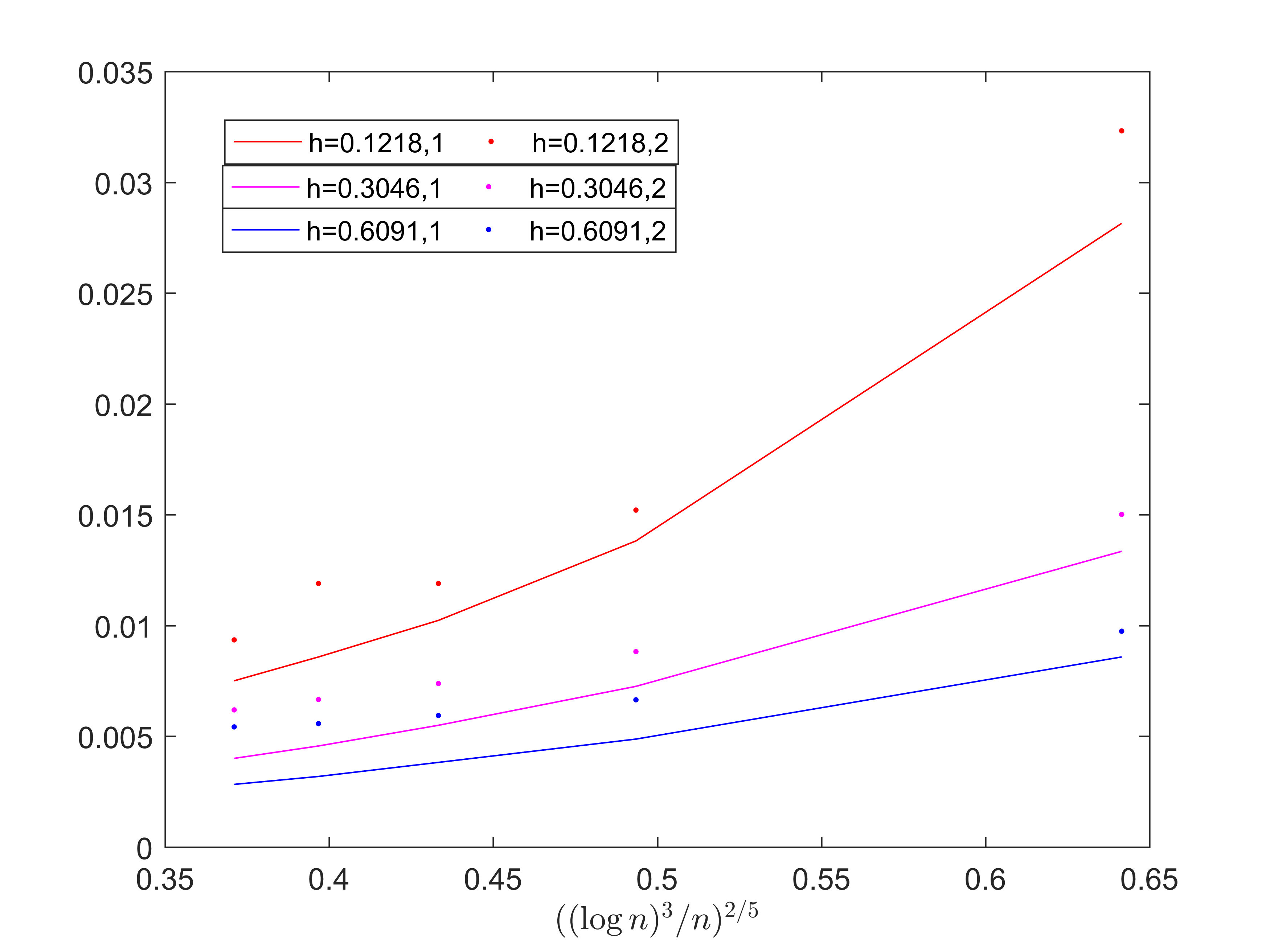}
    \caption*{$\Delta=0.008$}
  \end{minipage}
    \caption{Theoretical vs. Simulated Error: $((\log n)^{3}/n)^{2/5}$ vs. MAAE for the OU process}
  \label{comprisontheosimuOU}
\end{figure}

\section*{Funding}
This work was supported by the National Key R\&D Program of China (No. 2023YFA1008701) and the National Natural Science Foundation of China (Nos. 12431017 and 12326332).

\section*{Competing interests}
The authors declare that there is no conflict of interests regarding the publication of this article.
\section*{Availability of data}
Data sharing does not apply to this article as no datasets were generated or analyzed during the current study.


\begin{thebibliography}{99}
\bibitem{florens1993estimating}D. Florens-Zmirou, On estimating the diffusion coefficient from discrete observations, J. Appl. Probab. 30 (4) (1993) 790–804, https://doi.org/10.2307/3214513.
\bibitem{bibby1995martingale}B. M. Bibby, M. S{\o}rensen, Martingale estimation functions for discretely observed diffusion processes, Bernoulli (1995) 17–39, https://doi.org/10.2307/3318679.
    \bibitem{jiang1997nonparametric}G. J. Jiang, J. L. Knight, A nonparametric approach to the estimation of diffusion processes, with an application to a short-term interest rate model, Econom. Theory 13 (5) (1997) 615–645, https://doi.org/10.1017/S0266466600006101.
    \bibitem{bandi2003fully}F. M. Bandi, P. C. B. Phillips, Fully nonparametric estimation of scalar diffusion models, Econometrica 71 (1) (2003) 241–283, https://doi.org/10.1111/1468-0262.00395.
    \bibitem{fan2003}J. Fan, C. Zhang, A reexamination of diffusion estimators with applications to financial model validation, J. Am. Stat. Assoc. 98 (461) (2003) 118–134, https://doi.org/10.1198/016214503388619157.
    \bibitem{tang2009parameter}C. Tang, S. Chen, Parameter estimation and bias correction for diffusion processes, J. Econom. 149 (1) (2009) 65–81, https://doi.org/10.1016/j.jeconom.2008.11.001.
    \bibitem{ait2016bandwidth}Y. A{\"\i}t-Sahalia,  J. Y. Park, Bandwidth selection and asymptotic properties of local nonparametric estimators in possibly nonstationary continuous-time models, J. Econom. 192 (1) (2016) 119–138, https://doi.org/10.1016/j.jeconom.2015.11.002.
    \bibitem{ditlevsen2002fast}P. D. Ditlevsen, S. Ditlevsen, K. K. Andersen, The fast climate fluctuations during the stadial and interstadial climate states, Ann. Glaciol. 35 (2002) 457–462, https://doi.org/10.3189/172756402781816870.
    \bibitem{nicolau2007nonparametric}J. Nicolau, Nonparametric estimation of second-order stochastic differential equations, Econom. Theory 23 (5) (2007) 880–898, https://doi.org/10.1017/S0266466607070375.
    \bibitem{barndorff2002econometric}O. E. Barndorff-Nielsen, N. Shephard, Econometric analysis of realized volatility and its use in estimating stochastic volatility models, J. R. Stat. Soc., B: Stat. Methodol. 64 (2) (2002) 253–280, https://doi.org/10.1111/1467-9868.00336.
    \bibitem{andersen2009realized} T. G. Andersen, T. Ter{\"a}svirta, Realized volatility, in: Handbook of financial time series, Springer, 2009, pp. 555–575.
    \bibitem{gloter2001parameter}A. Gloter, Parameter estimation for a discrete sampling of an intergrated ornstein-uhlenbeck process, Statistics 35 (3) (2001) 225–243, https://doi.org/10.1080/02331880108802733.
    \bibitem{ditlevsen2004inference}S. Ditlevsen, M. S{\o}rensen, Inference for observations of integrated diffusion processes, Scand. J. Stat. 31 (3) (2004) 417–429, https://doi.org/10.1111/j.1467-9469.2004.02 023.x.
    \bibitem{gloter2006parameter}A. Gloter, Parameter estimation for a discretely observed integrated diffusion process, Scand. J. Stat. 33 (1) (2006) 83–104, https://doi.org/10.1111/J.1467-9469.2006.00465.X.
    \bibitem{wang2011locallinear}H. Wang, Z. Lin, Local linear estimation of second-order diffusion models, Communications in Statistics-Theory and Methods 40 (3) (2011) 394–407, https://doi.org/10.1080/03610920903391345.
    \bibitem{song2013reweighted}Y. Song, Z. Lin, H. Wang, Re-weighted Nadaraya–Watson estimation of second-order jump-diffusion model, J. Stat. Plan. Inference 143 (4) (2013) 730–744, http://dx.doi.org/10.1016/j.jspi.2012.09.010.
    \bibitem{hong2005nonparametric}Y. Hong, H. Li, Nonparametric specification testing for continuous-time models with applications to term structure of interest rates, Rev. Financ. Stud. 18 (1) (2005) 37–84, https://doi.org/10.1093/rfs/hhh006.
    \bibitem{li2007testing}F. Li, Testing the parametric specification of the diffusion function in a diffusion process, Econom. Theory 23 (2) (2007) 221–250, https://doi.org/10.1017/s0266466607070107.
    \bibitem{ait2009nonparametric}Y. A{\"\i}t-Sahalia, J. Fan, H. Peng, Nonparametric transition-based tests for jump diffusions, J. Am. Stat. Assoc. 104 (487) (2009) 1102–1116, https://doi.org/10.1198/jasa.2009.tm08198.
    \bibitem{kristensen2011semi}D. Kristensen, Semi-nonparametric estimation and misspecification testing of diffusion models, J. Econom. 164 (2) (2011) 382–403, https://doi.org/10.1016/j.jeconom.2011.07.001.
    \bibitem{bu2023uniform}R. Bu, J. Kim, B. Wang, Uniform and $\mathbb{L}_p$ convergences for nonparametric continuous time regressions with semiparametric applications, J. Econom. 235 (2) (2023) 1934–1954, https://doi.org/10.1016/j.jeconom.2023.02.006.
    \bibitem{bierens1983uniform}H. J. Bierens, Uniform consistency of kernel estimators of a regression function under generalized conditions, J. Am. Stat. Assoc. 78 (383) (1983) 699–707, https://doi.org/10.2307/2288140.
    \bibitem{andrews1995nonparametric}D. W. K. Andrews, Nonparametric kernel estimation for semiparametric models, Econom. Theory 11 (3) (1995) 560–596, https://doi.org/10.1017/S0266466600009427.
    \bibitem{liebscher1996strong} E. Liebscher, Strong convergence of sums of $\alpha$-mixing random variables with applications to density estimation, Stoch. Proc. Appl. 65 (1) (1996) 69–80, https://doi.org/10.1016/S0304-4149(96)00096-8.
    \bibitem{hansen2008uniform}B. E. Hansen, Uniform convergence rates for kernel estimation with dependent data, Econom. Theory 24 (3) (2008) 726–748, https://doi.org/10.1017/S0266466608080304.
    \bibitem{kristensen2009uniform} D. Kristensen, Uniform convergence rates of kernel estimators with heterogeneous dependent data, Econom. Theory 25 (5) (2009) 1433–1445, http://dx.doi.org/10.2139/ssrn.1144782.
    \bibitem{li2012local} D. Li, Z. Lu, O. Linton, Local linear fitting under near epoch dependence: uniform consistency with convergence rates, Econom. Theory 28 (5) (2012) 935–958, https://doi.org/10.1017/S0266466612000011.
    \bibitem{chan2014uniform}N. Chan, Q. Wang, Uniform convergence for nonparametric estimators with nonstationary data, Econom. Theory 30 (5) (2014) 1110–1133, https://doi.org/10.1017/S026646661400005X.
    \bibitem{lidegui2021nonparametric} H. Wang, B. Peng, D. Li, C. Leng, Nonparametric estimation of large covariance matrices with conditional sparsity, J. Econom. 223 (1) (2021) 53–72, https://doi.org/10.1016/j.jeconom.2020.09.002.
    \bibitem{koo2012estimation}B. Koo, O. Linton, Estimation of semiparametric locally stationary diffusion models, J. Econom. 170 (1) (2012) 210–233,  https://doi.org/10.1016/j.jeconom.2012.05.003.
    \bibitem{kanaya2017uniform} S. Kanaya, Uniform convergence rates of kernel-based nonparametric estimators for continuous time diffusion processes: A damping function approach, Econom. Theory 33 (4) (2017) 874–914, https://doi.org/10.1017/S0266466616000219.
    \bibitem{Doukhan1995Mixing}P. Doukhan, Mixing: Properties and examples, Lecture Notes in Statistics 85, https://doi.org/10.1007/978-1-4612-2642-0 (1995).
    \bibitem{veretennikov1997polynomial} A. Y. Veretennikov, On polynomial mixing bounds for stochastic differential equations, Stoch. Proc. Appl. 70 (1) (1997) 115–127, https://doi.org/10.1016/S0304-4149(97)00056-2.
    \bibitem{kanaya2016estimation}S. Kanaya, D. Kristensen, Estimation of stochastic volatility models by nonparametric filtering, Econom. Theory 32 (4) (2016) 861–916, https://doi.org/10.1017/S0266466615000079.
    \bibitem{Bosq1998Nonparametric} D. Bosq, Nonparametric statistics for stochastic processes, 2nd ed, Springer, New York, 1998.
    \bibitem{chapman2000short}D. A. Chapman, N. D. Pearson, Is the short rate drift actually nonlinear?, J. Finance 55 (1) (2000) 355–388, https://doi.org/10.1111/0022-1082.00208.
    \bibitem{stanton1997nonparametric}R. Stanton, A nonparametric model of term structure dynamics and the market price of interest rate risk, J. Finance 52 (5) (1997) 1973–2002, https://doi.org/10.1111/j.1540-6261.1997.tb02748.x.
\end{thebibliography}
\end{document}